\documentclass[a4paper,12pt]{article} 
\title{Logarithmic double ramification cycles}
\usepackage{amsmath, amssymb, mathrsfs, amsthm, shorttoc, geometry, stmaryrd}
\usepackage{mathtools} 
\usepackage{hyperref}
\usepackage{xcolor}
\hypersetup{
colorlinks,
linkcolor={black},
citecolor={blue!50!black},
urlcolor={blue!80!black}
}
\usepackage{listings}

\usepackage{graphicx}
\makeatletter
\DeclareRobustCommand{\cev}[1]{%
  {\mathpalette\do@cev{#1}}%
}
\newcommand{\do@cev}[2]{%
  \vbox{\offinterlineskip
    \sbox\z@{$\m@th#1 x$}%
    \ialign{##\cr
      \hidewidth\reflectbox{$\m@th#1\vec{}\mkern4mu$}\hidewidth\cr
      \noalign{\kern-\ht\z@}
      $\m@th#1#2$\cr
    }%
  }%
}
\makeatother

\usepackage[all]{xy}

\usepackage{tabularx}
\usepackage{longtable}

\usepackage{enumitem}

\usepackage[capitalise]{cleveref}
\crefformat{equation}{(#2#1#3)}

\usepackage{pgf,tikz}
\usetikzlibrary{matrix, calc, arrows,decorations.markings,decorations.pathmorphing}

\usepackage{tikz-cd} 

\makeatletter
\newcommand*{\doublerightarrow}[2]{\mathrel{
  \settowidth{\@tempdima}{$\scriptstyle#1$}
  \settowidth{\@tempdimb}{$\scriptstyle#2$}
  \ifdim\@tempdimb>\@tempdima \@tempdima=\@tempdimb\fi
  \mathop{\vcenter{
    \offinterlineskip\ialign{\hbox to\dimexpr\@tempdima+1em{##}\cr
    \rightarrowfill\cr\noalign{\kern.5ex}
    \rightarrowfill\cr}}}\limits^{\!#1}_{\!#2}}}
\newcommand*{\triplerightarrow}[1]{\mathrel{
  \settowidth{\@tempdima}{$\scriptstyle#1$}
  \mathop{\vcenter{
    \offinterlineskip\ialign{\hbox to\dimexpr\@tempdima+1em{##}\cr
    \rightarrowfill\cr\noalign{\kern.5ex}
    \rightarrowfill\cr\noalign{\kern.5ex}
    \rightarrowfill\cr}}}\limits^{\!#1}}}
\makeatother

\newcommand{\on}[1]{\operatorname{#1}}
\newcommand{\bb}[1]{{\mathbb{#1}}}

\newcommand{\ca}[1]{{\mathcal{#1}}}
\newcommand{\bd}[1]{{\mathbf{#1}}}

\newcommand{\ul}[1]{{\underline{#1}}}

\def\log{\mathrm{log}}

\def\virt{\mathrm{vir}}

\def\DR{\mathsf{DR}}

\def\CP{{{\mathbb {CP}}}}

\def\cO{\mathcal{O}}
\def\oM{\overline{\mathcal{M}}}
\def\cM{{\mathcal{M}}}
\def\C{{\mathcal{C}}}

\def\C{\mathbb{C}}

\def\Aut{{\rm Aut}}
\def\F{{\mathsf{F}}}
\def\E{\mathrm{E}}

\def\V{\mathrm{V}}

\def\G{\mathsf{G}}

\def\Pic{\mathcal{P}ic \,}

\newcommand{\mdeg}{\ul{\mathrm{de}}\mathrm{g}}

\DeclareMathAlphabet{\mymathbb}{U}{BOONDOX-ds}{m}{n}

\newcommand{\Span}[1]{\left<#1\right>}

\newcommand{\hra}{\hookrightarrow}

\newcommand{\sPPoly}[1]{\mathsf{sPP}(#1)}
\newcommand{\PPoly}[1]{\mathsf{PP}(#1)}
\newcommand{\Flow}{\mathrm{Flow}}

\newcommand{\newmarkedtheorem}[1]{%
  \newenvironment{#1}
    {\pushQED{\hfill $\lozenge$}\csname inner@#1\endcsname}
    {\popQED\csname endinner@#1\endcsname}%
  \newtheorem{inner@#1}%
}

\theoremstyle{definition}
\newtheorem{situation}{Situation}
\newmarkedtheorem{definition}[situation]{Definition} 

\newmarkedtheorem{problem}[situation]{Problem}

\newmarkedtheorem{remark}[situation]{Remark}

\newmarkedtheorem{example}[situation]{Example}
\newtheorem*{convention}{Convention}

\theoremstyle{plain}

\newtheorem{proposition}[situation]{Proposition}
\newtheorem{lemma}[situation]{Lemma}
\newtheorem{theorem}[situation]{Theorem}
\newtheorem*{theorem*}{Theorem}

\newtheorem{introtheorem}{Theorem}

\usepackage{letltxmacro}
\LetLtxMacro{\phiorig}{\phi}
\renewcommand{\phi}{\varphi}

\newcommand{\edits}[1]{#1}
\author{D. Holmes, S. Molcho, R. Pandharipande, A. Pixton, J. Schmitt}
\date{\edits{June 2024}}

\newcounter{nootje}
\newcommand{\beq}{\begin{equation}}
\newcommand{\eeq}{\end{equation}}
\newcommand{\beqs}{\begin{equation*}}
\newcommand{\eeqs}{\end{equation*}}

\renewcommand{\k}{k}

\newcommand{\DRpp}{\mathfrak{P}}
\newcommand{\DRpl}{\mathfrak{L}}

\tikzset{
  symbol/.style={
    draw=none,
    every to/.append style={
      edge node={node [sloped, allow upside down, auto=false]{$#1$}}}
  }
}
\tikzset{
    labl/.style={anchor=south, rotate=-90, inner sep=.5mm}
}

\begin{document}
\maketitle
\begin{abstract} 
Let $A=(a_1,\ldots, a_n)$ be a vector
of integers which sum to $k(2g-2+n)$.
The double ramification cycle
$\mathsf{DR}_{g,A}\in \mathsf{CH}^g(\oM_{g,n})$ 
on the moduli space of curves
is the virtual class of an Abel-Jacobi locus of
pointed curves $(C,x_1,\ldots,x_n)$
satisfying
$$\cO_C\Big(\sum_{i=1}^n a_i x_i\Big) 
\, \simeq\, \big(\omega^{\mathsf{log}}_{C}\big)^k\, .$$
The Abel-Jacobi construction
requires log 
blow-ups of $\oM_{g,n}$ to resolve
the indeterminacies of the Abel-Jacobi map.
Holmes \cite{Holmes2017Extending-the-d} has shown that 
 $\mathsf{DR}_{g,A}$
admits a canonical lift  
$\mathsf{logDR}_{g,A} \in \mathsf{logCH}^g(\oM_{g,n})$
to the logarithmic Chow ring, which
is the limit of  the intersection theories
of all such blow-ups.

The main result of the paper is
an explicit formula for $\mathsf{logDR}_{g,A}$ which lifts
Pixton's formula for $\mathsf{DR}_{g,A}$. The central idea
is to study the universal 
Jacobian over the moduli space of curves
(following Caporaso \cite{Caporaso1994A-compactificat},
Kass-Pagani \cite{Kass2017The-stability-s}, and Abreu-Pacini \cite{AbreuPacini})  for certain stability conditions.  Using the criterion of
Holmes-Schwarz \cite{HolmesSchwarz}, the universal
double ramification theory
of Bae-Holmes-Pandharipande-Schmitt-Schwarz \cite{BHPSS} applied to the
universal line bundle determines the logarithmic double ramification cycle. The resulting formula, written in the
language of piecewise polynomials,
depends upon the stability condition (and admits a wall-crossing study). 
Several examples of logarithmic and higher double ramification cycles are computed.


\end{abstract}


\tableofcontents

\newcommand{\Mtildes}{ \widetilde{\ca M}^\Sigma}
\newcommand{\sch}[1]{\textcolor{blue}{#1}}

\newcommand{\Mbar}{\overline{\ca M}}
\newcommand{\MD}{\ca M^\blacklozenge}
\newcommand{\Md}{\ca M^\lozenge}
\newcommand{\DRL}{\operatorname{DRL}}
\newcommand{\DRC}{\operatorname{DRC}}
\newcommand{\isom}{\stackrel{\sim}{\longrightarrow}}
\newcommand{\Ann}[1]{\on{Ann}(#1)}
\newcommand{\fm}{\mathfrak m}
\newcommand{\Mdk}{\Mbar^{\m, 1/\k}}
\newcommand{\field}{K}
\newcommand{\Mdm}{\Mbar^\m}
\newcommand{\m}{{\bd m}}
\newcommand{\cat}[1]{\mathbf{#1}}
\newcommand{\targetmap}{\ell}
\newcommand{\sourcemap}{\ell'}
\newcommand{\rel}{\mathsf{rel}}
\newcommand{\pre}{\mathsf{pre}}

\newcommand{\M}{\mathsf{M}}
\newcommand{\ghost}{\overline{\mathsf{M}}}
\newcommand{\gp}{\mathsf{gp}}
\newcommand{\fib}{\mathsf{cat}}
\newcommand{\et}{\mathsf{\acute{e}t}}
\renewcommand{\sf}[1]{\mathsf{#1}}

\newcommand{\Sym}{\on{Sym}}

\newcommand{\Spec}{\on{Spec}}

\newcommand{\divCHop}{\on{divCH}}

\newcommand{\Picabs}{\mathfrak{Pic}}
\newcommand{\Picrel}{\mathfrak{Pic}^{\mathrm{rel}}}
\newcommand{\Chow}{\mathsf{CH}}
\newcommand{\CHop}{\Chow_{\mathsf{op}}}

\newcommand{\LogChow}{\mathsf{logCH}}
\newcommand{\divLogChow}{\mathsf{divLogCH}}
\newcommand{\LogDR}{\sf{logDR}}
\newcommand{\Pictdz}{\mathfrak{Jac}}
\newcommand{\Jac}{\Pictdz}
\newcommand{\rJac}{\mathsf{Jac}}
\newcommand{\rPic}{\mathsf{Pic}}
\newcommand{\LogPic}{\mathfrak{LogPic}}
\newcommand{\rLogPic}{\mathsf{LogPic}}
\newcommand{\LogJac}{\mathfrak{LogJac}}
\newcommand{\rLogJac}{\mathsf{LogJac}}
\newcommand{\TroPic}{\mathfrak{TroPic}}
\newcommand{\rTroPic}{\mathsf{TroPic}}
\newcommand{\TroJac}{\mathfrak{TroJac}}
\newcommand{\rTroJac}{\mathsf{TroJac}}
\renewcommand{\log}{\sf{log}}
\newcommand{\trop}{\sf{trop}}
\newcommand{\op}{\sf{op}}
\newcommand{\SSS}{{\widehat{S}}}
\newcommand{\CCC}{{\widehat{C}}}
\newcommand{\LLL}{{\widehat{\mathcal{L}}}}
\newcommand{\GL}{\on{GL}}
\newcommand{\J}{\sf{J}}

\section{Introduction}
\label{intro111}
\subsection{Double ramification cycles}
\label{Ssec:DRclassic}

Let  $\cM_{g,n}$ be the moduli space of 
nonsingular curves of genus $g$ with $n$ distinct marked points over $\mathbb{C}$.
Given a vector of integers $A = (a_1, \dots, a_n)$ satisfying  $$\sum_{i=1}^n a_i = 0\,, $$ we can define a substack of $\cM_{g,n}$ 
by
\begin{equation}\label{g445}
\left\{ (C, x_1, \dots, x_n) \in \cM_{g,n} \ \rule[-1.0em]{0.05em}{2.5em} \ \cO_C\Big(\sum_{i=1}^n a_i x_i\Big) \simeq \cO_C \right\}\, . 
\end{equation}
From the point of view of relative Gromov-Witten theory, 
the most natural compactification of the substack \eqref{g445} 
is the space $\oM^{\sim}_{g,A}$ of stable maps to
{\em rubber} \cite{Graber2005Relative-virtua, Li2001Stable-morphism, Li2002A-degeneration-}: stable maps to
 $\CP^1$ relative to $0$ and $\infty$ modulo the 
 $\C^*$-action on $\CP^1$.

 The rubber moduli space carries a natural virtual fundamental class $\left[\oM^{\sim}_{g,A}\right]^\virt$ of dimension $2g-3+n$. The pushforward 
via the canonical morphism
$$ \epsilon:\oM^{\sim}_{g,A} \rightarrow  \oM_{g,n}$$
is the {\em double ramification cycle} on the moduli space of stable curves, 
\begin{equation}\label{relggww}
\epsilon_*\left[\oM^{\sim}_{g,A}\right]^\virt\, =\, \mathsf{DR}_{g,A}\ \in \mathsf{CH}^g(\oM_{g,n})
\, .
\end{equation}
The double ramification cycle 
$\mathsf{DR}_{g,A}$
can also be defined via log stable maps (and was
motivated in part by 
Symplectic Field Theory \cite{EGH}). 

The classical approach to the locus \eqref{g445} in 
$\cM_{g,n}$
is via Abel-Jacobi theory
for the universal curve. However, the Abel-Jacobi map does not extend over  
the boundary 
$$\partial\oM_{g,n} =\oM_{g,n} \setminus \cM_{g,n}$$ of the moduli space of stable
curves.  Approaches by Marcus-Wise \cite{Marcus2017Logarithmic-com} and
Holmes \cite{Holmes2017Extending-the-d}, motivated by log geometry, provide a partial
resolution of the Abel-Jacobi map which is sufficient to define a double
ramification cycle. 

We fix for the remainder of the paper an integer $k$ and a vector of integers $A=(a_1,\ldots,a_n)$
satisfying $$\sum_{i=1}^n a_i =k(2g-2+ n)\, .$$
The Abel-Jacobi construction in fact
yields a more general {\em $k$-twisted double ramification cycle} associated to the vector $A$, 
\begin{equation}\label{qtt35}
\mathsf{DR}_{g,A}\ \in \mathsf{CH}^g(\oM_{g,n})
\, , 
\end{equation}
and related
to the substack
\begin{equation*}\label{g445k}
\left\{ (C, x_1, \dots, x_n) \in \cM_{g,n} \; \rule[-0.8em]{0.05em}{2em} \; \cO_C\Big(\sum_{i=1}^n a_i x_i\Big) \simeq (\omega_C^{\mathrm{log}})^{\otimes k} \right\}\, ,
\end{equation*}
where $\omega_C^{\mathrm{log}} = \omega_C\left(\sum_{i=1}^n x_i\right)$ is the \emph{log canonical line bundle} of $C$.
The cycle \eqref{qtt35} agrees with  definition \eqref{relggww} in the $k=0$ case.


Eliashberg posed the question of computing $\mathsf{DR}_{g,A}$ in 2001. 
A complete formula for  $\mathsf{DR}_{g,A}$ in the
tautological ring of $\oM_{g,n}$
was conjectured by
Pixton in 2014 and proven in \cite{JPPZ} for $k=0$  and in  \cite{BHPSS} for general $k$. Pixton's formula expresses $\mathsf{DR}_{g,A}$
directly as a sum over stable graphs $\Gamma$  indexing the
boundary strata of $\oM_{g,n}$. The contribution of each stable
graph $\Gamma$ is the constant term  of a polynomial naturally associated
to the combinatorics of $\Gamma$ and $A$.

We refer the
reader to \cite[Section 0]{BHPSS},
\cite[Section 0]{JPPZ}, 
and \cite[Section 5]{RPSLC} for more
leisurely introductions to the subject of double
ramification cycles.
For a sampling of the 
development and application of the theory in a  variety of
directions, see \cite{B, BHPSS,BLS24,BGR, BR21, BSSZ,
Cav, CGJZ, cj, CSS,FanWu,Holmes2017Jacobian-extens,HO21,Holmes2017Multiplicativit, JPPZ,Janda2018Double-ramifica,
Molcho2021-Hodgebundle,Molcho2021-Case-Study,
ObPix,Pix18,Schmitt2016Dimension-theor,   vIOP20}.

\subsection{Logarithmic double ramification cycles}
The definition of the double ramification cycle
by Holmes \cite{Holmes2017Extending-the-d} yields  cycle classes on iterated blow-ups of boundary strata of the moduli space $\oM_{g,n}$ which push forward to
$\mathsf{DR}_{g,A}$ on $\oM_{g,n}$. 
The \emph{logarithmic Chow ring} $\mathsf{logCH}^g(\oM_{g,n})$
 describes the intersection theory on all suitable blow-ups of $\oM_{g,n}$, and Holmes' construction naturally yields a {\em logarithmic double ramification cycle} 
\begin{equation*} 
\mathsf{logDR}_{g,A}\ \in \mathsf{logCH}^g(\oM_{g,n})
\, .
\end{equation*}
The definition of $\mathsf{logCH}^g(\oM_{g,n})$ is reviewed in Sections \ref{aaa111}--\ref{subsec: Logintersectiontheory} below.
The class $\mathsf{logDR}_{g,A}$,
a refinement of  $\mathsf{DR}_{g,A}$,
plays
a fundamental role in logarithmic Gromov-Witten theory \edits{(and
has recently been applied in \cite{CMR22} to the study of double Hurwitz numbers and in \cite{RanganathanKumaran} to the log Gromov-Witten
theory of toric varieties)}.

The starting point of our paper is the following question:\,{\em can Pixton's formula for $\mathsf{DR}_{g,A}$ be lifted to
a formula for $\mathsf{logDR}_{g,A}$ in the logarithmic cycle theory of
the moduli space of curves?} Our main result is 
a formula for such a lift
obtained by applying the universal theory of $\cite{BHPSS}$ to 
the universal 
Jacobian over the moduli space of curves with respect to 
certain stability conditions (using
the criterion of \cite{HolmesSchwarz}).

The difficulties which arise in
 computing 
$\mathsf{logDR}_{g,A}$
via the original
definition of Holmes are explained
in Section \ref{resAJm}. Our new approach 
using
stability conditions for line
bundles on nodal curves
is presented in
 Section \ref{mlbsc}.
The main results about the
logarithmic double ramification 
cycle
are given in
an abstract form (Theorem \ref{111ooo}) in Section \ref{mlbsc}
and in an explicit form (Theorem \ref{222ttt}) in Section \ref{ffldr}. 
Pixton's double ramification cycle relations are lifted to
$\mathsf{logCH}^g(\oM_{g,n})$ by Theorem \ref{333ttt} presented
in Section \ref{ppplog}.
 
 We introduce
 the main constructions, ideas,
 and results
 of the paper in Sections \ref{aaa111}-\ref{ffldr}. The reader should
 be able to obtain a full overview
 of our argument
 by reading these sections.
 The body of 
 the paper contains a detailed presentation with complete
 proofs.



\subsection{Log modifications and cone stacks}
\label{aaa111}
We begin by discussing \emph{log modifications} of $\oM_{g,n}$, 
$$\widetilde{\cM} \to \oM_{g,n}\, ,$$
which are birational 
morphisms
used in the definition of the logarithmic Chow ring.
Log modifications generalize the notion of an iterated blow-up of boundary strata. Some background in log geometry is presented in Sections
\ref{logGeom} and \ref{sec:Logcurves}.

The most convenient way to describe a log modification of $\oM_{g,n}$ is via the \emph{cone stack} 
$\Sigma_{\oM_{g,n}}$
associated to the pair $(\oM_{g,n},\partial\oM_{g,n})$. 
As in the case of toric geometry, where every normal toric variety $X$ has an associated fan $\Sigma_X$, the cone stack $\Sigma_{\oM_{g,n}}$ is essentially
a cone complex describing the combinatorics of the boundary stratification of $\oM_{g,n}$. However, as the name suggests, in contrast to the toric case, the presence of automorphisms on the stack $\oM_{g,n}$ forces us to work with a {cone stack} (in the sense of \cite{Cavalieri2020-Conestack}) instead of a usual cone complex.\footnote{Some authors use the term {\em stacky fan} for the cone stack to emphasize the
  analogy to toric geometry, see \cite{StackyFan1,StackyFan3,StackyFan2}.
}
Just as toric modifications $\widehat{X} \to X$ of a toric variety $X$ are in bijective correspondence with subdivisions $\widetilde{\Sigma} \to \Sigma_X$ of the associated fan, log modifications of $\oM_{g,n}$ are in bijective correspondence with subdivisions $\widetilde{\Sigma} \to \Sigma_{\oM_{g,n}}$ of the cone stack.

The boundary strata of $\oM_{g,n}$ are indexed by
\emph{stable graphs} $\Gamma$: decorated graphs (with possible loops and multi-edges) describing the topological type of the generic
stable curve $(C, x_1, \ldots, x_n)$ parameterized by the strata. 
For the precise definition of a stable graph, we refer the reader to \cite[Appendix A]{GraberPandharipande}. See Figure \ref{fig:stablegraph} for an illustration.

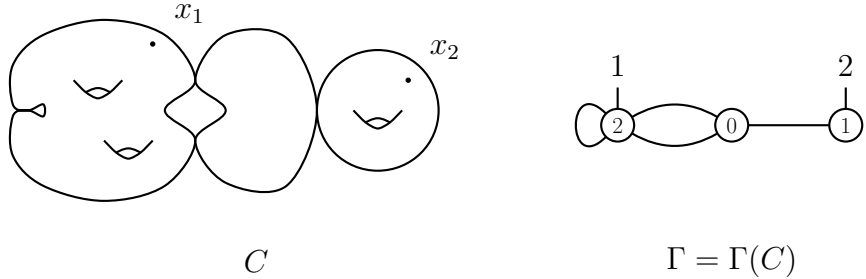
\begin{figure}[htb]
\[
\begin{tikzpicture}[baseline=1cm, scale=0.8]
\draw[thick] plot [smooth cycle, tension = 0.6, xshift=0 cm, yshift = 0cm] coordinates { (0,3) (1,2.75) (1.5,2) (1,1.5) (1.5,1) (1,0.25) (0,0) (-1,0.25) (-1.5,0.75) (-1.5,1.4) (-1.25,1.5) (-1,1.4) (-1,1.6) (-1.25,1.5) (-1.5,1.6) (-1.5,2.25) (-1,2.75) };
\draw[thick] plot [smooth cycle, tension = 0.6, xshift=0 cm, yshift = 0cm] coordinates {(3,2.7) (2,2.75) (1.5,2) (2,1.5) (1.5,1) (2,0.3) (3,0.25) (3.5,1.5) };
\draw[thick] plot [smooth cycle, tension = 1, xshift=0 cm, yshift = 0cm] coordinates { (3.5,1.5) (4.5,0.5) (5.5,1.5) (4.5,2.5) };

\draw[thick] plot [smooth, tension = 0.5, xshift=-1cm, yshift=3cm] coordinates {(1,-2) (1.2,-2.2) (1.4,-2.3) (1.6,-2.2) (1.8,-2)};
\draw[thick] plot [smooth, tension = 0.8, , xshift=-1cm, yshift=3cm] coordinates { (1.2,-2.2) (1.4,-2.13) (1.6,-2.2)};
\draw[thick] plot [smooth, tension = 0.5, xshift=-1.5cm, yshift=4cm] coordinates {(1,-2) (1.2,-2.2) (1.4,-2.3) (1.6,-2.2) (1.8,-2)};
\draw[thick] plot [smooth, tension = 0.8, , xshift=-1.5cm, yshift=4cm] coordinates { (1.2,-2.2) (1.4,-2.13) (1.6,-2.2)};
\draw[thick] plot [smooth, tension = 0.5, xshift=3.1cm, yshift=3.5cm] coordinates {(1,-2) (1.2,-2.2) (1.4,-2.3) (1.6,-2.2) (1.8,-2)};
\draw[thick] plot [smooth, tension = 0.8, , xshift=3.1cm, yshift=3.5cm] coordinates { (1.2,-2.2) (1.4,-2.13) (1.6,-2.2)};

\filldraw[black] (0.8,2.6) circle (1pt) node[above right = 4]{$x_1$};
\filldraw[black] (5,2) circle (1pt) node[above right = 4]{$x_2$};
\draw (2.5,-1) node {$C$} ;
\end{tikzpicture}
\quad \quad \quad
\begin{tikzpicture}[baseline=0pt, vertex/.style={circle,draw,font=\Large,scale=0.5, thick}]
\node[vertex] (A) at (-1.5,0) {2};
\node[vertex] (B) at (0,0) {0};
\node[vertex] (C) at (1.5,0) {1};
\draw[thick] (A) to[bend left] (B);
\draw[thick] (A) to[bend right] (B);
\draw[thick] (A) to[out=135, in=225, looseness = 6] (A);
\draw[thick] (B) to (C);
\draw[thick] (A) to (-1.5, 0.5) node[above] {$1$};
\draw[thick] (C) to (1.5, 0.5) node[above] {$2$};
\draw (0,-1.8) node {$\Gamma = \Gamma(C)$};
\end{tikzpicture}
\]    
    \caption{A curve $(C,x_1, x_2) \in \oM_{5,2}$ and the associated stable graph $\Gamma$}
    \label{fig:stablegraph}
\end{figure}

The cone stack  $\Sigma_{\oM_{g,n}}$ is
constructed from cones\footnote{We follow the definition of \cite[\S 2.1]{ACP}: cones are always strongly convex, rational polyhedral, and generate the vector space spanned by their ambient lattice. See Section \ref{conedefs} for a careful discussion.}
associated to stable graphs.
For \edits{a} stable graph $\Gamma$ with edge set $E(\Gamma)$, the associated cone $\sigma_\Gamma$ is defined by
\begin{equation*} 
    \sigma_\Gamma = (\mathbb{R}_{\geq 0})^{E(\Gamma)} = \{\ell: E(\Gamma) \to \mathbb{R}_{\geq 0}\}\,.
\end{equation*}
An element of $\sigma_\Gamma$ is an assignment of nonnegative lengths $\ell(e)$ to the edges $e$ of the graph $\Gamma$. 

If a stable graph $\Gamma'$ is obtained from $\Gamma$ by {\em contracting} a subset $E_0 \subset E(\Gamma)$ of the edges of $\Gamma$, we view the corresponding cone $\sigma_{\Gamma'}$ as the face of $\sigma_\Gamma$ defined by the conditions that $\ell(e_0)=0$ for $e_0 \in E_0$.
To make the face construction more flexible, recall the notion of a \emph{morphism of stable graphs} $$\varphi: \Gamma \to \Gamma'\,,$$
encoding a particular
way that $\Gamma'$ is obtained from $\Gamma$ by contracting a subset
of edges. \edits{See \cite[Appendix A]{GraberPandharipande} for the definition of these morphisms (which are called $\Gamma'$-structures on $\Gamma$ in \cite{GraberPandharipande}).}\footnote{\edits{See also \cite[Definition 2.5]{Schmitt2020-Intersections} for another exposition of stable graphs and morphisms.}}  Part of the data of $\phi$ is an injective map $\phi_E : E(\Gamma') \to E(\Gamma)$ identifying the edges of $\Gamma'$ as edges of $\Gamma$ not contracted by $\phi$. We then obtain a natural map of cones
\begin{equation*}
    \iota_\varphi : \sigma_{\Gamma'} \to \sigma_\Gamma\, , \ \ \ \ell' \mapsto \left(\ell : e \mapsto 
    \begin{cases}
    \ell'(\phi_E^{-1}(e)) & \text{for }e \in \phi_E(E(\Gamma')).\\
    0 & \text{otherwise.}
    \end{cases}
    \right)
\end{equation*}
representing $\sigma_{\Gamma'}$ as a face of $\sigma_\Gamma$.

The cone stack $\Sigma_{\oM_{g,n}}$ is defined as the direct limit \edits{(in the 2-categorical sense)}
\begin{equation}
\label{styfandef}
    \Sigma_{\oM_{g,n}} = \varinjlim_{\Gamma \in \mathcal{G}_{g,n}} \sigma_\Gamma
\end{equation}
over the category $\mathcal{G}_{g,n}$ of stable graphs (with morphisms $\varphi: \Gamma \to \Gamma'$ as above and associated morphisms $\iota_\varphi: \sigma_{\Gamma'} \to \sigma_\Gamma$ of the corresponding cones). 

The limit
\eqref{styfandef}
is formally defined as a cone stack in the sense of \cite{Cavalieri2020-Conestack} \edits{where 
    $\Sigma_{\oM_{g,n}}$
is denoted $\ca M_{g,n}^{\trop}$}, see in particular \cite[Sections 3.3--3.4]{Cavalieri2020-Conestack}.{\footnote{\edits{ 
$\Sigma_{\oM_{g,n}}$
is not the same as the generalised cone complex of \cite{ACP}, which can be seen as a kind of coarse moduli space for
$\Sigma_{\oM_{g,n}}$.}}} However, all the additional data that we will require  (subdivisions and piecewise polynomial functions on fans) will be defined on the individual cones $\sigma_\Gamma$.
It will {\em not} be necessary for the reader to recall the machinery of cone stacks to follow the remainder of Section
\ref{intro111}.

A \emph{subdivision} $\widetilde{\Sigma}
\rightarrow \Sigma_{\oM_{g,n}}$
 is a cone stack specified by a collection $(\widetilde{\Sigma}_\Gamma)_{\Gamma \in \mathcal{G}_{g,n}}$ of fans satisfying:
\begin{itemize}
    \item[(i)] each $\widetilde{\Sigma}_\Gamma $ is a collection of finitely many rational polyhedral cones
    in  $(\mathbb{R}_{\geq 0})^{E(\Gamma)}$, and has total support  
    equal to $\sigma_\Gamma
    = (\mathbb{R}_{\geq 0})^{E(\Gamma)}$,
    \item[(ii)] the fans $\widetilde{\Sigma}_\Gamma$ are compatible with morphisms $\varphi: \Gamma \to \Gamma'$ in the sense that $$\iota_\varphi^{-1}(\widetilde{\Sigma}_\Gamma) = \widetilde{\Sigma}_{\Gamma'}\, .$$
\end{itemize}
The subdivision 
$\widetilde{\Sigma}\to \Sigma_{\oM_{g,n}}$ determines 
a proper birational morphism $\widetilde{\cM} \to \oM_{g,n}$,
which we call a
{\em log modification}.

The essential idea is that the fan $\widetilde{\Sigma}_\Gamma$ determines, in \'etale local coordinates, the toric modifications near the boundary stratum associated to $\Gamma$. 
 The compatibility of the subdivisions $\widetilde{\Sigma}_\Gamma$ with face maps ensures that these modifications glue to a global birational morphism. A detailed definition is given in Section \ref{sec:subdivisions}.
We illustrate a subdivision of the cone stack $\Sigma_{\oM_{1,2}}$ and the associated log modification, in Figure \ref{fig:M12logblow-up}.

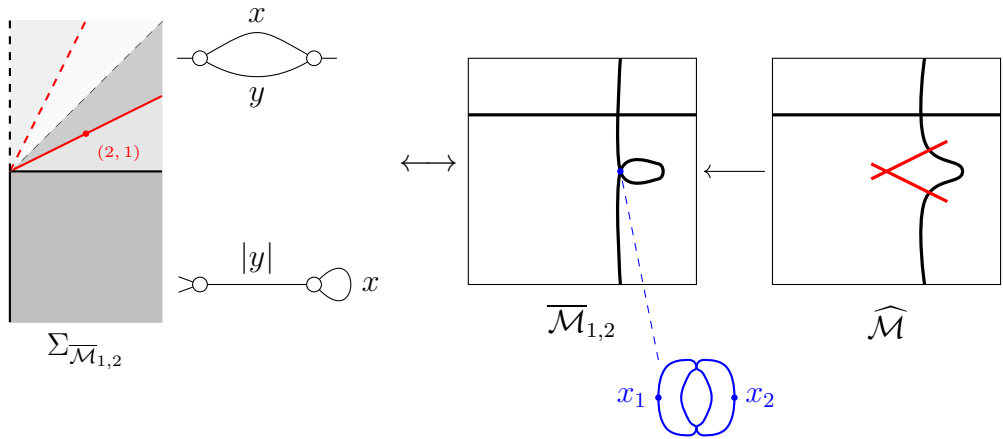
\begin{figure}[htb]
\[
\begin{tikzpicture}[baseline=0pt, vertex/.style={circle,draw,font=\Large,scale=0.5}]
\filldraw[color=black!20] (0,0)--(2,1)--(2,2)--(0,0);
\filldraw[color=black!10] (0,0)--(2,0)--(2,1)--(0,0);
\filldraw[color=black!25] (0,0)--(2,0)--(2,-2)-- (0,-2)--(0,0);

\draw[thick, dashed] (0,0) -- (2,2);
\draw[thick, red] (0,0) -- (2,1);
\draw[thick] (0,0) -- (2,0);
\draw[thick] (0,0) -- (0,-2);

\node[vertex] (A) at (2.5,1.5) {};
\node[vertex] (B) at  (4,1.5) {};
\draw (A) to[bend left, looseness = 1.5] node[midway, above] {$x$} (B)  ;
\draw (A) to[bend right] node[midway, below] {$y$} (B);
\draw (A) -- ++(180:0.3cm); 
\draw (B) -- ++(0:0.3cm);

\node[vertex] (C) at (2.5,-1.5) {};
\node[vertex] (D) at  (4,-1.5) {};
\draw (C) to node[midway, above] {$|y|$} (D)  ;
\draw (D) to[out = 45, in = -45, looseness=14] node[midway, right] {$x$} (D);
\draw (C) -- ++(160:0.3cm);
\draw (C) -- ++(200:0.3cm);

\draw (1,-2) node[below] {$\Sigma_{\oM_{1,2}}$};
\filldraw[red] (1,0.5) circle (1pt) node[below right]{\tiny $(2,1)$};

\filldraw[color=black!2] (0,0)--(1,2)--(2,2)--(0,0);
\filldraw[color=black!6] (0,0)--(0,2)--(1,2)--(0,0);

\draw[thick, red, dashed] (0,0) -- (1,2);
\draw[thick, dashed] (0,0) -- (0,2);

\end{tikzpicture}
\longleftrightarrow
\begin{tikzpicture}[baseline=0pt]
\draw (-3.5,-1.5) rectangle (-0.5,1.5); 
\draw (-2,-2)  node {$\oM_{1,2}$};
\draw (0.5,-1.5) rectangle (3.5,1.5);
\draw (2,-2)  node {$\widehat \cM$};
\draw[->] (0.4,0) -- (-0.4, 0);

\draw[very thick] (-3.5,0.75) -- (-0.5,0.75);
\draw[very thick] plot [smooth, tension = 0.6, xshift=0.5 cm] coordinates {(-2,1.5) (-2,0) (-1.5,-0.1) (-1.5,0.1) (-2,0) (-2,-1.5)  };
\draw[blue, dashed] (-1,-2.5) -- (-1.5,0);
\draw[thick, blue] plot [smooth cycle, tension = 0.7, xshift=-0.5 cm, yshift = -3cm] coordinates { (-0.1,0.5) (0,0.4) (-0.1,0.3) (-0.2,0) (-0.1,-0.3) (0,-0.4) (-0.1,-0.5) (-0.4,-0.4) (-0.5,0) (-0.4,0.4) };
\draw[thick, blue, xscale=-1] plot [smooth cycle, tension = 0.7, xshift=0.5 cm, yshift = -3cm] coordinates { (-0.1,0.5) (0,0.4) (-0.1,0.3) (-0.2,0) (-0.1,-0.3) (0,-0.4) (-0.1,-0.5) (-0.4,-0.4) (-0.5,0) (-0.4,0.4) };
\filldraw[blue] (-1.5,0) circle (1pt);
\filldraw[blue] (-1,-3) circle (1pt) node[left]{$x_1$};
\filldraw[blue] (0,-3) circle (1pt) node[right]{$x_2$};

\draw[very thick, xshift=4 cm] (-3.5,0.75) -- (-0.5,0.75);
\draw[very thick, xshift=4cm] plot [smooth, tension = 0.8, xshift=0.5 cm] coordinates {(-2,1.5) (-2,0.4) (-1.5,0) (-2,-0.4) (-2,-1.5)  };
\draw[very thick, red, xshift=0.3cm] (1.5,-0.1) -- (2.5,0.4);
\draw[very thick, red, xshift=0.3cm] (1.5,0.1) -- (2.5,-0.4);
\end{tikzpicture}
\]    
    \caption{The cone stack $\Sigma_{\oM_{1,2}}$ with a subdivision (in red) and the corresponding log modification $\widehat{\cM}$ of $\oM_{1,2}$, which replaces the self-intersection of the boundary divisor of irreducible curves with a chain of rational curves of length $2$. For the cone stack (on the left), we draw the double cover of the upper (stacky) cone, and correspondingly, the local pictures around the self-intersection of $\delta_\text{irr}$ on the right represent an \'etale double cover of the neighbourhood of this point. \edits{If $X$ and $Y$ are the \'etale local coordinates around the node corresponding to $x$ and $y$, then the three affine patches of the subdivision are given by 
    ${\mathbb{C}}[X,Y,t]/(X - tY^2)$, ${\mathbb{C}}[X,Y,u,v]/(Y^2-uX, X^2-vY, XY-uv)$, and ${\mathbb{C}}[X,Y,s]/(Y - sX^2)$ respectively}. 
    }
    \label{fig:M12logblow-up}
\end{figure}

There is an equivalence of categories between log modifications and the subdivisions of $\Sigma_{\oM_{g,n}}$. Under the equivalence, the blow-up of $\oM_{g,n}$ along a normal closed stratum corresponds to the star subdivision along the barycenter of the corresponding cone. The full iterated boundary blow-up of $\oM_{g,n}$, called the explosion in \cite[Section 5]{Molcho2021-Hodgebundle}, precisely corresponds to the full iterated star subdivision of $\Sigma_{\oM_{g,n}}$. It is however not straightforward in general to describe the log modification corresponding to a given subdivision in terms of more familiar algebro-geometric operations.

\subsection{Log intersection theory}
    \label{subsec: Logintersectiontheory}

We can now define the logarithmic Chow ring of $\oM_{g,n}$:
\begin{equation}\label{lim99}
    \mathsf{logCH}^*(\oM_{g,n}) = \varinjlim_{\widetilde{\cM}} \mathsf{CH}^*(\widetilde{\cM})\,.
\end{equation}
The direct limit is taken over those log modifications $\widetilde{\cM} \to \oM_{g,n}$  where the domain $\widetilde{\cM}$ is a nonsingular Deligne-Mumford stack. 
There exists a morphism $$\widetilde{\cM} \to \widetilde{\cM}'$$ precisely if the associated subdivision $\widetilde{\Sigma}$ refines the subdivision $\widetilde{\Sigma}'$. Given such a morphism, there is a pullback map $\mathsf{CH}^*(\widetilde{\cM}') \to \mathsf{CH}^*(\widetilde{\cM})$ 
which is used to define the above direct limit.

A singular Deligne-Mumford
stack $\cM$ has an operational Chow ring $\mathsf{CH}^*_{\mathsf{op}}(\cM)$, as defined in \cite[Section 5]{Vistoli1989Intersection-th}.
Since operational classes admit pullbacks under arbitrary maps, the limit \eqref{lim99} can  be taken over {\em all} log modifications 
$\widetilde{\cM} \to \oM_{g,n}$ using 
operational Chow theory:
\begin{equation}\label{lim99!}
    \mathsf{logCH}^*(\oM_{g,n}) = \varinjlim_{\text{all\, } \widetilde{\cM}} \mathsf{CH}^*_{\mathsf{op}}(\widetilde{\cM})\,.
\end{equation}
Definitions \eqref{lim99} and
\eqref{lim99!} agree since
every log modification can be further modified to desingularize the domain and since operational and classical Chow groups agree for nonsingular Deligne-Mumford stacks \cite[Proposition 5.6]{Vistoli1989Intersection-th}.

Viewing $\oM_{g,n}$ as the trivial log modification of itself, there exists a canonical algebra morphism, $$\mathsf{CH}^*(\oM_{g,n}) \to \mathsf{logCH}^*(\oM_{g,n})\,,$$
which is injective, since an inverse map of $\mathbb{Q}$-vector spaces
$$\mathsf{logCH}^*(\oM_{g,n}) \to \mathsf{CH}^*(\oM_{g,n})$$
is given by proper pushforward under the maps $\widetilde{\cM}\to \oM_{g,n}$. 
The second map fails to be a ring morphism in general.
For more details on logarithmic Chow rings and intersection theory see 
\cite{Barrott2019Logarithmic-Cho,HolmesSchwarz, Molcho2021-Hodgebundle, Molcho2021-Case-Study}.

\subsection{Resolving the Abel-Jacobi map}
\label{resAJm}

We describe now Holmes' construction \cite{Holmes2017Extending-the-d} of the log double ramification cycle in $\mathsf{logCH}^*(\oM_{g,n})$.
Let  
\begin{equation*} 
\mu:{\mathcal{J}}_{g,n} \rightarrow \oM_{g,n}
\end{equation*}
be the (noncompact) 
universal Jacobian of line bundles of multidegree $0$ on stable curves of genus $g$ with $n$ marked points.
Recalling that $A=(a_1,\ldots,a_n)$ is a vector of
 integers with sum $k(2g-2+n)$, the Abel-Jacobi map $\mathsf{aj}_A$ is the rational map defined by
\begin{equation}\label{cc55g}
\mathsf{aj}_A: \oM_{g,n} \dashrightarrow {\mathcal{J}}_{g,n}\,, \ \ \ \mathsf{aj}_A([C,x_1,\ldots,x_n])=
(\omega_C^\mathrm{log})^{\otimes k} \Big(-\sum_{i=1}^n a_i x_i\Big)\, .
\end{equation}
In \cite{Holmes2017Extending-the-d}, Holmes constructs a universal birational map $\mathcal{U}^\diamond_{g,A} \to \oM_{g,n}$ on which the Abel-Jacobi map \eqref{cc55g} can be extended. The space $\mathcal{U}^\diamond_{g,A}$ sits as an open substack in a 
\emph{noncanonical}
log modification 
\begin{equation} \label{vrk554}
\rho:\oM_{g,A}^\diamond \rightarrow\oM_{g,n}\,,
\end{equation}
so we have a diagram
\tikzset{
  symbol/.style={
    draw=none,
    every to/.append style={
      edge node={node [sloped, allow upside down, auto=false]{$#1$}}}
  }
}
\[
\begin{tikzcd}
&\rho ^*\mathcal{J}_{g,n} \arrow[d] \arrow[r] & \mathcal{J}_{g,n} \arrow[d]\\
\mathcal{U}^\diamond_{g,A} \arrow[ur,"\mathsf{aj}_A"] \arrow[r, symbol=\subseteq] & \oM_{g,A}^\diamond \arrow[r,"\rho"] & \oM_{g,n} \arrow[u, bend right, dashed, swap, "\mathsf{aj}_A"]
\end{tikzcd}\,.
\]
Denoting by $e \subseteq \rho^* \mathcal{J}_{g,n}$ the preimage of the zero section of the universal Jacobian, 
the inverse image
$\mathsf{aj}_A^{-1}(e)
\subset \mathcal{U}^\diamond_{g,A}$
is proper (compact).
The refined intersection product $\mathsf{aj}_A^*([e])$
then defines a cycle class in $\mathsf{CH}^{g}(\oM_{g,n}^\diamond)$
which represents $\mathsf{logDR}_{g,A}$.

While the construction of the blow-up 
\eqref{vrk554}
is  {not}
canonical nor even explicit\footnote{An abstract resolution of singularities is required.}, the resulting logarithmic cycle class 
\begin{equation*} 
\mathsf{logDR}_{g,A}\ \in \mathsf{logCH}^g(\oM_{g,n})
\, ,
\end{equation*}
is well-defined by \cite[Theorem 1.2]{Holmes2017Extending-the-d}.
The most basic properties are:

\begin{enumerate}
\item[$\bullet$] If $n=0$ and $A=\emptyset$, then
$\mathsf{logDR}_{g,\emptyset} = (-1)^g \lambda_g$,
where the top Chern class of the Hodge bundle $\lambda_g$ is
pulled back from $\mathsf{CH}^g(\oM_g)$, see \cite{Molcho2021-Case-Study}.

\item[$\bullet$] If $k=0$, the class $\mathsf{logDR}_{g,A}$
pushes forward to the standard double ramification
cycle
$\mathsf{DR}_{g,A}\in \mathsf{CH}^g(\oM_{g,n})$
defined via the moduli space of rubber maps by \cite[Theorem 1.3]{Holmes2017Extending-the-d}.
\end{enumerate}
Since the $n=0$ case is solved, we will always assume $n\geq 1$.

In order to calculate $\mathsf{logDR}_{g,A}$ using the above construction via Abel-Jacobi theory, several
difficulties must be overcome: 
\begin{enumerate}
    \item [(i)] Since the construction of the blow-up $\rho:\oM_{g,A}^\diamond \rightarrow \oM_{g,n}$ in \cite{Holmes2017Extending-the-d} is only implicit  (depending upon noncanonical choices),
    a direct study of $\oM_{g,A}^\diamond$ is difficult.
    \item[(ii)] The class $\mathsf{aj}_A^*([e])$ arises from
    the geometry of the Abel-Jacobi map on the
    open set $\mathcal{U}^\diamond_{g,A}$, not the global
    geometry of $\oM_{g,A}^\diamond$, and so is
    not directly accessible via intersection theory on $\oM_{g,A}^\diamond$.
    \item[(iii)] Even if the above issues
    could be overcome, in what language would
    the answer be expressed?
\end{enumerate}

Our solution to both (i) and (ii) is to find geometrically 
meaningful models for $\oM_{g,A}^\diamond$ via the moduli spaces of
line bundles on quasi-stable curves. The noncanonical
aspect of $\oM_{g,A}^\diamond$ is not completely lost. There is still a choice of stability
condition needed to define the moduli space of line bundles, but the stability
condition is the only choice. By applying the main result of
\cite{HolmesSchwarz} together with the formula of
 \cite{BHPSS} 
for
the universal double ramification
cycle, we can calculate  $\mathsf{logDR}_{g,A}$.
For (iii),
the answer is expressed in the subring of tautological classes in
$\mathsf{logCH}^\star(\oM_{g,n})$,
suggested by D. Ranganathan{\footnote{In the lecture by D. Ranganathan in
the {\em Algebraic Geometry and Moduli Zoominar} at ETH Z\"urich in April 2020.}} and
developed in \cite{HolmesSchwarz,Molcho2021-Hodgebundle, Molcho2021-Case-Study}. Tautological classes include
$\kappa_1$, the cotangent line classes $\psi_i$,
and classes coming from the algebra of piecewise polynomials on the
cone stack associated to $(\oM_{g,n},\partial \oM_{g,n})$.

As discussed above, the formula for $\mathsf{logDR}_{g,A}$
depends upon the choice of an appropriate{\footnote{The formula is
well-defined for all {\em nondegenerate} stability conditions, but calculates
$\mathsf{logDR}_{g,A}$ only for {\em small} stability conditions. The precise definitions
are given in Section \ref{sec:stability_conditions}.}} stability condition.
The dependence
can be useful: special stability conditions can be selected
to simplify the formula depending upon the properties of 
the vector $A$. The wall-crossing study of the formula
leads to relations in $\mathsf{logCH}^g(\oM_{g,n})$
when different stability conditions calculate the
same class.

\subsection{Moduli of line bundles on stable curves} \label{Sect:intromodulilinebundles}
\label{mlbsc}

A basic difficulty in writing a formula for the logarithmic double ramification cycle is the noncompactness of the universal Jacobian $\mathcal{J}_{g,n}$ of
multidegree 0 line bundles on stable curves. As a consequence, the universal space $\mathcal{U}_{g,A}^\diamond$ is also not compact.
Our approach here is to view the rational map $$\mathsf{aj}_A: \oM_{g,n} 
 \dashrightarrow \mathcal{J}_{g,n}$$ as taking values in a \emph{compactified} Jacobian and to resolve the indeterminacies of $\mathsf{aj}_A$. Such resolutions of indeterminancies turn out to be proper and provide modular compactifications of $\mathcal{U}_{g,A}^\diamond$
of precisely the type needed
to compute  $\mathsf{logDR}_{g,A}$.

The construction of compactifications of the moduli spaces of line bundles on stable
curves goes back at least to \cite{Caporaso1994A-compactificat,Pandharipande}. In the past decades, there
has been a continuous study of these spaces, see \cite{AbreuPaciniUniversal, BiniFontanariViviani, CaporasoChrist, C-MKV, Esteves2001Compactifying-t, EstevesPacini, Kass2017The-stability-s,Melo2015Compactificatio, Melo2019Universal}. 
A short summary of what we need for our approach to the logarithmic
double ramification cycle is presented here.

Since  $\mathcal{J}_{g,n}$ is not compact, we can find 1-parameter families of nonsingular curves carrying
line bundles of degree $0$ 
 degenerating to a  stable nodal curve where
 the line bundle fails to degenerate to a line bundle of degree $0$ on every irreducible component of the nodal curve. The issue is not simply about  multidegrees. Searching
 for limits among all line bundles of total degree $0$ does not suffice and, in 
 fact, further complicates the geometry: the Picard scheme $\Pic(\mathcal{C}_{g,n}/\oM_{g,n})$ of 
 line bundles of all multidegrees on stable curves
 $$\mathcal{C}_{g,n} \to \oM_{g,n}$$
 is neither separated nor universally closed.  

In order to obtain a compactification of 
$\mathcal{J}_{g,n}$, we
must include limits that are
not line bundles on stable
curves. There are two  equivalent approaches.
The first is to consider the moduli space of all rank $1$ torsion free sheaves on stable curves. The second, which we will follow, is to
consider the moduli space of \emph{admissible} line bundles on \emph{quasi-stable} curves. 

A flat family of nodal curves $\mathcal{C} \to \mathcal{S}$ is {\em quasi-stable} if the
relative dualizing sheaf $\omega_{\mathcal{C}/\mathcal{S}}$
is nef and chains
of unstable components have length at most 1.
A line bundle 
$\mathcal{L}$ on 
$\mathcal{C}$ is {\em admissible} if $\mathcal{L}$ has degree $1$ on each unstable component of $\mathcal{C}$. The data of an admissible line bundle on a quasi-stable curve is equivalent to the data of a rank $1$ torsion free sheaf on the associated stabilization, 
$$\mathsf{st}:\mathcal{C} \to \mathcal{C}^{\textup{st}}\, ,$$
by assigning to an admissible line bundle $\mathcal{L}$ the rank $1$ torsion free sheaf $\mathsf{st}_*\mathcal{L}$, \edits{see \cite{EstevesPacini}}. 

\begin{figure}
\centering
\begin{tikzpicture}
\draw[thick, domain= 100:260] plot({3+cos \x},{2*sin \x});
\draw[thick, domain= -80:80] plot({2+cos \x}, {2*sin \x});

\draw[thick, domain = 290:365] plot({-3+cos \x},{1.75+4*sin \x}); 
\draw[thick, domain = 175:250] plot({-2+cos \x},{1.75+4*sin \x});
\draw[thick, red, domain = -3.5:-1.5] plot({\x},{1.75});
\node[draw,circle,inner sep=1.5pt, red, fill] at (2.5,1.75){};
\node[above] at (-2.5,2){$\mathcal{L}|_{\mathbb{P}^1} \cong \mathcal{O}(1)$};
\node[below] at (-2.5,-2){$C' = C_1 \cup \mathbb{P}^1 \cup C_2$};
\node[below] at (2.5,-2){$C = C_1 \cup C_2$};
\draw[-stealth] (-1,0)--(1,0);

\end{tikzpicture}
\caption{A quasi-stable curve $C'$ with stabilization $C$, and an admissible line bundle $\mathcal{L}$ on it.}
\end{figure}
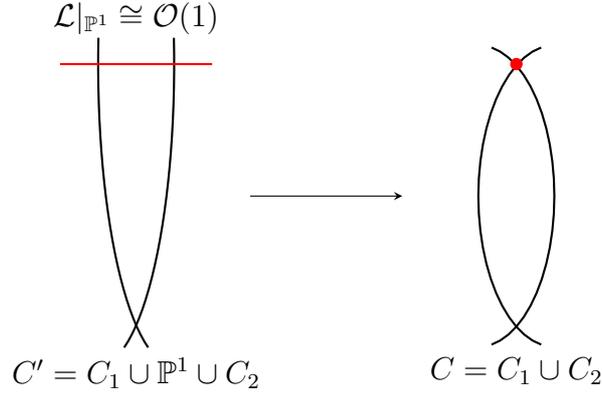

The resulting moduli space parametrizing 
either admissible line bundles on quasi-stable curves
or
rank $1$ torsion free sheaves on stable curves is
the {\em universal Jacobian} which
is
universally
closed 
over the moduli
space of stable curves, but  
 {not} separated.

 The failure of separation of the universal Jacobian occurs because of the possibility of twisting $\mathcal{L}$
by divisors  $\mathcal{T}\subset  \mathcal{C}$ supported
over nodal curves.
Twisting defines an equivalence relation on the \emph{universal} Jacobian by declaring 
$$
(C,\mathcal{L}) \sim (C,\mathcal{L} \otimes \mathcal{O}(\mathcal{T}))\, .
$$
 To find compactifications of the Jacobian $\mathcal{J}_{g,n}$
 which are
 proper and separated, we
 must choose a unique representative in each equivalence class of twistings. Stability conditions exactly
 make such choices.


Following \cite{Kass2017The-stability-s}, a {\em stability condition} $\theta$ of type $(g,n)$ and degree $d$
is a rule which assigns a 
rational number to every irreducible component of
every stable curve
$(C,x_1,\ldots,x_n)$
of genus $g$ with $n$ marked points and satisfies two properties:
\begin{enumerate}
    \item [(i)]
    the sum of the   
    values of $\theta$ over the 
    irreducible components of $C$
    equals  $d$,
    \item[(ii)]
 $\theta$ is additive under all partial smoothings:
\end{enumerate}

\[
\begin{tikzpicture}[scale=0.75]
\coordinate(a1) at (4,2){};
\node[right] at (a1){$C$};
\coordinate(a2) at (3,0.5){};
\coordinate(a3) at (2.5,1){};
\coordinate(a4) at (2,0){};
\coordinate(a5) at (2.5,-1){};
\coordinate(a6) at (3,-0.5){};
\coordinate(a7) at (4,-2){};

\coordinate (b1) at (0,2){};
\node[right] at (b1){$C_1$};
\coordinate (b2) at (-0.5,0){};
\node[draw,circle,inner sep=1pt,fill] at (b2){};
\coordinate (b3) at (-1,-0.5){};
\coordinate (b4) at (-1,0.5){}; 
\coordinate (b5) at (0,-2){};
\node[right] at (b5){$C_2$};
\node[below] at (2, -3){$\theta(C_1) +\theta(C_2) = \theta(C)$};

\coordinate(c1) at (-1,-4){};
\node[draw,circle,inner sep=1pt,fill] at (0,-4){};
\coordinate(c2) at (5,-4){};
\node[below] at (c2){$\Spec \bb{C}[[t]]$};

\draw plot[smooth] coordinates {(a1) (a2) (a3) (a4) (a5) (a6) (a7)};
\draw plot[smooth] coordinates {(b1) (b2) (b3)};
\draw plot[smooth] coordinates {(b5) (b2) (b4)};
\draw plot[smooth] coordinates {(c1) (c2)};
\end{tikzpicture}
\] 

\noindent We may also view a stability condition $\theta$ in combinatorial terms:
 $\theta$ is an assignment of a rational value to every vertex of every stable graph $G$ of type $(g,n)$ which sums to $d$ and
  which is additive with respect to contractions of edges.

There is
a {\em trivial stability condition} $\theta^{\mathsf{tr}}$ of degree $0$  given by
$$
\theta^{\mathsf{tr}}(D) = 0
$$
for every irreducible
component $D\subset C$.
A more interesting example 
\edits{(introduced in \cite{Caporaso1994A-compactificat,Melo2015Compactificatio})} 
is the {\em canonical stability condition} $\theta^{K}$ of degree $2g-2+n$ defined by
$$\theta^K(D) = \mathsf{deg}\left[ \, \omega_C\Big(\sum_{i=1}^n x_i\Big)\Big|_D\right]\,  $$
for an irreducible component $D\subset C$.
The value of $\theta^K$
depends on the genus
of $D$ and the number nodes and markings 
of $C$ which lie on $D$.

Fix a stability condition $\theta$ of degree $d$. An admissible line bundle $\mathcal{L}$ of degree $d$ on a quasi-stable curve $C$ is   {\em $\theta$-stable} (respectively, {\em $\theta$-semistable}) if and only if,
for every proper subcurve\footnote{We will later phrase these inequalities in combinatorial terms using the dual graph of $C$.}
$\emptyset \subsetneq\ul C \subsetneq C$ with neither $\ul C$ nor its complement consisting entirely of unstable components, we have
$$  \theta(\ul C) -\frac{E(\ul C,\ul C^c)}{2}  \ < \, (\leq)\ \  \mathsf{deg}(\mathcal{L}|_{\ul C})\ \   < \, (\leq) \ \  \
\theta(\ul C) +
\frac{E(\ul C,\ul C^c)}{2} \, ,$$
where $E(\ul C,\ul C^c)$ is the number of intersection points of $\ul C$ with the complementary subcurve
$\ul C^c\subset C$ and $$\theta(\ul C) = \sum_{D \subset \ul C} \theta(D)$$ is the sum of the values of  $\theta$ over all irreducible components $D \subset \ul C$.\footnote{The stability condition $\theta$ extends to quasi-stable curves by assigning weight $0$ to all unstable components.}
In other words, an admissible line bundle
$\mathcal{L}$ is $\theta$-semistable if the degree of $\mathcal{L}$ restricted to each subcurve $\ul C \subset C$ is close enough to the assigned
value $\theta(\ul C)$.

The connection with the above discussion about compactifications of the Jacobian can be seen as follows.
A nontrivial
twist by $\mathcal{O}(\mathcal{T})$ in
families changes 
the degree on $\ul C$ by at least $E(\ul C,\ul C^c)$ for some
subcurve $\ul C\subset C$.
Thus, there can be 
at most one representative of the equivalence class of a line bundle which is $\theta$-stable. 
A stability condition $\theta$ is {\em nondegenerate} if every $\theta$-semistable line bundle on every quasi-stable curve $C$ is stable. 

Let $\theta$ be a nondegenerate stability condition of type $(g,n)$.
By a result of \cite{Kass2017The-stability-s} (\edits{see also \cite{Melo2019Universal} for an earlier, alternative perspective}), there exists a moduli stack 
$\mathcal{P}^{\theta}_{g,n}$
of
$\theta$-stable admissible line bundles on quasi-stable curves, satisfying the following properties:
\begin{enumerate}
    \item [(i)] The stack $\mathcal{P}^{\theta}_{g,n}$ is proper, nonsingular,
    and of dimension $4g-3+n$.
    \item [(ii)] There is natural morphism $\mu:\mathcal{P}^{\theta}_{g,n}\rightarrow \oM_{g,n}$.
    \item[(iii)] The stack $\mathcal{P}^{\theta}_{g,n}$ carries a universal quasi-stable curve with a universal admissible line 
    bundle.{\footnote{The condition $n\geq 1$ is used for the existence of the universal line bundle.}}
\end{enumerate}
In general, neither the
trivial stability 
condition $\theta^{\mathrm{tr}}$
nor
the canonical stability condition $\theta^K$
is nondegenerate.
However, in case $n \ge 1$,
explicit perturbations can be easily constructed in both cases and seen
to be nondegenerate (as explained in \cite{Kass2017The-stability-s} and reviewed in Section \ref{sec:stability_conditions}).

A stability condition $\theta$ of degree 0 is {\em small} 
if line bundles of multidegree $0$ on the universal curve $\mathcal{C}_{g,n}$ are $\theta$-stable. 
 Nondegenerate perturbations of $\theta^{\mathsf{tr}}$
sufficiently close
to $\theta^{\mathsf{tr}}$ are small.
For our formula
for the logarithmic double ramification cycle, we will
require $\theta$ to be a small nondegenerate stability condition of degree 0.

Recall that $A=(a_1,\ldots, a_n)$ is a vector of integers summing to $k(2g - 2 + n)$. Given a small
nondegenerate stability condition $\theta$ of degree 0 and type $(g,n)$, there is a rational Abel-Jacobi section of
the morphism $\mu$,
$$\mathsf{aj}_A: \oM_{g,n} \dashrightarrow \mathcal{P}^\theta_{g,n}\, ,$$
defined (as before) by 
$$\mathsf{aj}_A([C,x_1,\ldots,x_n]) = (\omega_C^\mathrm{log})^{\otimes k}\Big(-\sum_{i=1}^n a_i x_i\Big)\, .$$
As the stability condition $\theta$ is nondegenerate, the space $\mathcal{P}_{g,n}^\theta$
is proper. 
The stability condition $\theta$ then canonically{\footnote{The subdivision $\widetilde{\Sigma}^\theta \rightarrow \Sigma_{\oM_{g,n}}$ which determines the log
 modification $\rho$ is defined in Section \ref{sasc}.
 The resolution of the Abel-Jacobi map 
 $\mathsf{aj}: \oM_{g,A}^\theta \rightarrow \mathcal{P}^\theta_{g,n}$
 is constructed in Section \ref{sec:stability_conditions},  see Definition \ref{def:Mgntheta_detailed}. }} 
determines a log modification
\begin{equation}\label{eq:intro_resolve_aj}
    \rho: \oM_{g,A}^\theta \rightarrow \oM_{g,n}
\end{equation}
 which resolves the Abel-Jacobi map, 
$$\mathsf{aj}: \oM_{g,A}^\theta \rightarrow \mathcal{P}^\theta_{g,n}\, , 
 \ \ \ \ \ \mathsf{aj}= \mathsf{aj}_A \circ \rho\, .$$
Over $\oM^\theta_{g,A}$, we have a universal family 
$$\pi: \mathcal{C}^\theta \rightarrow \oM^\theta_{g,A}$$
which is a quasi-stable model of the pullback $\oM_{g,A}^\theta \times_{\oM_{g,n}} \mathcal{C}$
of the 
universal stable curve, 
and an admissible line bundle $$\mathcal{L}^\theta
\rightarrow \mathcal{C}^{\theta}$$ obtained from the universal admissible line bundle on the moduli stack of $\theta$-stable sheaves.
Since $\theta$ is small, $\mathcal{J}_{g,n}
\subset \mathcal{P}_{g,n}^\theta$
is an open substack.
In fact,  $\mathcal{U}_{g,A}^\diamond$ is an open substack of $\oM_{g,A}^\theta$, and $\oM_{g,A}^\theta$ can play the role of $\oM_{g,A}^\diamond$ in the definition of \cite{Holmes2017Extending-the-d}. 
The class $$\mathsf{aj}_A^{*}[e] \in \mathsf{CH}^g_{\mathsf{op}}(\oM_{g,A}^\theta)$$  therefore represents $\mathsf{logDR}_{g,A}$.   

The geometry here
is much more favorable than for the abstractly defined spaces $\oM_{g,A}^\diamond$. After the small nondegenerate stability condition $\theta$
has been chosen, there are {\em no} further choices, and we have complete understanding of the points of  $\oM_{g,A}^\theta$.
A modular
interpretation of $\oM_{g,A}^\theta$ is 
described in Definition \ref{def:S_theta} of Section \ref{ndsc}. Furthermore, $\oM_{g,A}^\theta$ naturally \edits{carries a universal double ramification cycle class 
$$
\mathsf{DR}_{g, \mathcal L^\theta}^\mathsf{op} = \varphi_{\mathcal L^\theta}^* \mathsf{DR}_{g, \emptyset}^\mathsf{op} \in \mathsf{CH}^g_\mathsf{op}(\oM_{g,A}^\theta)\, ,
$$
where $\mathfrak{Pic}_g$ is the universal Picard stack (parameterizing nodal genus $g$ curves together with a line bundle), the map   $$\varphi_{\mathcal L^\theta} : \oM_{g,A}^\theta \to \mathfrak{Pic}_g$$ is defined by the family $(\mathcal{C}^\theta, \mathcal{L}^\theta)$ over $\oM_{g,A}^\theta$, and $\mathsf{DR}_{g, \emptyset}^\mathsf{op} \in \mathsf{CH}^g_\mathsf{op}(\mathfrak{Pic}_g)$ is the universal double ramification cycle of \cite{BHPSS}.\footnote{The notation $\emptyset$ in $\mathsf{DR}_{g, \emptyset}^\mathsf{op}$ indicates the absence of marked points.}} 
By applying
the theory of almost-twistable families developed in \cite{HolmesSchwarz},
we obtain
the first form of our main result (which simply states that the two natural classes on 
$\oM_{g,A}^\theta$
are equal).


\begin{introtheorem} \label{111ooo}
Let $\theta$ be a small nondegenerate  stability condition. 
The universal double ramification cycle associated to
the line bundle $\mathcal{L}^\theta$
on $\mathcal{C}^\theta
\to
\oM_{g,A}^\theta$,
\begin{equation*}
\mathsf{DR}^{\mathsf{op}}_{g,\emptyset, \mathcal{L}^\theta} 
\in \mathsf{CH}^g_{\mathsf{op}}( \oM_{g,A}^\theta)\, ,
\end{equation*}
provides a representative for
$\mathsf{logDR}_{g,A}$.
\end{introtheorem}

\edits{Though} the claim of
Theorem \ref{111ooo} 
is conceptual,
the statement can be transformed
into an explicit formula for $\mathsf{logDR}_{g,A}$ since the universal double ramification cycle has been explicitly computed in \cite{BHPSS}. 
Theorem \ref{111ooo} is proven in Section \ref{sec:conditions_for_almost_twistability}.


\subsection{Formula for \texorpdfstring{$\log \mathsf{DR}$}{logDR}} \label{ffldr}
The final step is to translate Theorem \ref{111ooo}
into an explicit class in the logarithmic
Chow ring using the formula for the universal
double ramification cycle in \cite{BHPSS}.

\subsubsection{Tautological classes in the logarithmic Chow ring}
\label{Sect:logtautclasses}
As we have seen,
 the logarithmic modifications $\widetilde{\cM} \to \oM_{g,n}$ appearing in the definition of the logarithmic Chow ring are described by subdivisions $\widetilde{\Sigma}\to \Sigma_{\oM_{g,n}}$
of the cone stack. Following \cite{HolmesSchwarz, Molcho2021-Hodgebundle, Molcho2021-Case-Study}, we can use the same convex-geometric language to describe natural cycle classes on $\widetilde{\cM}$ which define elements of
$\mathsf{logCH}^*(\oM_{g,n})$.\footnote{The language of log structures and ghost sheaves is used in \cite{HolmesSchwarz}. 
Our presentation avoids ghosts and follows instead 
the language of fans and cones from toric geometry. The comparison of terminology
is straightforward and discussed in more detail in Remark \ref{rem:compare_PP_defs}.}


Let $\Sigma$ be a fan in $\mathbb{R}^m$ with  maximal cones all of dimension $m$ (the dimension of the ambient vector space). A \emph{strict piecewise polynomial} $f \in \sPPoly{\Sigma}$ is a continuous function $f : |\Sigma| \to \mathbb{R}$ on the support $|\Sigma| \subseteq \mathbb{R}^m$ of $\Sigma$ which is given by a polynomial with rational coefficients on each maximal cone of $\Sigma$. A \emph{piecewise polynomial function} $f \in \PPoly{\Sigma}$ is a continuous function $f: |\Sigma| \to \mathbb{R}$ which is a strict piecewise polynomial on \emph{some} subdivision $\Sigma'$ of $\Sigma$.

Given a subdivision $\widetilde{\Sigma} = (\widetilde{\Sigma}_\Gamma)_{\Gamma \in \mathcal{G}_{g,n}}$ of the cone stack $\Sigma_{\oM_{g,n}}$, a {strict piecewise polynomial function} 
$$f = (f_\Gamma)_{\Gamma \in \mathcal{G}_{g,n}} \in \sPPoly{\widetilde{\Sigma}}$$
is a collection of strict piecewise polynomials $f_\Gamma$ on the fans $\widetilde{\Sigma}_\Gamma$, which are compatible with the face maps $\iota_\varphi$. More precisely, for $\varphi: \Gamma \to \Gamma'$, we have $f_\Gamma \circ \iota_\varphi = f_{\Gamma'}$. 
A piecewise polynomial on $\widetilde{\Sigma}$ is then similarly a collection $f = (f_\Gamma)_{\Gamma}$ of piecewise polynomials $f_\Gamma$ satisfying the same compatibility.

The (strict) piecewise polynomials on $\widetilde{\Sigma}$ form a $\mathbb{Q}$-algebra. For 
the log modification $\widetilde{\cM}\to \oM_{g,n}$ associated to $\widetilde{\Sigma}$, 
there are natural ring morphisms
\begin{equation}\label{eq:Phi_def}
    \Phi: \sPPoly{\widetilde{\Sigma}} \to \mathsf{CH}^*_{\mathrm{op}}(\widetilde{\cM})\,
\end{equation}
constructed in \cite[Section 3.3]{HolmesSchwarz}.\footnote{Following our analogy with toric varieties, the equivariant Chow ring of a toric variety is isomorphic to the ring of strict piecewise polynomials on the associated fan, see \cite{Payne2006EquivariantChow}.}
When $\widetilde \Sigma$ is the cone over a simplicial complex the construction proceeds in two steps:
\begin{enumerate}
\item[(i)] There exists a natural map from strict piecewise \emph{linear} functions on $\widetilde{\Sigma}$ to divisor classes on $\widetilde{\cM}$ using the fact that rays of $\widetilde{\Sigma}$ correspond to boundary divisors of  $\widetilde{\cM}$. 
\item[(ii)]
The strict piecewise linear functions generate $\sPPoly{\widetilde{\Sigma}}$ as a $\mathbb{Q}$-algebra.
The unique extension of (i) defines $\Phi$.
\end{enumerate}
In general, we can subdivide  $\widetilde \Sigma$ until it is the cone over a simplicial complex and then push forward the resulting class.

The maps $\Phi$ are compatible under subdivision of $\widetilde{\Sigma}$ and the associated birational modification of $\widetilde{\cM}$, and thus give rise to an algebra morphism
\begin{equation} \label{eqn:Phimapintro}
    \Phi: \PPoly{{\Sigma}_{\oM_{g,n}}} \to \mathsf{logCH}^*({\oM}_{g,n})\,.
\end{equation}
\edits{An abstract construction of $\Phi$ can be found in \cite{Molcho2021-Hodgebundle, Molcho2021-Case-Study}. A more explicit
description of $\Phi$ is presented in
Section \ref{excor}.}

\begin{definition} \cite[Definition 3.18]{HolmesSchwarz}\, 
The \emph{logarithmic tautological ring} $$\mathsf{logR}^*(\oM_{g,n}) \subset \mathsf{logCH}^*(\oM_{g,n})$$
is the subring generated by the image of $\Phi$ and the usual tautological ring\footnote{We refer the reader to \cite{RPSLC} for an introduction to the tautological ring of $\oM_{g,n}$. } $$\mathsf{R}^*(\oM_{g,n}) \subseteq \mathsf{CH}^*(\oM_{g,n}) \subseteq \mathsf{logCH}^*(\oM_{g,n})\, .$$
\end{definition}
In particular, the above definition gives a natural surjection
\begin{equation} \label{eqn:logRQQtensor}
   \mathsf{R}^*(\oM_{g,n}) \otimes_{\mathbb{Q}}  \PPoly{{\Sigma}_{\oM_{g,n}}} \to \mathsf{logR}^*(\oM_{g,n})\,.
\end{equation}
By the results of \cite{HolmesSchwarz, Molcho2021-Case-Study}, we know
$\mathsf{logDR}_{g,A} \in
\mathsf{logR}^*(\oM_{g,n})$.

Our goal is to describe an explicit element of the left side of \eqref{eqn:logRQQtensor} which maps to $\mathsf{logDR}_{g,A}$.
As we will see in Section \ref{Sect:logDRformulaintro}, the formulas describing the element naturally involve taking exponentials of certain piecewise linear functions. Exponentiation leaves the realm of piecewise polynomials and yields  \emph{piecewise power series}. To make sense of the formulas below, we simply observe that the map $\Phi$ above naturally extends to such piecewise power series by first truncating the power series to degree (at most) $3g-3+n$ to obtain a piecewise polynomial.


\subsubsection{The subdivision associated to a stability condition}
\label{sasc}
Let $\theta$ be a nondegenerate stability
condition of type $(g,n)$ and degree 0 (as defined in Section \ref{mlbsc}).
The data of the vector $A = (a_1, \dots, a_n)$ and $\theta$ together
determine
 a canonical subdivision $$\widetilde{\Sigma}^\theta\to \Sigma_{\oM_{g,n}}$$
 which we construct here.
 The subdivision $\widetilde{\Sigma}^\theta$
 defines the
 log modification
$$\rho: \oM_{g,A}^\theta \rightarrow \oM_{g,n}$$
from \eqref{eq:intro_resolve_aj}. 
 We will require the following notation for the construction:

\vspace{8pt}
\noindent $\bullet$ For a 
quasi-stable graph $\Gamma$,
let $\vec E(\Gamma)$ be the set of \emph{oriented edges} of $\Gamma$. The natural map $$\vec E(\Gamma) \to E(\Gamma)$$ has fibers $\{\vec e, \cev e\}$ of size two, the two orientations on a given edge $e$. For $\vec e \in \vec E(\Gamma)$, we write $$\vec e \to v \in V(\Gamma)$$ 
if $\vec{e}$ is incident to and points \emph{towards} $v$. A \emph{cycle} $\gamma$ on $\Gamma$ is a collection of oriented edges forming a cycle without self-intersection. 

Given a family $(C,\mathcal{L})$ of curves and line bundles, the change of the multidegree of $\mathcal{L}$ on a special fiber $C_0$ by twisting with a vertical divisor $\mathcal{T}$ (as in Section \ref{Sect:intromodulilinebundles}) is described by an acyclic flow on $\Gamma_{C_0}$. Acyclic flows appear under the name of \emph{twists} in \cite{Farkas2016The-moduli-spac}. 




\vspace{8pt}
\noindent \label{flowdef} $\bullet$ A \emph{flow} on $\Gamma$ is a function $f : \vec E(\Gamma) \to \mathbb{Z}$ satisfying $f(\cev e) = - f(\vec e)$ for all edges $e \in E(\Gamma)$.
\edits{The group of flows on $\Gamma$ is denoted $\Flow(\Gamma)$}.{\footnote{\edits{In the literature, flows are also known as \emph{1-cocycles}, and the group is denoted $C^1(\Gamma, \bb Z)$.}}} By choosing some fixed orientation on $\Gamma$, we can {\edits{noncanonically}} identify the group $\Flow(\Gamma)$ with $\mathbb{Z}^{E(\Gamma)}$. A flow $f$ is called \emph{acyclic} if there exists no cycle $\gamma$ of $\Gamma$ satisfying
the following conditions: $$\forall \vec e \in \gamma\, ,
\  f(\vec e) \geq 0
\ \ \ \text{and} \ \ \ 
\exists \vec e \in \gamma\, ,
\  f(\vec e) > 0\, .$$ 

\vspace{8pt}
\noindent $\bullet$ A \emph{divisor} on $\Gamma$ is an element of $\bb Z^{V(\Gamma)}$. For $f \in \Flow(\Gamma)$, we write
\begin{equation}\label{eq:div_of_flow}
    \mathrm{div}(f) = \left( \sum_{\vec e \to v} f(\vec e) \right)_{v \in V(\Gamma)} \in \mathbb{Z}^{V(\Gamma)}
\end{equation}
for the divisor of $f$. The first homology group $H_1(\Gamma)$ is precisely the kernel of the group homomorphism $\mathrm{div} \colon \Flow(\Gamma) \to \mathbb{Z}^{V(\Gamma)}$. 
There is  a straightforward generalization of flows and divisors allowing values in an arbitrary abelian group 
instead of $\mathbb{Z}$.

\vspace{8pt}
\noindent $\bullet$ Given a length assignment $\ell : E(\Gamma) \to \mathbb{R}_{\geq 0}$, we have a \emph{pairing}{\footnote{The importance of the pairing will be seen in Lemma \ref{lem:flow_as_slopes}: a flow arises as the slopes of a {\em piecewise linear function} (Definition \ref{def:comb_PL}) on the metric graph $(\Gamma, \ell)$ if and only if the pairing with every element of $H_1(\Gamma)$ vanishes.}}
\begin{equation} \label{eqn:ellpairing}
    \langle f, g \rangle_\ell = \frac{1}{2} \sum_{\vec e \in \vec E(\Gamma)} \ell(e) f(\vec e) g(\vec e)\, 
\end{equation}
for $f,g \in \Flow(\Gamma)$.


\vspace{8pt}
\noindent $\bullet$ 
The multidegree of the line bundle $(\omega_C^{\mathrm{log}})^{\otimes k}(-\sum_i a_i x_i)$ on a quasi-stable curve $C$ with  graph $\Gamma$ is
\begin{equation*}
    \mdeg_{k,A} = \left(k(2g(v)-2+n(v))-\sum_{i\text{ at }v} a_i \right)_{v \in V(\Gamma)} \in \mathbb{Z}^{V(\Gamma)}\, , 
\end{equation*}
where $n(v)$ is the number of half-edges attached to $v$ (the number of markings at $v$ plus the number of ends of edges at $v$). 

\vspace{8pt}

We can now describe the subdivision $\widetilde{\Sigma}^\theta = \{ \widetilde{\Sigma}^{\theta}_{\Gamma}\}_{\Gamma \in \mathcal{G}_{g,n}}$ associated to the nondegenerate stability condition $\theta$.
Fix a stable graph $\Gamma$ with associated cone $$\sigma_{\Gamma} = (\mathbb{R}_{\geq 0})^{E(\Gamma)} \ \  \text{in}\ \  \Sigma_{\oM_{g,n}}\,. $$  The fan $\widetilde{\Sigma}^{\theta}_ {\Gamma}$ of the subdivision has interior cones\footnote{\label{footnote:interior}An \emph{interior cone} of $\widetilde{\Sigma}^{\theta}_{\Gamma}$ is a cone of the fan intersecting the interior $(\mathbb{R}_{>0})^{E(\Gamma)}$ of  $\sigma_\Gamma$. The interior cones uniquely determine the fan $\widetilde{\Sigma}^{\theta}_{\Gamma}$: the fan is  the collection of these interior cones together with  all of their faces.} in bijective
correspondence with tuples $(\widehat{\Gamma}, D, I)$ where
\begin{itemize}
    \item[(i)] $\widehat{\Gamma}$ is a quasi-stable graph with stabilization $\widehat{\Gamma}^s = \Gamma$, 
    \item[(ii)] $D \in \mathbb{Z}^{V(\widehat{\Gamma})}$ is a $\theta$-stable multidegree, 
    \item[(iii)] $I$ is an acyclic flow on  $\widehat{\Gamma}$ satisfying
    \begin{equation}
        \mathrm{div}(I) = \mdeg_{k,A} - D\,.
    \end{equation}
\end{itemize}

For each tuple $(\widehat{\Gamma}, D, I)$ as above, we will define a cone $\sigma_{\widehat{\Gamma},I} \subseteq \sigma_\Gamma$. The collection of cones \{$\sigma_{\widehat{\Gamma},I}\} $ together with their faces defines the subdivision $\widetilde{\Sigma}^{\theta}_ {\Gamma}$.

Let $\sigma_{\widehat{\Gamma}} = (\mathbb{R}_{\geq 0})^{E(\widehat{\Gamma})}$ be the cone associated to the quasi-stable graph $\widehat{\Gamma}$. There is
a natural map
$${\mathsf{pr}}: \sigma_{\widehat{\Gamma}} \to \sigma_\Gamma\, ,\ \ \  \widehat{\ell} \, \mapsto\,  \ell : e \mapsto 
\begin{cases}
\widehat{\ell}(e_1)+\widehat{\ell}(e_2) &\text{if $\widehat \Gamma$ subdivides edge $e$ into $e_1, e_2$},\\
\widehat{\ell}(e) & \text{otherwise}.
\end{cases}
$$
We define a subcone $\tau_{\widehat{\Gamma},I} \subseteq \sigma_{\widehat \Gamma}$ by the condition
\begin{equation*}
    \tau_{\widehat{\Gamma},I} = \left\{\widehat \ell \in \sigma_{\widehat \Gamma} \ \Big| \  \langle \gamma, I \rangle_{\widehat \ell} = 0 \text{ for all }\gamma \in H_1(\widehat{\Gamma}) \right\}\,,
\end{equation*}
where the pairing is  \eqref{eqn:ellpairing} and $H_1(\widehat{\Gamma})$ is identified with a subgroup of $\Flow(\widehat{\Gamma})$ as explained above.
The following crucial claim is proven in Lemma \ref{Dcomment-I-should-try-to-prove-these-claims} of Section 
\ref{sec:stability_conditions}.

\vspace{8pt}
\noindent{\bf Claim 1.}
{\em The map ${\mathsf{pr}}\colon \sigma_{\widehat{\Gamma}} \to \sigma_\Gamma$ induces an isomorphism from the cone $\tau_{\widehat{\Gamma},I}$ to the
image $$\sigma_{\widehat{\Gamma},I} = {\mathsf{pr}}(\tau_{\widehat{\Gamma},I}) \subseteq \sigma_\Gamma\, .$$
Moreover,  the collection of cones $\{ \sigma_{\widehat{\Gamma},I}\} $ together with their faces forms a fan $\widetilde{\Sigma}^{\theta} _{\Gamma}$ with support $\sigma_\Gamma$.}

\vspace{8pt}

We record an important consequence of Claim 1. Let $\ell_e, \widehat{\ell}_{\widehat e}$ be the coordinate functions on $\sigma_\Gamma~\edits{=(\mathbb{R}_{\geq 0})^{E(\Gamma)}}, \sigma_{\widehat{\Gamma}}~\edits{=(\mathbb{R}_{\geq 0})^{E(\widehat{\Gamma})}}$ giving the lengths of edges $e~\edits{\in E(\Gamma)}$, $\widehat{e}~ \edits{\in E(\widehat \Gamma)}$, viewed now as polynomials on these cones.

\vspace{8pt}
\noindent{\bf Claim 2.} {\em On the cone $\sigma_{\widehat{\Gamma},I}$, there exist unique linear functions $\widehat{\ell}_{\widehat{e}} = \widehat{\ell}_{\widehat{e}}(\ell)$ in the variables $\ell_e$ 
which define a section of the map ${\mathsf{pr}}: \tau_{\widehat{\Gamma},I} \to \sigma_{\widehat{\Gamma},I}$.}

The above definition of the cones $\sigma_{\widehat \Gamma, I}$ is rather formal. Geometric intuition for the construction  can be found in the theory developed in Sections \ref{sec:Logcurves} and \ref{sec:stability_conditions}, see Remark \ref{rem:cone_intuition}. \edits{A detailed example is worked out in Section \ref{eg:sam_big_example}. }

\subsubsection{The formula for \texorpdfstring{$\mathsf{logDR}$}{logDR} } \label{Sect:logDRformulaintro}
Let $\theta$ be a small nondegenerate stability
condition. 
We present here an explicit formula for the cycle 
$$\mathsf{logDR}_{g,A}\in \mathsf{logCH}^g(\oM_{g,n})\, .$$

To write the formula,
we must first define two special strict piecewise power series
$\DRpp, \DRpl$ on ${\widetilde{\Sigma}^{\theta}}$.
The functions $\DRpp, \DRpl$ are uniquely determined
by their restriction to the interior cones of $\widetilde{\Sigma}^{\theta}$. As discussed in  Section
\ref{sasc}, the interior cones correspond to a 
stable graph $\Gamma \in \mathcal{G}_{g,n}$ together with a tuple $(\widehat{\Gamma}, D, I)$. We will define the functions $\DRpp, \DRpl$ on 
the associated cone $\sigma_{\widehat{\Gamma},I} \in \widetilde{\Sigma}^{\theta}_ {\Gamma}$.

\vspace{8pt}

\noindent $\bullet$ 
The definition of $\DRpp$ requires a sum over weightings:
for a positive integer $r$, an \emph{admissible weighting mod $r$} on $\widehat{\Gamma}$ is a flow $w$ with values in $\mathbb{Z}/r\mathbb{Z}$ 
such that \edits{$\mathrm{div}(w)$ is equal modulo $r$ to $D$:}
\[
\mathrm{div}(w) = \edits{(D\text{ reduced mod $r$})} \in (\mathbb{Z}/r\mathbb{Z})^{V(\widehat{\Gamma})}\, .
\]
We define
\begin{equation*}
\mathrm{Cont}_{(\widehat{\Gamma}, D, I)}^r = \sum_{w} 
r^{-h^1(\widehat{\Gamma})}\prod_{e \in E(\widehat{\Gamma})} \exp\left(\frac{\overline{w}(\vec e) \cdot \overline{w}(\cev e)}{2} \widehat{\ell}_e\right) \ \in \mathbb{Q}[[\widehat{\ell}_e : e \in E(\widehat{\Gamma})]]\, ,
\end{equation*}
where the sum runs over admissible weightings $w$ mod $r$. 
Inside the exponential, 
$\overline{w}(\vec e)$ and $\overline{w}(\cev e)$
denote the unique representative of $w(\vec e)  \in \mathbb{Z}/r\mathbb{Z}$
and $w(\cev e)  \in \mathbb{Z}/r\mathbb{Z}$
in $\{0, \ldots, r-1\}$.

As in \cite[Appendix]{JPPZ}, for each fixed degree in the variables $\widehat{\ell}_e$, the element
$\mathrm{Cont}_{(\widehat{\Gamma}, D, I)}^r$ 
is polynomial in $r$ for sufficiently large $r$. We denote by $\mathrm{Cont}_{(\widehat{\Gamma}, D, I)}$ the polynomial in the variables $\widehat{\ell}_e$ obtained by substituting $r=0$ into the polynomial
expression for $\mathrm{Cont}_{(\widehat{\Gamma}, D, I)}^r$.
We define
\begin{equation*}
    \DRpp\big|_{\sigma_{\widehat{\Gamma},I}} = \mathrm{Cont}_{(\widehat{\Gamma}, D, I)} \big|_{\widehat{\ell} = \widehat{\ell}(\ell)} \in \mathbb{Q}[[{\ell}_e : e \in E({\Gamma})]]\,,
\end{equation*}
where we use the variable substitution $\widehat{\ell} = \widehat{\ell}(\ell)$ associated to $\sigma_{\widehat{\Gamma},I}$ from Claim 2. 
We claim that these functions fit together to give a well-defined strict piecewise power series $\DRpp$ on $\widetilde{\Sigma}^{\theta}$.

\vspace{8pt} 
 \noindent $\bullet$ To define $\DRpl$ on $\widetilde{\Sigma}^{\theta}$, we fix a vertex $v_0 \in V(\widehat{\Gamma})$. For every length assignment $\widehat{\ell}$ in the cone $\tau_{\widehat{\Gamma},I}$  and any vertex $v \in V(\widehat{\Gamma})$, let $\gamma_{v_0 \to v}$ be a path from $v_0$ to $v$ in $\widehat{\Gamma}$. We define
\begin{equation} \label{xsse3}
    \alpha(v) = \sum_{\vec e \in \gamma_{v_0 \to v}} I(\vec e) \cdot \widehat{\ell}_e\,,
\end{equation}
where the sum is over the oriented edges $\vec e$ constituting the path $\gamma_{v_0 \to v}$. The defining equations of $\tau_{\widehat{\Gamma},I}$ imply that for $\widehat{\ell} \in \tau_{\widehat{\Gamma},I}$ the expression  \eqref{xsse3}
is independent of the chosen path $\gamma_{v_0 \to v}$. We define
\begin{equation}\label{eq:L_formula}
\DRpl = \sum_{v \in V(\widehat{\Gamma})} (D + \mdeg_{k,A})(v) \cdot \alpha(v) \big|_{\widehat{\ell} = \widehat{\ell}(\ell)} \in \mathbb{Q}[{\ell}_e : e \in E({\Gamma})]\,\, .
\end{equation}
The substitution of variables $\widehat{\ell} = \widehat{\ell}(\ell)$, which give the inverse of the isomorphism $\tau_{\widehat{\Gamma},I} \to \sigma_{\widehat{\Gamma},I}$ and thus have image in $\tau_{\widehat{\Gamma},I}$, ensure that the expression is independent of the choice of the paths $\gamma_{v_0 \to v}$. 
The expression is also independent of the base vertex $v_0$, which follows from the fact that the divisor $D + \mdeg_{k,A}$ has total degree $0$ on $\widehat{\Gamma}$. 

\vspace{8pt}

For the  $\mathsf{logDR}_{g,A}$ formula,
in addition to
$\DRpp$ and $\DRpl$, we will also require
the tautological class
\begin{equation} \label{eqn:etadefinition}
\eta = k^2 \kappa_1 - \sum_{i=1}^n a_i^2 \psi_i \in \mathsf{R}^*(\oM_{g,n})\, .
\end{equation}
Define the mixed degree logarithmic class
\begin{equation} \label{eqn:logDRmixed}
    {\mathbf{P}}_{g,A}^\theta = \exp\left(-\frac{1}{2}(\eta + \Phi(\DRpl)) \right) \cdot \Phi(\DRpp) \in \mathsf{logR}^*(\oM_{g,n})\,,
\end{equation}
where $\Phi$ is the extension of the map \eqref{eqn:Phimapintro} to piecewise power series as described at the end of Section \ref{Sect:logtautclasses}.

\begin{introtheorem}
\label{222ttt} Let $\theta$ be a small nondegenerate  stability condition. 
The log double ramification cycle is  the degree $g$ part
of ${\mathbf{P}}^\theta_{g,A}$,
\begin{equation*} 
    \mathsf{logDR}_{g,A} = {\mathbf{P}}_{g,A}^{g,\theta} \in \mathsf{logR}^g(\oM_{g,n})\, .
\end{equation*}
\end{introtheorem}

As a consequence of Theorem \ref{222ttt}, the class
$${\mathbf{P}}_{g,A}^{g,\theta} \in \mathsf{logR}^g(\oM_{g,n})$$ is
{\em independent} of $\theta$ (for all
small nondegenerate stability conditions
$\theta$). The definition of ${\mathbf{P}}_{g,A}^{g,\theta}$, however,
only requires nondegeneracy of $\theta$. 
The interpretation and wall-crossing
analysis of ${\mathbf{P}}_{g,A}^{g,\theta}$
for large stability conditions $\theta$
is an interesting direction of study. \edits{A separate question is whether a formula 
for $\mathsf{logDR}_{g,A}$ can be written using degenerate stability conditions (such as the trivial one). Our approach requires nondegeneracy (for the existence of
a universal sheaf), and the formula given in Theorem \ref{222ttt} is not well-defined for degenerate $\theta$.}

The calculation of $\mathsf{logDR}_{g,A}$ plays a fundamental role in the log Gromov-Witten theory of toric
varieties: by an elegant argument of Ranganathan and Urundolil Kumaran \cite{RanganathanKumaran}, the log Gromov-Witten theory
of all nonsingular projective
toric varieties (with log structure
determined by the full toric boundary) can be reduced
to $\mathsf{logDR}_{g,A}$. 

After
a discussion of Pixton's formula in the language of
piecewise polynomials in Section \ref{PPsclass}, Theorem \ref{222ttt}
is proven in Section \ref{pprr22}.

\subsection{Example in genus 1}

 We illustrate Theorem \ref{222ttt} for $\mathsf{logDR}_{1,(3,-3)}$ in Figure \ref{fig:logDR133} (giving a similar illustration of the regular DR cycle $\mathsf{DR}_{1,(3,-3)}$ for comparison). The general genus 1 computation is worked out in Section \ref{sec:genus_1}.  A genus 2 calculation using the Sage package \href{https://gitlab.com/jo314schmitt/admcycles/-/tree/LogDR}{\texttt{logtaut}}  
\edits{(part of \texttt{admcycles} \cite{admcycles})}
  is presented in
Section \ref{sagesage}.

\begin{figure}[htb]
\[
\mathsf{DR}_{1,(3,-3)} = \frac{9}{2}(\psi_1 + \psi_2) + \Phi\left(
\begin{tikzpicture}[baseline=0pt]
\filldraw[color=black!10] (0,0)--(2,0)--(2,2)--(0,0);
\filldraw[color=black!25] (0,0)--(2,0)--(2,-2)-- (0,-2)--(0,0);

\draw[thick, dashed] (0,0) -- (2,2);
\draw[thick] (0,0) -- (2,0);
\draw[thick] (0,0) -- (0,-2);

\draw (2,1) node[right] {$-\frac{1}{12}x-\frac{1}{12}y$};
\draw (2,-1) node[right] {$-\frac{1}{12}x$};
\end{tikzpicture}
\right) \in \mathsf{R}^1(\oM_{1,2})\,,
\]
\[
\mathsf{logDR}_{1,(3,-3)} = \frac{9}{2}(\psi_1 + \psi_2) + \Phi\left(
\begin{tikzpicture}[baseline=0pt]
\filldraw[color=black!20] (0,0)--(2,1)--(2,2)--(0,0);
\filldraw[color=black!10] (0,0)--(2,0)--(2,1)--(0,0);
\filldraw[color=black!25] (0,0)--(2,0)--(2,-2)-- (0,-2)--(0,0);

\draw[thick, dashed] (0,0) -- (2,2);
\draw[thick, red] (0,0) -- (2,1);
\draw[thick] (0,0) -- (2,0);
\draw[thick] (0,0) -- (0,-2);

\draw (2,1.5) node[right] {$-\frac{13}{12}x-\frac{13}{12}y$};
\draw (2,0.5) node[right] {$-\frac{1}{12}x-\frac{37}{12}y$};
\draw (2,-1) node[right] {$-\frac{1}{12}x$};
\end{tikzpicture}
\right) \in \mathsf{logR}^1(\oM_{1,2})\,.
\]
    \caption{Formulas for the cycles $\mathsf{DR}_{1,(3,-3)} = \frac{9}{2}(\psi_1+\psi_2)-\frac{1}{12}\delta_0$ and $\mathsf{logDR}_{1,(3,-3)}$, each consisting of a linear combination of classes $\psi_1, \psi_2$ from $\oM_{1,2}$ and a contribution from a piecewise linear function on a subdivision $\widetilde{\Sigma}$ of $\Sigma_{\oM_{1,2}}$ (the trivial subdivision in the case of the regular DR cycle).
    See Figure \ref{fig:M12logblow-up} for an explanation of the coordinates $x,y$ on the cone stack, where we now only draw the coarse space of the stacky cone. }
    \label{fig:logDR133}
\end{figure}

\subsection{Relations in \texorpdfstring{$\mathsf{logR}^*(\oM_{g,n})$}{logR*(Mbar\_gn)} }
The logarithmic double ramification cycle provides
two nontrivial paths to tautological relations in 
$\mathsf{logR}^*(\oM_{g,n})$.
The first path uses the choice of stability condition
in Theorem \ref{222ttt}, and
the second path is via Pixton's relations for the universal double ramification cycle
\cite{BHPSS}.

\label{ppplog}

\begin{introtheorem}
\label{333ttt}
The following two constructions yield
relations in $\mathsf{logR}^*(\oM_{g,n})$:
\begin{enumerate}
    \item [(i)] Let $\theta$ and $\widehat{\theta}$ be two small
    nondegenerate stability conditions of type $(g,n)$. Then,
    $$ {\mathbf{P}}_{g,A}^{g,\theta} =
    {\mathbf{P}}_{g,A}^{g,\widehat{\theta}}\, \in
    \mathsf{logR}^g(\oM_{g,n}) .$$
    \item[(ii)] For every nondegenerate condition
    $\theta$ of type $(g,n)$, 
    $$ {\mathbf{P}}_{g,A}^{h,\theta} = 0 \, \in
    \mathsf{logR}^h(\oM_{g,n}) $$
    for all $h>g$.
\end{enumerate}
\end{introtheorem}

\noindent Theorem \ref{333ttt} is proven together with
Theorem \ref{222ttt}  in Section \ref{pprr22}. 

Higher double ramification cycles are discussed in Section \ref{sagesage1}.
A third
construction of relations in $\mathsf{logR}^*(\Mbar_{g,n})$ 
is presented in Section \ref{sagesage3} via the
{\em double-double ramification cycle}.
The general study of $\mathsf{logR}^*(\oM_{g,n})$ will be taken
up in a future paper.

\edits{
\subsection*{Conventions}
Throughout the paper, we work over the field $\mathbb{C}$ of complex numbers. All Chow groups are taken with $\mathbb{Q}$-coefficients. Furthermore, our monoids and log structures will always be fine and saturated (in the sense of \cite{Kato1989Logarithmic-str,ogus}), unless otherwise mentioned.
}

\subsection*{Acknowledgements}

We are very grateful to
T. Graber, L. Herr, D. Ranganathan, and J. Wise for discussions about
the log Chow ring, log Gromov-Witten theory, and  the
possibility of a formula for the log double ramification cycle.
We have also benefited from related conversations with A. Abreu, Y. Bae, L. Caporaso,
C. Faber, N. Pagani, R. Schwarz, P. Spelier, and D. Zvonkine. 
We thank A. Szenes for organizing the
{\em Helvetic Algebraic Geometry Seminar} in Geneva in August 2021
where the 
results of the paper were 
first presented. \edits{We thank the referees for many improvements.}

D.H. was supported by NWO grants VI.Vidi.193.006 and 613.009.103. 
S.M. was supported by ERC-2017-AdG-786580-MACI.
R.P. was supported by SNF-200020-182181,  
\edits{SNF-200020-219369},
ERC-2017-AdG-786580-MACI, and SwissMAP.  
A.P.  was supported by NSF grant 1807079.
\edits{J.S. was supported by SNF Early Postdoc Mobility grant 184245 
held at 
the Max Planck Institute for Mathematics in Bonn, SNF-184613 
 at the University of Z\"urich, 
{SNF-200020-219369}, and SwissMAP.}
 This project has received funding
from the European Research Council (ERC) under the European Union Horizon 2020 research and innovation program (grant agreement No 786580).

\section{Logarithmic Geometry}\label{logGeom}


\subsection{Overview}
We recall here constructions from log geometry which will be needed later: tropicalization, Artin fans, and subdivisions. \edits{The discussion here is standard, but included for the convenience of the reader. Further details can be found in \cite{ACP}, \cite{ACMUW}, \cite{AbramovichWiseBirational}, \cite{Cavalieri2020-Conestack}, \cite{Olsson2001Log-algebraic-s}.} 
The main
case of interest for us is the moduli space 
$\oM_{g,n}$
 equipped with the divisorial log structure 
obtained from $\partial\oM_{g,n}$.

\subsection{Artin fans \edits{and cone stacks}}
\label{conedefs}

\edits{
\begin{definition}
    A {\em cone} is a pair $(\sigma,N)$ consisting of a lattice $N$ and a full dimensional strongly convex rational polyhedral cone $\sigma \subseteq N_\mathbb{R}$. The intersection $\sigma \cap N$ is an {\em integral structure} on $\sigma$. A {\em morphism} of cones $$(\sigma,N) \to (\sigma',N')$$ is a homomorphism of lattices $N \to N'$ such that the induced map of vector spaces $N_{\mathbb{R}} \to N'_{\mathbb{R}}$ maps $\sigma$ into $\sigma'$. 
 \end{definition}
 
 Let $\textbf{RPC}$ be the category of cones. The category of cones is dual to the category of \emph{fine and saturated sharp monoids}. A cone $(\sigma,N)$ corresponds to the monoid 
\[
\overline{M}_\sigma = \mathsf{Hom}_{\textup{Mon}}(\sigma \cap N, \mathbb{N})\, .
\]
Conversely, a monoid $P$ corresponds to the cone 
\[
\left(\mathsf{Hom}_{\textup{Mon}}(P, \mathbb{R}_{\ge 0})\, ,\  \mathsf{Hom}_{\textup{Gp}}(P^\gp,\mathbb{Z})\right). 
\]
}
\edits{A \emph{face morphism} $(\sigma,N) \to (\sigma',N')$ is a map of cones that induces an isomorphism of $\sigma$ with a (non-necessarily proper) face of $\sigma'$. We write $\textbf{RPC}^f$ for the category of cones with face morphisms between them. Following \cite{Cavalieri2020-Conestack}, a {\em cone stack} is a category fibered in groupoids over $\textbf{RPC}^f$.}

 \edits{An \emph{Artin cone} is an algebraic stack of the form
\[
\mathcal{A}_P= [\on{Spec}\mathbb{C}[P]/\on{Spec}\mathbb{C}[P^\gp]]
\]
for a fine and saturated sharp monoid $P$. Artin cones carry natural logarithmic structures, obtained from the toric logarithmic structure on $\on{Spec}\mathbb{C}[P]$ by descent. }

\edits{A map of logarithmic stacks is {\em strict} if the induced map on log structures is an isomorphism.}
\edits{\begin{definition}[\cite{Cavalieri2020-Conestack}]
    An Artin fan $\mathcal{A}$ is an algebraic stack with a logarithmic structure that has a strict \'etale cover by a disjoint union of Artin cones. 
\end{definition}
}

\edits{A cone $(\sigma,N)$ determines an Artin cone by 
\[
\mathcal{A}_{(\sigma,N)} = \mathcal{A}_{\overline{M}_\sigma} \, .
\]
Moreover, we have \cite{AbramovichWiseBirational,Olsson2001Log-algebraic-s}
\[
\mathsf{Hom}_{\textbf{RPC}}((\sigma,N),(\sigma',N')) = \mathsf{Hom}_{\textup{LogStacks}}(\mathcal{A}_{(\sigma,N)},\mathcal{A}_{(\sigma',N')})\, .
\]
In other words, the category of cones and the category of Artin cones are equivalent. In fact, this equivalence of categories extends to an equivalence of 2-categories between cone stacks and Artin fans \cite[Theorem 3]{Cavalieri2020-Conestack}. For a cone stack $\Sigma$, we will write $\mathcal{A}_\Sigma$ for the corresponding Artin fan.}

\begin{convention}
It is cumbersome to carry both the cones and the integral structures around everywhere in the notation. Thus, we will simply denote a cone by $\sigma$, which should be always understood as coming with an implicit integral structure. 
\end{convention}

\subsection{Tropicalization}
\label{subsection:Tropicalization}

Let $X = (X, \M_X)$ be a log smooth algebraic stack with a logarithmic structure. 
To $X$, we can assign a combinatorial shadow of the \emph{characteristic monoid} $$\ghost_{X} = \M_X/\mathcal{O}_X^*\, .$$ 

\edits{
\begin{definition}
    A {\em tropicalization} of $X$ is an Artin fan $\mathcal{A}_X$ with a strict map $$t_X: X \to \mathcal{A}_X$$ which has connected nonempty fibers. 
\end{definition}
We will write $\Sigma_X$ for the cone stack corresponding to $\mathcal{A}_X$ and refer to $\Sigma_X$  as a tropicalization for $X$ as well. Tropicalizations 
are not unique.}




\edits{A} tropicalization is easy to \edits{construct} when $X$ is ``sufficiently local". Local here 
is with respect to the log structure. The most local case as far as the characteristic monoid $\ghost_X$ is concerned is when $X$ is an \emph{atomic} log scheme: $X$ has a unique closed stratum, and for every point $x$ in the closed stratum, 
$$
\Gamma(X,\ghost_X) = \ghost_{X,x}\, . 
$$ 
Atomic neighborhoods exist on a log smooth algebraic stack with a log structure smooth-locally. 
In the atomic case, $\Sigma_X$ is simply the rational polyhedral cone 
$$
\sigma_{X,x} = (\mathrm{Hom}(\ghost_{X,x},\mathbb{R}_{\ge 0}),\mathrm{Hom}(\ghost_{X,x},\mathbb{N}))\,
$$
\edits{for $x$ a point in the closed stratum. A tropicalization map is constructed by taking a local chart $U_y \to \on{Spec}\mathbb{C}[\overline{M}_{X,y}]$, and composing the natural map 
\[
U_y \to \mathcal{A}_{\ghost_{X,y}} 
\]
with the map $\mathcal{A}_{\ghost_{X,y}} \to \mathcal{A}_{\ghost_{X,x}}$. The assumptions imply that the maps $U_y \to \mathcal{A}_{\ghost_{X,x}}$ descend to $X$ and have connected nonempty fibers.}  

In general, we \edits{can construct a tropicalization by presenting} $X$ as a \edits{(2-categorical)} colimit 
\begin{equation} \label{eqn:Xicolimit}
\varinjlim X_i \cong X
\end{equation}
with each $X_i$ atomic, and we define 
\begin{equation}\label{ccoolimit}
\Sigma_X = \varinjlim \Sigma_{X_i}\, . 
\end{equation}
For a map $f: X_i \to X_j$ in the colimit \eqref{eqn:Xicolimit}, the associated map $\Sigma_{X_i} \to \Sigma_{X_j}$ is induced by the map
\[
\ghost_{X_j, x_j} \cong \Gamma(X_j, \ghost_{X_j}) \xrightarrow{f^*} \Gamma(X_i, \ghost_{X_i}) \cong \ghost_{X_i, x_i}
\]
of monoids. \edits{Equivalently, 
\[
\mathcal{A}_X = \varinjlim \mathcal{A}_{X_i}. 
\]
The corresponding tropicalization map $X \to \mathcal{A}_X$ is obtained by observing that the composed maps 
\[
X_i \to \mathcal{A}_{X_i} \to \mathcal{A}_X
\]
descend to $X$.}
The colimit  \eqref{ccoolimit} is taken in the 2-category of stacks over rational polyhedral cones, \edits{and the corresponding colimit of Artin fans is taken in the 2-category of stacks with logarithmic structure}. 

\begin{remark} 
\edits{The above construction in general depends on the choice of atomic cover of $X$ and does not yield a canonical tropicalization. A canonical 
tropicalization 
$X\to \mathcal{A}_X$
is constructed in
\cite{ACMw17,AbramovichWiseBirational}. The construction of the canonical Artin fan is however \emph{not} functorial: a non-strict map $X \to Y$ does not in general induce a morphism $\mathcal{A}_X \to \mathcal{A}_Y$ between the canonical Artin fans of $X$ and $Y$. However, given a tropicalization $Y \to \mathcal{A}_Y$, the substitute for functoriality of \cite{AbramovichWiseBirational} says that we can find some tropicalization $X \to \mathcal{A}_X$ such that the diagram} 
\[
\begin{tikzcd}
    X \ar[r] \ar[d] & \mathcal{A}_X \ar[d] \\ 
    Y \ar[r] & \mathcal{A}_Y
\end{tikzcd}
\]
\edits{commutes.  From
the point of view of functoriality, working with the flexibility of an arbitrary tropicalization is more natural than just studying the canonical tropicalization.}
\end{remark}

\begin{remark}
    \edits{Suppose $X$ is a log smooth algebraic stack. Then every tropicalization $\Sigma_X$ of $X$ has the same cones as the canonical tropicalization, but with potentially different automorphism groups. In fact, both the cones of $\Sigma_X$ and the cones of the canonical tropicalization are in bijection with the strata of $X$.}
\end{remark}


To understand the above constructions, there are two examples to keep in mind. The first is when $X$ is a toric variety. Then, the canonical tropicalization $\Sigma_X$ is the fan of $X$ (though $\Sigma_X$ is not embedded in the cocharacter lattice of the torus\footnote{For example, if the toric variety has torus factors, the cocharacter lattice cannot be recovered from $\Sigma_X$. }). 

The second example is when $X$ is the log scheme associated to a nonsingular variety $V$ with a normal crossings divisor $D$. When $D$ is normal crossings in the Zariski topology, \edits{the canonical tropicalization} $\Sigma_X$ is the cone over the dual intersection complex of the irreducible components $D_i$ of $D$. \edits{The construction is explained in \cite{KKMSD}.} Concretely, each stratum of $X$, which corresponds to a connected component of an intersection 
$$
D_{i_1} \cap D_{i_2} \cdots \cap D_{i_n}
$$
of irreducible components, contributes the cone
$$
(\mathbb{R}_{\ge 0}^n, \mathbb{Z}^n)
$$
in $\Sigma_X$. These rational polyhedral cones
are glued together in the obvious way: a cone becomes a face of every
cone corresponding to a deeper stratum (further nonempty intersections of the $D_i$).  
In the Zariski normal crossings case, $\Sigma_X$ does not have nontrivial stack structure.

In general, $D$ may only be normal crossings \'etale locally: some $D_i$ may self-intersect. Then, for integers $k>1$, there can be strata which \'etale locally are the intersections of $k$ distinct divisors 
$D_{i,1},\ldots, D_{i,k}$
which are globally the $k$ branches of a single divisor $D_i$ coming together. Thus, intersections of the form 
\begin{equation} \label{ccvvb}
\bigcap_{j=1}^n\, 
D_{i_j,1} \cap \cdots \cap  D_{i_j,k_j} 
\end{equation}
 can also be strata of $X$. 


\edits{\begin{definition} \cite[\S 6.2]{ACP}
The {\em monodromy group} $G$ 
of a stratum of $X$
is the 
image of the fundamental group of the stratum in the
permutation group 
of the branches of the divisors cutting out the stratum.
\label{monodef}
\end{definition}}

For the stratum \eqref{ccvvb},
the monodromy group $G$ acts on 
$$
\mathbb{R}_{\ge 0}^{k_1 + \cdots + k_n}\, ,
$$ 
\edits{The interior -- that is, the complement of its proper faces -- of the corresponding stacky cone in $\Sigma_X$ coincides with the stack quotient
\[
[\mathbb{R}_{> 0}^{k_1 + \cdots +k_n}/G]. 
\]
The details of the construction are explained in \cite[\S 6.2]{ACP}.}
\edits{\begin{example} \label{Exa:Mbar_monodromy_first}
    Let $X=(\oM_{g,n},\partial\oM_{g,n})$. Then the cone stack 
    \[
    \Sigma_{\oM_{g,n}} = \varinjlim_{\Gamma \in \mathcal{G}_{g,n}} \sigma_\Gamma
    \]
    in equation \eqref{styfandef} of the Introduction is a tropicalization for $X$. A tropicalization map 
    \[
    \oM_{g,n} \to \varinjlim_{\Gamma \in \mathcal{G}_{g,n}} \mathcal{A}_{\sigma_\Gamma}
    \]
    is constructed in \cite[Theorem 4]{Cavalieri2020-Conestack}. The above tropicalization does \emph{not} agree with the canonical Artin fan of \cite{ACMw17, AbramovichWiseBirational}. 
    
    By the definition of $\Sigma_{\oM_{g,n}}$, the automorphism group of the stacky cone $\sigma_\Gamma$ is given by the group $\Aut_\Gamma$ of automorphisms of the graph $\Gamma$ itself. However, $\Aut_\Gamma$ is different from the monodromy group $G_\Gamma$ of the stratum $\mathcal{M}^\Gamma \subseteq \oM_{g,n}$ associated to $\Gamma$. Indeed, the branches of the boundary divisor of $\oM_{g,n}$ cutting out $\mathcal{M}^\Gamma$ precisely correspond to the edges $E(\Gamma)$. The monodromy group $G_\Gamma \subseteq \mathrm{Sym}(E(\Gamma))$ of the stratum $\mathcal{M}^\Gamma$ is the image of $\Aut_\Gamma \to \mathrm{Sym}(E(\Gamma))$ and
     fits into an exact sequence
    \[
    \{\mathrm{id}_\Gamma\} \to \on{Aut}_\Gamma^\textup{loop} \to \on{Aut}_\Gamma \to G_\Gamma \to \{[\mathrm{id}_{E(\Gamma)}]\}\,,
    \]
    where $\Aut_\Gamma^\textup{loop} \subseteq \Aut_\Gamma$ is the subgroup of automorphisms of $\Gamma$ acting as the identity on $E(\Gamma)$. The group of such automorphisms is the free abelian group generated by the involutions which interchange two half-edges forming a loop.
    
    Since the automorphism groups $\Aut_\Gamma$ and $G_\Gamma$ of the stacky cones in $\Sigma_{\oM_{g,n}}$ and the cone stack associated to the canonical tropicalization of $\oM_{g,n}$ are distinct in general, we see that these two tropicalizations do not agree.
\end{example}}


\subsection{Subdivisions}\label{sec:subdivisions}

\edits{Let $X$ be a logarithmic algebraic stack. A choice of strict map from $X$ to an Artin fan $\mathcal{A}_X$} provides access to one of the most important operations in the subject.

\begin{definition}
A \emph{subdivision} \edits{(sometimes called a \emph{proper subdivision} in the literature)} of a rational polyhedral cone $\sigma$ is a fan $F \subseteq \sigma$ whose support $\mathsf{Supp}(F)$ is equal to that of $\sigma$. 
\end{definition}

\begin{definition}
A \emph{subdivision} of a cone stack $\Sigma_X$ is a map $\widetilde{\Sigma}_X \to \Sigma_X$ from a cone stack $\widetilde{\Sigma}_X$ such that for every map $\sigma \to \Sigma_X$ from a rational polyhedral cone $\sigma$, the fiber product $\widetilde{\Sigma}_X \times_{\Sigma_X} \sigma \to \sigma$ is a subdivision. 
\end{definition}

\begin{remark}
Alternatively, a subdivision of a cone stack $\Sigma_X$ can be defined
by descent. The stack $\Sigma_X$ \edits{is a combinatorial cone stack (in the sense of \cite{Cavalieri2020-Conestack}): $\Sigma_X$ has a presentation as a colimit
$$
\Sigma_{X} = \varinjlim_{i \in \mathcal{C}} \Sigma_{X_i}
$$
of a diagram of rationally polyhedral cones $\Sigma_{X_i}$, where all maps $\Sigma_{X_i} \to \Sigma_{X_j}$ in the colimit are isomorphisms to a (not necessarily proper) face}. Every map from a cone $\sigma \to \Sigma_X$ necessarily factors through one of the $\Sigma_{X_i}$, and so subdivisions of $\Sigma_X$ are the same as subdivisions of all $\Sigma_{X_i}$ which respect all maps in the diagram. In particular, subdivisions respect identifications of faces and automorphisms in $\Sigma_X$. 
\end{remark}

A subdivision $\widetilde{\Sigma}_{X} \to \Sigma_X$ can be pulled back to $X$ via the map from $X$ to the Artin fan $\ca{A}_X$. Under the equivalence of categories between cone stacks and Artin fans,
the map $\widetilde{\Sigma}_X \to \Sigma_X$ induces a proper, representable, and birational map of Artin fans $\widetilde{\mathcal{A}}_X \to \mathcal{A}_X$ which we can pull back to $X$:
$$
\widetilde{X}  = X \times_{\mathcal{A}_X} \widetilde{\mathcal{A}}_X \to X\, .
$$    
The map $\widetilde{X}  \to X$  is proper and representable. When $X$ is log smooth, the map $X \to \mathcal{A}_X$ is smooth, and thus $\widetilde{X} \to X$ is also birational.  

\begin{definition}
A map $\widetilde{X} \to X$ of the form $X \times_{\mathcal{A}_X} \widetilde{\mathcal{A}}_X \to X$ is a {\em subdivision} or
a
{\em log modification} of $X$. 
\end{definition}

 The functor of points of the
log modification is characterized by the following result, which is essentially standard \cite{Kato1994Toric-singulari,NakayamaLogetale}. As we could not find a precise reference covering the required level of generality, we include a proof for completeness. 

\begin{lemma}
\label{lem:subfunctorofpoints}
Let $p:\widetilde{X} \to X$ be a log modification. Then, the following two properties hold:
\begin{enumerate}
\item[(i)] $p$ is a representable monomorphism of logarithmic algebraic stacks, 
\item[(ii)] a log map $S \to X$ lifts to $\widetilde{X}$ if and only if, \'etale locally on $S$, the map\footnote{The map $\Sigma_S \to \Sigma_X$ may 
only exist \'etale locally on $S$. } $\Sigma_S \to \Sigma_X$ lifts through $\widetilde{\Sigma}_X$.  
\end{enumerate}
\end{lemma}

\begin{proof}
Since $\widetilde{X} = X \times_{\ca{A}_X} \widetilde{\ca{A}}_X$, it suffices to prove the properties for 
$$q: \widetilde{\ca{A}}_{X} \to \ca{A}_X\,$$
instead of $p$. \edits{We fix a map 
$f:S \to \ca{A}_X$. 

For property (i), we must check that any two lifts of $f$ to maps $S \to \widetilde{\ca{A}}_{X}$ are equal. 
We can work \'etale locally on $S$.} In particular, we may assume throughout that $S$ and $X$ are atomic, with global charts given by sharp fine saturated monoids $Q$ and $P$ respectively. Let
\begin{align*}
\sigma &= \mathsf{Hom}(P,\mathbb{R}_{\ge 0})\, ,\\
\tau  &= \mathsf{Hom}(Q,\mathbb{R}_{\ge 0})\, ,
\end{align*}
and let $N = \mathsf{Hom}(P,\mathbb{Z})$, $L = \mathsf{Hom}(Q,\mathbb{Z})$. 
Then, $\Sigma_S$ is the single cone $\tau$ in the lattice $L$, and $\Sigma_X$ 
is the single cone $\sigma$ in the lattice $N$. 
The toric variety associated to $\sigma$ in $N$ is
$\Spec k[P] = V(\sigma,N)$,
the Artin fan is
$$\mathcal{A}_X =[\Spec k[P]/\Spec k[P^\textup{gp}]] = [V(\sigma,N)/T]\, ,$$
the global quotient of $V(\sigma,N)$ by its dense torus, and 
$$\widetilde{\mathcal{A}}_{X} = [V(\widetilde{\sigma}_X,N)/T]$$ 
is the quotient of the toric variety defined by the fan $\widetilde{\Sigma}_X$ in $N$ by its dense torus. 
Logarithmic maps $S \to \mathcal{A}_X$ are then very simple, with  
$$
\mathsf{Hom}(S,\mathcal{A}_X) = \mathsf{Hom}_{\textup{Mon}}(P,Q) = \mathsf{Hom}_{\textup{Cones}}(\Sigma_S,\Sigma_X)\, . 
$$
Furthermore, $\widetilde{\mathcal{A}}_X$ has an open cover by the global quotients $$\widetilde{\mathcal{A}}_{i} = [V(\sigma_i,N)/T] \, ,$$ where the $\sigma_i$ are the cones of $\widetilde{\Sigma}_X$. If $\sigma_i < \sigma_j$ then $\widetilde{\mathcal{A}}_i \subset \widetilde{\mathcal{A}}_j$. 

Suppose now that we are given two maps $g,h: S \to \widetilde{\mathcal{A}}_X$ lifting $f$. As we are working locally on $S$, we can assume that $g$ factors through $\widetilde{\mathcal{A}}_i$ and $h$ through $\widetilde{\mathcal{A}}_j$ for two cones $\sigma_i,\sigma_j$. 

\edits{
Since 
$$
\mathsf{Hom}(S,\widetilde{\mathcal{A}}_i) = \mathsf{Hom}(\Sigma_S,\sigma_i)
$$  
and similarly for $\widetilde{\mathcal{A}}_j$, the only way $$q \circ g = q \circ h$$ can agree
is if the corresponding maps $\Sigma_S \to \sigma_i \to \sigma$ and $\Sigma_S \to \sigma_j \to \sigma$ agree. This can only happen if the maps factor through $\sigma_i \cap \sigma_j$, which is a common face $\sigma_k$ of both. But since $\sigma_k \to \sigma$ is a monomorphism, we must have $g = h$. Thus $q$ is a monomorphism.  
}


Once property (i) is established, to prove property (ii), we need only study the lifting of $f$ \'etale locally on $S$.
If $$g:S \to \widetilde{\mathcal{A}}_X$$ is a lift $f$, then $g$ will factor through one of the $\widetilde{\mathcal{A}}_i$, and so the map $\Sigma_S \to \Sigma_X$ will factor through $\widetilde{\Sigma}_{X}$. Conversely, suppose $\Sigma_S \to \Sigma_X$ factors through $\widetilde{\Sigma}_X$. The factorization determines maps $S \to \widetilde{\mathcal{A}}_i \to \widetilde{\mathcal{A}}_X$. 
\end{proof}

\begin{remark}
When $X = (\Mbar_{g,n},\partial \Mbar_{g,n})$, the notions discussed here precisely coincide with the discussion of Section \ref{aaa111} of the Introduction. In particular, the definition of log Chow rings there also generalizes in a straightforward way to an arbitrary log smooth $X$:
$$
\mathsf{logCH}(X) = \varinjlim_{X' \to X}\mathsf{CH}^*_{\textup{op}}(X')\, . 
$$
where $X' \to X$ ranges over all log modifications of $X$. 
\end{remark}

\subsection{Piecewise linear functions}\label{sec:PL}

\begin{definition}
A {\em \edits{strict} piecewise linear function} on a logarithmic algebraic stack $X$ is a global section 
$$\alpha \in H^0(X, \ghost_X^\gp) \edits{\otimes_{\mathbb{Z}} \mathbb{Q}}\, .$$
\end{definition}
\edits{For log smooth $X$}, there is a natural interpretation of $\alpha$ as a function on \edits{any tropicalization} $\Sigma_X$: for $x \in X$, restriction gives $\alpha_x \in \ghost_{X,x}^\gp \edits{\otimes_{\mathbb{Z}} \mathbb{Q}}$ which induces an evaluation map
\[
\mathsf{Hom}(\ghost_{X,x}, \mathbb{R}_{\geq 0}) \to \mathbb{R}\, , \ \ \varphi \mapsto \varphi(\alpha_x)\,,
\]
where $\varphi$ is extended in the obvious way from $\ghost_{X,x}$ to $\ghost_{X,x}^\gp$. Recall from subsection \ref{subsection:Tropicalization} that the cone $\mathsf{Hom}(\ghost_{X,x}, \mathbb{R}_{\geq 0})$ \edits{is a strict} cover of the cone $\sigma_x$ in $\Sigma_X$. \edits{Since the $\alpha_x$ come from a global section $\alpha$,} the above functions descend to a function on $\Sigma_X$. \edits{This function is strict piecewise linear, in the sense that it is continuous and linear on each cone of $\Sigma_X$. Furthermore, it takes rational values on the integral structure of $\Sigma_X$. We will refer to such a function as a strict piecewise linear function on $\Sigma_X$, and we will write $\mathsf{sPL}$ for the group of strict piecewise linear functions on a cone stack.} 

The construction above induces an isomorphism between the group of (combinatorial) \edits{strict} PL functions on $\Sigma_X$ and the group
$$H^0(X, \ghost_X^\gp)\otimes_{\bb Z} \bb Q\, .$$
\edits{of strict piecewise linear functions on $X$. In particular, the group $\mathsf{sPL}(\Sigma_X)$ does not depend on the choice of tropicalization of $X$.}

\begin{remark}\label{rem:compare_PP_defs}
In \cite{HolmesSchwarz}, the ring of strict piecewise polynomials on $X$ was defined as the ring of global sections of the sheaf $$\on{Sym}^\bullet_\bb Q \ghost_X^\gp\, . $$ In particular, the degree 1 part is given by 
$$H^0(X, \ghost_X^\gp)\otimes_{\bb Z} \bb Q\, ,$$
the same as the group of strict piecewise linear functions defined above. Locally on $X$ the ring $\on{Sym}^\bullet_\bb Q \ghost_X^\gp$ is just the ring of polynomials in linear functions, and the same holds for the combinatorially-defined ring of piecewise polynomial functions, so   the definition of \cite{HolmesSchwarz} coincides with that of the present paper in all degrees.  
\end{remark}

If $\alpha \in H^0(X, \ghost_X^\gp)$, we define the $\ca O_X^\times$-torsor $\ca O_X^\times(\alpha)$ to be the preimage of $\alpha$ under the natural map $\M_X^\gp \to \ghost_X^\gp$. The preimage is an $\ca O_X^\times$-torsor because of the exact sequence 
$$1 \to \ca O_X^\times \to \M_X^\gp \to \ghost_X^\gp \to 0\, . $$
\begin{definition}
\label{def: associatedbundle}
Let $\alpha \in H^0(X, \ghost_X^\gp)$. We define $\ca O_X(\alpha)$ to be the line bundle obtained from $\ca O_X^\times(\alpha)$ by glueing in the \emph{infinity} section.
\end{definition} 

If $\alpha \le \alpha'$, in the sense that $$\alpha' - \alpha \in H^0(X,\ghost_X) \subset H^0(X,\ghost_X^\gp)\, ,$$ then we have a natural map of line bundles
$$\ca O_X(\alpha) \to \ca O_X(\alpha')\, . $$
 More generally, if $\ca L$ is a line bundle on $X$, we define $\ca L(\alpha) = \ca L \otimes_{\ca O_X} \ca O_X(\alpha)$.



\begin{remark}
The induced homomorphism $$H^0(X, \ghost_X^\gp) \edits{\otimes_{\mathbb{Z}} \mathbb{Q}} \to \rPic(X) \edits{\otimes_{\mathbb{Z}} \mathbb{Q}} \to \mathsf{CH}^1(X)$$ coincides with the degree $1$ part of the homomorphism 
$$
\Phi\colon \mathsf{sPP}(\Sigma_X) \to \mathsf{CH}(X)
$$
from \eqref{eq:Phi_def}. 
\end{remark}

\section{Logarithmic Curves}
\label{sec:Logcurves}

\subsection{Definitions}
Up to now, we have been discussing  {\em absolute}  logarithmic geometry. We now turn our attention to the simplest relative situation: \emph{logarithmic curves}. 

\begin{definition} \label{logcur}
Let $S$ be a log scheme. A \emph{logarithmic curve} $C \to S$ over $S$ is a log smooth, proper, connected, integral, and vertical morphism with reduced geometric fibers which are of pure dimension $1$. 
\end{definition}

Briefly, the adjective {\em integral} concerns   the flatness of $C \to S$, and the adjective {\em vertical} means that we do not put additional log structure on the marked points of the family.

\begin{remark}
\edits{A warning to the reader about the notation: given a logarithmic curve $\pi: C \to S$ with sections $x_1, \ldots, x_n : S \to C$, later constructions (such as Definition \ref{def:Mgntheta_detailed}) will use the notation
\[
\omega^\log = \omega_\pi(\sum_{i=1}^n x_i)
\]
for the relative log canonical sheaf, which is \emph{not} the same as the sheaf of log differentials of $C/S$. Indeed, since we insist on the log-structure being vertical, the sheaf of log differentials is  $\omega_\pi$.}
\end{remark} 
We do not formally review the relevant terminology, and refer the reader to \cite{Kato1996Log-smooth-defo} for a discussion. We only will use the most basic consequences of 
Definition \ref{logcur}: the map of underlying schemes of every log curve is a prestable curve, and the log structures on $C$ and $S$ give us access to a family of tropical curves. 

More precisely, over every geometric point $s \in S$, the characteristic monoid $\ghost_{C_s}$ of the fiber $C_s$ satisfies:
\begin{itemize}
\item at the generic point $c$ of an irreducible component, $\ghost_{C_s,c} = \ghost_{S,s}$,
\item at a node $q$, there exists a unique (nonzero) element $\ell_q\in \ghost_{S, s}$ such that $$\ghost_{C_s,q} = \ghost_{S,s} \oplus_\mathbb{N} \mathbb{N}^2\, ,$$ where the homomorphism $\bb N \to \bb N^2$ is the diagonal, and the homomorphism $\mathbb{N} \to \ghost_{S,s}$ sends $1$ to $\ell_q$. 
\end{itemize}
\edits{Following \cite{Cavalieri2020-Conestack},} we encode the above data as a \emph{tropical curve over $S$}. 

\begin{definition}\label{tropcurdef}
Let $M$ be a sharp monoid. A \emph{tropical curve} metrized by $M$ is the data of a connected graph $\Gamma$ together with
elements $\ell(e) \in M-\{0\}$ for each edge of $e$. We call $\ell(e)$ the \emph{length} of $e$. 
\end{definition}

For every log curve $C/S$ and  geometric point $s \in S$,
we obtain a \emph{tropical curve metrized by $\ghost_{S,s}$} consisting of the dual graph $\Gamma_s$ of $C_s$  with each edge $e$ of $\Gamma_s$ corresponding to a node $q \in C_s$ assigned the {\em length} 
$$\ell(e) = \ell_q \in \ghost_{S,s}\, .$$
Since the above data comes from a log map $C \to S$, it is severely constrained. First of all, the tropical curve $(\Gamma_s,\ell(e))$ is \'etale locally constant along strata of $S$. 
Furthermore, for each \'etale specialization $\zeta: t \rightsquigarrow s$, there is an induced map $$p_{\zeta}^\sharp\colon \ghost_{S,s} \to \ghost_{S,t}\, $$ and an induced {\em edge contraction} map 
$$
f_{\zeta}:\Gamma_s \to \Gamma_t
$$
which is compatible with $p_{\zeta}^\sharp$.
If an edge $e$ of $\Gamma_s$ has length mapping to $0$ then the edge is contracted, otherwise it is sent to an edge with length $p_{s,t}^\sharp(\ell(e))$.

\begin{remark} \label{rmk:conesinterpretation}
\edits{When $S$ is log smooth}, the connection of Definition \ref{tropcurdef} with the constructions in  Section \ref{subsection:Tropicalization} is as follows. The strata of $S$ are in bijection with the cones of \edits{any tropicalization} $\Sigma_S$. For a geometric point $s \in S$ in a given stratum $O_{s}$, the corresponding stacky cone $\sigma_s$ in $\Sigma_S$ has a \edits{strict \'etale} cover by the rational polyhedral cone $\mathsf{Hom}(\ghost_{S,s}, \mathbb R_{\ge 0})$. The \'etale specializations $\zeta\colon t \rightsquigarrow s$ correspond to gluing relations and automorphisms that produce $\sigma_s$ out of the cover $\mathsf{Hom}(\ghost_{S,s}, \mathbb R_{\ge 0})$. 

The graphs $\Gamma$ also appear naturally in two ways: as the dual graphs of the various fibers of  $C \to S$ and as the fibers 
of the induced map of cone stacks $\Sigma_C \to \Sigma_S$.
The latter fibers have the same combinatorial type in the interior of each cone of $\Sigma_S$, but their edge lengths vary. The parameter $\ell(e) \in \ghost_{S,s}$ associated to an edge of a dual graph induces by evaluation a homomorphism $$\ell(e)\colon  \mathsf{Hom}(\ghost_{S,s}, \mathbb R_{\ge 0}) \to \mathbb{R}_{\ge 0}\, .$$ 
The compatibility of the $\ell(e)$ with \'etale specializations means that these homomorphisms descend to $\Sigma_S$, and under this interpretation, they measure the length of the edge in the fibers of $\Sigma_C \to \Sigma_S$ as $x$ moves in $\Sigma_S$.
\end{remark}

\subsection{Subdivisions of log curves}
\label{subsec:subsoflogcurves}
Let $C \to S$ be a logarithmic curve. As $C$ is a logarithmic scheme, we may perform subdivisions on $C$, but the composition of a subdivision $C' \to C$ with $C \to S$ may no longer be a logarithmic curve (as $C' \to S$ may fail to be integral or have nonreduced fibers\footnote{In good cases, when $S$ is log smooth, $C' \to C$ is essentially a blow-up, and integrality is equivalent to flatness. Then, the pathology can be rephrased as {\em even though $C \to S$ is flat, the blow-up $C'$ may fail to be flat over $S$}. Moreover, even if $C'\to S$
is flat, the fibers need not be reduced.}). 

The subdivisions $C' \to C$ which are logarithmic curves over $S$ are special. In terms of the tropicalizations $(\Gamma_s, \ell(e) \in \ghost_{S,s}, f_\zeta)$,
the subdivisons $C' \to C$ which are logarithmic curves over $S$ are exactly the subdivisions of the following form:
\begin{itemize}
\item (Subdivision of each fiber) For each geometric point $s \in S$, a subdivision of each edge $e_s$ of $\Gamma_s$ into a union of edges $e(i)_s$ of length $\ell(e(i)_s) \in \ghost_{S,s}-\{0\}$, with $$\ell(e_s) = \sum \ell(e(i)_s)\, .$$ 
\item (Compatibility with \edits{specialization}) For each \'etale specialization $\zeta: t \rightsquigarrow s$ and edge $e_t$ mapping to $e_s$ under $E(\Gamma_{t}) \to E(\Gamma_s)$, a bijection $\phi$ between the edges $e(i)_s$ for which $\ell(e(i)_s)$ does not map to $0$ under the map $p^\sharp_{\zeta}: \ghost_{S,s} \to \ghost_{S,t}$ and the edges $e(j)_t$ subdividing $e_t$, so that $p^\sharp_{\zeta}(\ell(e(i)_s))=\ell(\phi(e(i)_s)) \in \ghost_{S,t}$.
\end{itemize} 
Even though $C' \to S$ is not a log curve for every subdivision $C' \to C$, we have: 
\begin{itemize}
\item[(i)] $C \times_S S' \to S'$ is a log curve for every subdivision $S' \to S$ (as the
definition of a log curve is stable under base change with respect to any log map $S' \to S$).  
\item[(ii)] Subdivisions by log curves are in a sense cofinal among subdivisions of the total space $C$: for any subdivision $C' \to C$, there exists a log alteration (a log alteration is a composition of a log modification with a root stack along the log structure) $S' \to S$ and a further subdivision $C'' \to C' \times_S S'$ such that $C'' \to S'$ is a log curve. 
\end{itemize}
We will not use  (ii)  in the paper. \edits{The result follows by the formulation of the weak semistable reduction theorem in \cite{Molchosemistable}}.

A special class of subdivisions of log curves $C \to S$ will be central in what follows. 

\begin{definition}[Quasi-stable models]
\label{def: qstablemodel}
A quasi-stable model of a log curve $C \to S$
is a subdivision $\widehat{C} \to C$ such that
$\widehat{C} \to S$ is a log curve and, fiberwise, the exceptional locus of $\widehat{C} \to C$ consists of isolated rational curves (every connected component of the exceptional locus is a nonsingular rational curve). 
\end{definition}


In particular, on the level of the dual graphs $\Gamma_s$, a quasi-stable model has at most one exceptional vertex on each edge. We have chosen the terminology quasi-stable model because, when the underlying family $C \to S$ is a family of stable curves, the curve $\widehat{C} \to S$ is called quasi-stable in the literature.   

\subsection{Abel-Jacobi theory on a tropical curve}
\label{sec: TropicalAJ}



The main notions of classical Abel-Jacobi theory have tropical analogues\edits{, originating in \cite{baker2011metric,baker2007riemann,mikhalkin2008tropical}. We present here a logarithmic approach to the part of the theory that we need for the present work. }

Fix a tropical curve $\Gamma$ metrized by a sharp monoid $M$. \edits{We briefly recall the definitions of divisors and flows on $\Gamma$ from Section \ref{sasc}. A \emph{divisor} on a tropical curve $\Gamma$ is an element of $\bb Z^{V(\Gamma)}$, while a \emph{flow} on $\Gamma$ is a function $f \colon \vec E(\Gamma) \to \mathbb{Z}$ satisfying $f(\cev e) = - f(\vec e)$ for all edges $e \in E(\Gamma)$. As written in equation} \eqref{eq:div_of_flow}, a flow on $\Gamma$ induces a divisor by taking outgoing slopes,
$$
\mathrm{div}\colon \mathsf{Flow}(\Gamma) \to \bb Z^{V(\Gamma)}\, . 
$$





Flows and divisors are essentially combinatorial notions: they depend only on the graph $\Gamma$ and not the metric structure of $\Gamma$. On the other hand, the analogue of a rational function on a curve is a \emph{piecewise linear} function, which requires the metric structure.

\begin{definition}\label{def:comb_PL}
A \emph{\edits{strict} piecewise linear} function on $\Gamma$ is a function $$F \colon V(\Gamma) \to M^\gp$$ 
satisfying the following property: for every directed edge $\vec e$ from vertex $u$ to vertex $v$, 
there is an integer $s(\vec e)\in \mathbb{Z}$, the \emph{slope of $F$ along $\vec e$}, such that 
\begin{equation}\label{eq:slope}
F(v) - F(u) = s(\vec e)\ell(e).
\end{equation}
We will use the abbreviation \edits{sPL} for \edits{strict} piecewise linear, and we write 
$
\mathsf{sPL}(\Gamma)
$
for the group of sPL functions on $\Gamma$.
\end{definition}

By taking slopes, we obtain a homomorphism 
\[
\mathsf{sPL}(\Gamma) \to \mathsf{Flow}(\Gamma).
\]
Composing with the homomorphism $\mathrm{div}$ yields a homomorphism 
\[
\mathsf{sPL}(\Gamma) \to \mathsf{Div}(\Gamma)
\]
which we also denote by $\mathrm{div}$. Explicitly, if $s(\vec e)$ is the slope of $F$ along the oriented edge $\vec{e}$, we have 
\begin{equation*}
    \mathrm{div}(F) = \left( \sum_{\vec e \to v} s(\vec e) \right)_{v \in V(\Gamma)} \in \mathbb{Z}^{V(\Gamma)}\, .
\end{equation*}
We say a divisor is \emph{principal} if it is the divisor of some sPL function. Two divisors are {\em linearly equivalent} if their difference is principal. 
Whether a divisor is principal depends strongly on the monoid $M$ (via the edge lengths),
and \emph{not} just on the underlying graph $\Gamma$.  

If two sPL functions have the same slopes, then they differ by addition of a global constant. The next result determines when an (acyclic) flow arises as the slopes of a sPL function. \edits{To state the result, we require a slight generalization of the formula \eqref{eqn:ellpairing}. 

\begin{definition}
We define the pairing 
\begin{equation*}
    \langle-, -\rangle_\ell\colon \mathsf{Flow}(\Gamma) \times \mathsf{Flow}(\Gamma) \to \M^\gp
\end{equation*}
by the formula
\begin{equation}\label{gg79}
\langle f, g \rangle_\ell = 
\frac{1}{2} \sum_{\vec e \in \vec E(\Gamma)} \ell(e) f(\vec e) g(\vec e). 
\end{equation}
The prefactor $\frac{1}{2}$ is placed to count every edge contribution only once.
\end{definition}
Every cycle $\gamma\in H_1(\Gamma)$ induces a flow by identifying $H_1(\Gamma)$ with the kernel of $\mathrm{div}$ as in \eqref{eq:div_of_flow}.  Hence, we sometimes view the pairing as being between flows and cycles. 
}

\begin{lemma}\label{lem:flow_as_slopes}
Let $f$ be a flow on $\Gamma$. Then $f$ arises as the slopes of a sPL function if and only if, for every cycle $\gamma$ in $\Gamma$, we have 
$$
\langle \gamma, f \rangle_{\ell} = 0. 
$$
\end{lemma}
\begin{proof}
The statement is well-known. That the pairing vanishes on the slopes of a sPL function,
we leave to the reader. For the converse, start by choosing any value for $F$ at a
particular vertex. By  moving around the graph, 
the condition \eqref{eq:slope} determines the value of $F$ at every other vertex. 
The method can fail if two different ways of reaching a vertex were to yield different answers, which is precisely precluded by the vanishing of the pairing. 
\end{proof}

In particular, if a flow $f$ arises as the slopes of a sPL function, then $f$ must be acyclic: starting at a vertex $v$ in the cycle, the values $F(w)$ of a sPL function $F$ must strictly increase when traversing edges with positive slope, but must return to $F(v)$ when traversing the whole cycle. \edits{The order here on $M^\gp$ is induced by the inclusion $M \subset M^\gp$: given $x,y \in M^\gp$, we say $x \le y \in M^\gp$ if $y-x \in M \subset M^\gp$. }

\edits{
\begin{remark}
    In the tropical literature, a tropical curve is typically a metric graph, and a divisor is a finite sum 
    \[
    \sum a_p p
    \]
    with $a_p \in \mathbb{Z}$ and $p \in \Gamma$. For us, a tropical curve is a finite set of vertices and edges, with edges metrized by a monoid $M$. When the monoid $M$ is $\mathbb{R}_{\ge 0}$, our tropical curves can be regarded as metric graphs. But for general monoids $M$, we do not
    have metric graphs.
    Moreover,
    the notion of a point of an edge of $\Gamma$ does not have a clear meaning. 
    Therefore, we choose to use finite formal sums of vertices in our definition of divisors. If an edge can be subdivided over $M$ into two edges, it is reasonable to consider the new vertex of the subdivision a point of $\Gamma$, but we prefer to consider it as a divisor on the subdivision rather than on the original curve.  
\end{remark}
}
\subsection{Algebraizing tropical Abel-Jacobi theory}\label{sec:algebraizing_tropical_AJ}

Logarithmic geometry allows us to lift the combinatorics of tropical Abel-Jacobi theory
to algebraic geometry. 

A \emph{\edits{strict} piecewise linear} function on a log curve $C \to S$ is a global section of $\ghost_C^\gp$. Over a geometric point $s$ of $S$,
such a sPL function $$\alpha \in H^0(C, \ghost_C^\gp)$$ induces a combinatorial sPL function (in the sense of Definition \ref{def:comb_PL}) on the tropical curve, by sending a vertex $v$  of $\Gamma_s$ to the value of $\alpha$ at the generic point of the irreducible component of $C_s$ corresponding to $v$. 

When $S$ is a geometric point, this gives a bijection between combinatorial sPL functions and $H^0(C, \ghost_C^\gp)$. On the other hand, suppose that we are given, for each geometric point $s$ of $S$, a combinatorial sPL function 
$$F_s\colon V(\Gamma_S)\to \ghost_{S,s}^\gp\, . $$
Suppose, moreover, that the $F_s$ are compatible with \'etale specialization: whenever $\zeta\colon t \rightsquigarrow s$ is an \'etale specialization inducing
$$p_\zeta^\#\colon \ghost_{S,s} \to \ghost_{S,t}$$
and
$$f_\zeta\colon \Gamma_s \to \Gamma_t\, ,$$
we have 
$$p_\zeta^\#F_s(v) = F_t(f_\zeta(v))\, .$$
Then there is a unique global section 
$$\alpha \in H^0(C, \ghost_C^\gp)$$
inducing the combinatorial sPL functions $F_s$.

The geometric connection of the two constructions is the following: The sPL function $\alpha$ induces a line bundle $\mathcal{O}(\alpha)$ on $C$, as in Definition \ref{def: associatedbundle}. Then, the multidegree of $\mathcal{O}(\alpha)$ on $C_s$ is opposite to the tropical divisor associated to $F_s$: 
\begin{equation}
    \label{eq: multideg}
    \mdeg \mathcal{O}(\alpha)|_{C_s} = - \mathsf{div}(F_s)
\end{equation}
This equality is in particular responsible for our convention to glue the infinity rather than the zero section to the torsor $\mathcal{O}_{C}^\times(\alpha)$.


\section{Stability conditions}\label{sec:stability_conditions}

\subsection{Definitions} \label{Sect:Stab_cond_Defs}

The  definitions related to stability of line bundles
here are adapted from \cite{MMUV}. Let $$\pi: C \to S$$
be a logarithmic curve. 

\begin{definition} \label{def:stabilitycondition}
A \emph{stability condition $\theta$ of degree $d$ for $\pi$} consists of a function $$\theta:V(\Gamma_s) \to \edits{\mathbb{Q}}$$ for each geometric point $s \in S$, which
satisfies:
 \begin{enumerate}
    \item[(i)] For $s\in S$, we have    $\sum_{v \in V(\Gamma_s)} \theta(v) = d$.
 \item[(ii)] For every \'etale specialization $\zeta\colon t \rightsquigarrow s$, the stability condition $\theta$ respects the induced map $\Gamma_s \to \Gamma_t$, in the sense that 
$$
\theta(w) = \sum_{v_i} \theta(v_i)
$$
whenever $v_1,\cdots,v_n$ are the vertices of $\Gamma_s$ mapping to $w \in \Gamma_t$.\qedhere 
\end{enumerate}
\end{definition}
For the universal family over the moduli space of stable curves,
$$\pi : C \to S = \oM_{g,n}\,,$$
a stability condition of the form of Definition \ref{def:stabilitycondition} is exactly
a stability condition in the sense of Section \ref{Sect:intromodulilinebundles}.

If $T \to S$ is a logarithmic map, and $\widehat{C} \to C \times_S T = C_T$  is a quasi-stable model for $C_T/T$, then we can lift a degree-$d$ stability condition $\theta$ for $C/S$ canonically to a degree-$d$ stability condition $\widehat{\theta}$ for $\widehat{C}/T$ by setting the  value of $\widehat{\theta}$ to zero on every exceptional vertex of the dual graph of every fiber $\widehat{C}_{t}$ over every geometric point $t \in T$ (and setting the value of $\widehat{\theta}$  to equal the value of $\theta$ on all other vertices).


\begin{definition}
Let $T \to S$ be a logarithmic map. A line bundle $\ca L$ on a quasi-stable model $\widehat{C} \to C_T$ is  \emph{admissible} if for any geometric point $t \in T$ and exceptional component $E$ of $\widehat{C}_t \to (C_T)_t$, the degree of $\ca L$ on $E$ is $1$.  
\end{definition}

\edits{The line bundle $\ca L$ tropicalizes to a divisor on the tropicalization of $\widehat{C}\to C_T$ in each fiber}: for each geometric point $t \in T$, the divisor is given by the multidegree $$\mdeg(\ca L|_{\widehat{C}_t}) \in \mathbb{Z}^{V(\widehat{\Gamma}_t)} = \textup{Div}(\widehat{\Gamma}_t)\, .$$ 
A line bundle $\ca L$ is then admissible if and only if the
associated tropical divisor has degree $1$ on every exceptional vertex. 

\begin{definition} \label{defss}
Let $T \to S$ be a logarithmic map, and let $\ca L$ be an admissible line bundle on a quasi-stable model $\widehat{C} \to C_T$ of $C_T/T$. 
Let $\theta$ be a stability condition for $C/S$ with lift $\widehat{\theta}$ to $\widehat{C}/T$. Then, $\ca L$ is \emph{$\theta$-semistable} if, for every geometric point $t \in T$, 
the associated tropical divisor $\mdeg(\ca L|_{\widehat{C}_t})$ on the graph $\widehat{\Gamma}_t$ of $\widehat{C}_t/t$ satisfies the condition: 
$$
\widehat{\theta}(G)-\frac{\textup{val}(G)}{2} \le \mdeg(\ca L|_{\widehat{C}_t})(G) \le \widehat{\theta}(G) + \frac{\textup{val}(G)}{2} 
$$
for every subset $G \subseteq V(\widehat{\Gamma}_t)$, where 
$$\widehat{\theta}(G) = \sum_{v \in G} \widehat{\theta}(v)\, , \ \ \  \mdeg(\ca L|_{\widehat{C}_t})(G) = \sum_{v \in  G} \mdeg(\ca L|_{\widehat{C}_t})(v)\, $$ and $\textup{val}(G)$ is the valence of $G$, defined as the number of edges connecting $G$ and its complement.

An admissible line bundle $\ca L$ on $\widehat{C}/T$ is  \emph{$\theta$-stable} if the inequalities of
Definition \ref{defss} 
are strict except possibly when $G$ or its complement are unions of exceptional vertices. 
\end{definition}
In Section \ref{Sect:intromodulilinebundles}, we cast the stability condition in terms of sub-curves which correspond here exactly to subsets of the vertices.

\begin{definition}
A stability condition $\theta$ for $C/S$ is  
\begin{itemize}
    \item \emph{nondegenerate} if, for every logarithmic map $T \to S$ and quasi-stable model $\widehat{C} \to C_T$, every $\theta$-semistable line bundle $\ca L$
on $\widehat{C}/T$ is $\theta$-stable,
    \item \emph{small} if the trivial bundle is $\theta$-semistable on $C_T$. 
\end{itemize}
\end{definition}

The category fibered in groupoids over $\textbf{LogSch}/S$ which assigns to $T \to S$ the groupoid consisting of pairs of 
\begin{enumerate}
    \item [(i)] a quasi-stable model $\widehat{C} \to C_T$, 
    \item[(ii)] an admissible line bundle $\ca L$ on $\widehat{C} \to T$
\end{enumerate}
forms an algebraic stack\footnote{\edits{It is easy to see that $\mathbf{P}^\theta_\pi$ is a stack in the strict \'etale topology. We justify that this stack is algebraic. The Picard stack is algebraic \cite[Theorem 7.3]{Artin1969Algebraization-}, and admissibility is an open condition, so it suffices to check that the stack of subdivisions is algebraic. But to specify a subdivision of a log curve along an edge $e$ is to specify a monoid element that is strictly between $0$ and the length of $e$. The space of monoid elements is representable by $[\bb A^1/\bb G_m]$, and strict inequalities between monoid elements cut out closed substacks. }}. 
Imposing the condition that $\ca L$  is
$\theta$-semistable \edits{(for some nondegenerate $\theta$)} defines \edits{an} open substack which we denote by $\mathbf{P}_\pi^\theta$, or $\mathbf{P}^\theta$ when $\pi:C \to S$ is understood. The relative inertia of $\mathbf{P}_\pi^\theta$ over $S$ is given by $\bb G_m$. 
The associated rigidification over $S$, denoted by
$\mathcal{P}_\pi^\theta$,
is a proper algebraic space over $S$. 
Both $\mathbf{P}_\pi^\theta$ and $\mathcal{P}_\pi^\theta$ are log smooth over $S$, and their tropicalizations parametrize $\theta$-semistable tropical divisors on the tropicalization $\Sigma_C \to \Sigma_S$ (see \cite{MMUV} for details).

\begin{remark}
    According to the definition of $\theta$-semistability, the trivial bundle is \emph{not} semistable on quasi-stable models $\widehat{C} \to C_T$ that have exceptional components, as its degree fails to be $1$ on them. 
    Instability is forced here
    as otherwise there would be no chance of getting a separated space $\mathcal{P}_\pi^\theta$: over a discrete valuation ring $T$ with generic point $\eta$, both $(C_T,\mathcal{O}_{C_T})$ and $(\widehat{C},\mathcal{O}_{\widehat{C}})$ would be limits of $(C_\eta,\mathcal{O}_{C_\eta})$. 
\end{remark}

\subsection{Nondegenerate stability conditions}
\label{ndsc}
\subsubsection{Moduli of \texorpdfstring{$\theta$}{theta}-stable bundles}
Let $\pi: C \to S$ be a logarithmic curve together with a section $x_1$, \edits{where $S$ is an algebraic stack with log structure}. Let 
$$\mathcal{L} \to C$$
be a line bundle on $C$ of degree{\footnote{Degree here means the
degree of $\mathcal L$ on the fibers of $\pi$.}} $d$. Let $\theta$ 
be a nondegenerate stability condition of degree $d$ for $C/S$. We present here
a generalization of the construction of $\Mbar_{g,A}^\theta$ in Section \ref{sasc}. The main ideas are already contained in \cite{AbreuPacini}, but in a setting less general than we require. 

\begin{definition}\label{def:S_theta}
Let $S^\theta_\ca L$ be the fibered category over $\mathbf{LogSch}/S$ with objects tuples
\begin{equation}
    (T/S, \widehat C \to C_T, \alpha)
\end{equation}
where $T$ is a log scheme over $S$, $\widehat C \to C_T$ is a quasi-stable model of $C_T/T$, and $\alpha$ is \edits{a} sPL function on $\widehat C/T$ vanishing on \edits{the component through which $x_1$ passes}, and for which $\ca L_T(\alpha)$ is $\theta$-stable.  
\end{definition}

Standard arguments show that $S^\theta_\ca L$ is a stack in groupoids for the strict \'etale topology: a quasi-stable model of $C_T$ is associated to a subdivision of $\Sigma_{C_T}$, and $\alpha$ is a \edits{strict} piecewise linear function on $\Sigma_{C_T}$ (both of these are by definition \'etale local on $T$).  

We \edits{write $C_{S_\ca L^{\theta}} \to S_\ca L^{\theta}$ for the pullback of $C$ under the map $S_\ca L^{\theta} \to S$. Furthermore, we} denote \edits{its} universal quasi-stable model \edits{(written as $\widehat C$ in Definition \ref{def:S_theta} when taking $T=S_\ca L^{\theta}$)} by
$$C^\theta\to C_{S_\ca L^{\theta}}\, $$
and the universal sPL function
on $C^\theta$ by $\alpha^\theta$.

\begin{remark}\label{rem:aj_map_on_S^theta}
Let $\ca P^\theta_\pi$ be the moduli space of $\theta$-stable line bundles over $S$ trivialized along the section $x_1$ \edits{discussed at the end of Section \ref{Sect:Stab_cond_Defs}}. There is a natural Abel-Jacobi map 
\begin{equation}
  \mathsf{aj}\colon S^\theta_\ca L \to \ca 
  P^\theta_\pi
\end{equation}
sending $(T/S, \widehat C \to C_T, \alpha)$ to the \edits{rigidification of the} line bundle $\ca L_T(\alpha)$ on $\widehat C$ \edits{along the section $x_1$}. 
\end{remark}

The definition of $S_{\mathcal{L}}^\theta$ is natural from the point of view of logarithmic geometry. While standard arguments \edits{using \cite{Molcho2018The-logarithmic}} show that $S_{\mathcal{L}}^\theta$ is a stack over $\textbf{LogSch}/S$, what is not standard is 
whether $S_{\mathcal{L}}^\theta$ \emph{can be represented by an algebraic stack with a log structure}. Theorem \ref{thm:is_subdivis} gives an explicit description of $S_{\mathcal{L}}^\theta$ as a log modification of $S$, thus showing  $S_{\mathcal{L}}^\theta$ is an algebraic stack with log structure.  



\begin{theorem}\label{thm:is_subdivis} The structure morphism
$\rho: S^\theta_\ca L \to S$ is a log modification (obtained from a subdivision). In particular, $\rho$ is proper, log \'etale, and relatively representable by logarithmic schemes. If $S$ is log smooth then $\rho$ is also birational. 
\end{theorem}

\begin{remark}
In Definition \ref{def:S_theta}, the notation $S_{\mathcal{L}}^\theta$ is used to highlight the
fact that $S_{\mathcal{L}}^\theta$ depends both on the stability condition $\theta$ and the line bundle $\mathcal{L}$. The proof of Theorem \ref{thm:is_subdivis} will however show that $S_{\mathcal{L}}^\theta$ only depends on  $\mathcal{L}$ through 
the associated multidegree $\mdeg(\mathcal{L})$, which we can view as a tropical divisor on $\Sigma_C$. We could have thus equally written $S_{\mdeg(\mathcal{L})}^\theta$.  
\end{remark}

Theorem \ref{thm:is_subdivis} generalizes the construction of
\begin{equation*}
    \rho: \oM_{g,A}^\theta \rightarrow \oM_{g,n}
\end{equation*}
promised in \eqref{eq:intro_resolve_aj}.
The following definition makes the latter claim precise. 
\begin{definition}\label{def:Mgntheta_detailed}
Let $S = \Mbar_{g,n}$ with universal curve
$\pi: C\to \Mbar_{g,n}$
carrying the standard
log structure. Let
\begin{equation*}
    \ca L = (\omega_C^\log)^{\otimes k}\left(- \sum_{i=1}^n a_i x_i\right)
\end{equation*}
be a line bundle
on $C$ with markings $x_1,\ldots,x_n$.
 For a nondegenerate stability condition $\theta$ for $\pi$,
we define
\begin{equation}
    \Mbar_{g, A}^\theta = (\Mbar_{g, n})_{\ca L}^\theta\, , 
\end{equation}
with Abel-Jacobi map 
\begin{equation}
  \mathsf{aj}\colon \Mbar_{g, A}^\theta \to \ca P_{g,n}^\theta
\end{equation}
as in Remark \ref{rem:aj_map_on_S^theta}. 
\end{definition}

\subsubsection{Proof of Theorem \ref{thm:is_subdivis}}\label{sec:proof_of_thm_is_subdivision}
\edits{To check that a map $S' \to S$ is a log modification, it suffices to work locally in $S$. We may thus assume $S$ has a strict map to an Artin fan $\mathcal{A}_S$.}
We present an explicit construction of the subdivision of \edits{the corresponding cone stack} $\Sigma_S$ which defines $\rho$.
In the special case 
of Definition \ref{def:Mgntheta_detailed},
where $S = \Mbar_{g,n}$, the subdivision 
reduces exactly to the construction in Section \ref{sasc}. 

Let $s$ be a geometric point of $S$ and let $\sigma_s = \mathsf{Hom}(\overline{M}_{S,s},\mathbb{R}_{\ge 0})$ be the corresponding cover of the stacky cone in $\Sigma_S$. By definition, a subdivision of $\Sigma_S$ is a subdivision of the various cones $\sigma_s$ as $s$ ranges through $S$, 
which is compatible with all identifications of faces and automorphisms required in $\Sigma_S$. Let $\Gamma = \Gamma_s$ be the dual graph of $C_s$. We will write $\sigma_{\Gamma}$ for $\sigma_s$ to highlight the connection of the constructions below with $\Gamma$. 

The logarithmic map $C \to S$ induces the structure of a tropical curve on $\Gamma$ metrized by $$\ghost_{\sigma_\Gamma} = \ghost_{S,s}$$
via the lengths $\ell(e) \in \ghost_{S,s}$, as in Section \ref{sec:Logcurves}. All together, the lengths $\ell(e)$ can be thought of as a homomorphism 
$$
\ell \colon \sigma_\Gamma \to \mathbb{R}_{\ge 0}^{E(\Gamma)}\, , 
$$
\edits{so for $p \in \sigma_\Gamma = \mathsf{Hom}(\ghost_{S,s}, \mathbb{R}_{\geq 0})$ we can write $(\ell(p))_e = \ell_e(p) = p(\ell(e))$.}
Let $\mdeg(\mathcal{L})$ be the multidegree of $\mathcal{L}$
viewed as a tropical divisor on $\Gamma$. The cones of our subdivision in $\sigma_\Gamma$ will be indexed by triples $(\widehat \Gamma, D, I)$ where
\begin{enumerate}
    \item[(i)] $\widehat \Gamma$ is a quasi-stable graph with stabilization $\Gamma$,
    \item[(ii)] $D$ is a $\theta$-stable divisor on $\widehat{\Gamma}$,
    \item[(iii)] $I$ an acyclic flow on $\widehat \Gamma$ satisfying $\on{div}(I) = \mdeg(\mathcal{L})-D$. 
\end{enumerate}

We define the cone $\sigma_{\widehat{\Gamma}}$ by
$$\sigma_{\widehat{\Gamma}} = \sigma_\Gamma \times_{\mathbb{R}_{\ge 0}^{E(\Gamma)}} \mathbb{R}_{\ge 0}^{E(\widehat{\Gamma})}\, .$$ 
Concretely, a point of the cone $\sigma_{\widehat{\Gamma}}$ is given by 
\begin{enumerate}
    \item[$\bullet$] a point $p \in \sigma_\Gamma$,
    \item[$\bullet$] for every edge $e$ of $\Gamma$ that is subdivided into edges $e_1$ and $e_2$ in $\widehat\Gamma$, a pair of nonnegative real numbers $\widehat{\ell}(e_1)$ and  $\widehat{\ell}(e_2)$ such that $\ell(p)_e = \widehat{\ell}(e_1) + \widehat{\ell}(e_2)$. 
\end{enumerate}
 The cone $\sigma_{\widehat{\Gamma}}$ has relative dimension over $\sigma_\Gamma$ equal to the number of exceptional vertices of $\widehat{\Gamma}$. 
 
 In the universal case, when $S = \oM_{g,n}$, the map $\ell$ is an isomorphism, $\sigma_{\widehat{\Gamma}} =\mathbb{R}_{\ge 0}^{E(\widehat{\Gamma})}$, and the projection $\sigma_{\widehat{\Gamma}} \to \sigma_\Gamma$ is the map
$${\mathsf{pr}}\colon \sigma_{\widehat{\Gamma}} \to \sigma_\Gamma\, ,\ \ \  \widehat{p} \, \mapsto\,  p : e \mapsto 
\begin{cases}
\widehat{p}(e_1)+\widehat{p}(e_2) &\text{if $\widehat \Gamma$ subdivides edge $e$ into $e_1, e_2$},\\
\widehat{p}(e) & \text{otherwise},
\end{cases}
$$
from Section \ref{sasc}. 



We define a subcone $\tau_{\widehat{\Gamma}, I} \subseteq \sigma_{\widehat \Gamma}$ by the condition
\begin{equation}\label{eq:tau_cone_def}
    \tau_{\widehat{\Gamma},I} = \left\{\widehat p \in \sigma_{\widehat \Gamma} \ \Big| \  \langle \gamma, I \rangle_{\widehat p} = 0 \text{ for all }\gamma \in H_1(\widehat{\Gamma}) \right\}\,,
\end{equation}
where the pairing is \eqref{gg79} evaluated at the point $\widehat p$. 


\begin{lemma}\label{Dcomment-I-should-try-to-prove-these-claims}
The map $\mathsf{pr}:\sigma_{\widehat{\Gamma}} \to \sigma_\Gamma$ induces an isomorphism from the cone $\tau_{\widehat{\Gamma},I}$ to the
image $$\sigma_{\widehat{\Gamma},I} = {\mathsf{pr}}(\tau_{\widehat{\Gamma},I}) \subseteq \sigma .$$
Moreover,  the collection of cones $\{ \sigma_{\widehat{\Gamma},I}\} $ together with their faces forms a fan $\widetilde{\Sigma}^{\theta}_{\Gamma}$ with support $\sigma_\Gamma$.
\end{lemma}
\begin{proof}
The claim is purely combinatorial and is proven by Abreu and Pacini \edits{\cite{AbreuPaciniUniversal,AbreuPacini}. The second paper \cite{AbreuPacini} works in a 
slightly more restricted setting (over the moduli stack of stable 1-marked curves with a particular stability condition), though the arguments go through essentially unchanged in our more general setting.
The heart of the proof is in the first paper \cite{AbreuPaciniUniversal}, which works with stability conditions \emph{more} general than ours.}

\edits{We explain how to translate the notation of \cite{AbreuPacini} to ours.}
The cones $C_{\Gamma, \ca E, \phi}$ of \cite[Definition 3.4]{AbreuPacini} correspond to our cones $\tau_{\widehat{\Gamma},I}$, and the cones $K_{\Gamma, \ca E, \phi}$ are our $\sigma_{\widehat{\Gamma},I}$. 
By \cite[Proposition 3.7]{AbreuPacini}, the map 
$${\mathsf{pr}} : \tau_{\widehat{\Gamma},I} \to  \sigma_{\widehat{\Gamma},I}$$ 
is an isomorphism. That these cones fit together to \edits{a fan which surjects onto $\sigma_\Gamma$} is proven in \cite[Theorem 3.9]{AbreuPacini}.
The content of the proof of \edits{surjectivity} is \cite[Theorem 5.6]{AbreuPaciniUniversal}: every degree
$d$ tropical divisor is linearly equivalent to a unique $\theta$-stable divisor on a unique quasi-stable subdivision. 
\end{proof}

To show that the construction of Lemma \ref{Dcomment-I-should-try-to-prove-these-claims}
defines a global subdivision $\widetilde{\Sigma}^{\theta}$ of $\Sigma_S$, we need furthermore to check that the cones $\sigma_{\widehat{\Gamma},I}$ descend to $\Sigma_S$: the cones must fit together along all identifications of faces and automorphisms of the various cones $\sigma_\Gamma$ which define $\Sigma_S$. The proof of \cite[Theorem 3.9]{AbreuPacini} again goes through unchanged in our setting to show that a cone $\sigma_{\widehat{\Gamma}',I'}$ is a face of $\sigma_{\widehat{\Gamma},I}$ whenever 
\begin{itemize}
    \item $\sigma_{\Gamma'}$ is a face of $\sigma_\Gamma$, 
    \item $\widehat{\Gamma}'$ is obtained from $\widehat{\Gamma}$ by contracting some edges, 
    \item $I'(e) = I(e)$ for all edges $e \in E(\widehat{\Gamma}') \subset E(\widehat{\Gamma})$.
\end{itemize}
Thus, the cones $\sigma_{\widehat{\Gamma},I}$ form a polyhedral complex and define a global subdivision of $\Sigma_S$.  For the gluing, it is necessary that the data $\widehat{\Gamma},I$ respects the monodromy and identifications of faces of $\sigma_\Gamma$. But this is precisely implied by condition (ii) in the
definition of a stability condition. 

\begin{remark}\label{rem:cone_intuition}
The intuition behind the construction of the cones $\sigma_{\widehat{\Gamma},I}$ is as follows.
Given $I$ on $\widehat{\Gamma}$ and $p$ in the interior
of $\sigma_\Gamma$, it is in general impossible to both
\begin{enumerate}
    \item[(i)] assign lengths to the edges of $\widehat{\Gamma}$ summing to the lengths on edges of $\Gamma$ associated to $\ell(p)$,
    \item[(ii)]find a strict piecewise linear function on $\widehat{\Gamma}$ whose slopes are given by $I$. 
\end{enumerate}
 Indeed, by Lemma \ref{lem:flow_as_slopes}, the existence of such a lifting imposes several conditions on the lengths of the edges of $\widehat{\Gamma}$. The cone $\sigma_{\widehat{\Gamma}}$ is the universal cone over which the pullback of $\Gamma$ admits a quasi-stable model of combinatorial type $\widehat{\Gamma}$. 
The subcones $\tau_{\widehat{\Gamma},I}$ are cut out inside $\sigma_{\widehat{\Gamma}}$ precisely by the linear conditions necessary to lift $I$ into a strict piecewise linear function. Thus, the content of Lemma \ref{Dcomment-I-should-try-to-prove-these-claims} is that the locus of points $p \in \sigma$ which support a quasi-stable model $\widehat{\Gamma}$ of $\Gamma$ with a sPL function with slopes $I$ is a cone $\sigma_{\widehat{\Gamma},I}$, isomorphic to $\tau_{\widehat{\Gamma},I}$.

Furthermore, the cones for various choices $\widehat{\Gamma},I$ that come from the stability condition $\theta$ subdivide $\sigma_\Gamma$: each $p$ in the interior of  $\sigma_{\Gamma}$ belongs to the interior of precisely one such cone. Points $p$ in the boundary of $\sigma_{\Gamma}$ similarly belong to the interior of cones $\sigma_{\widehat{\Gamma}',I'}$, where $\Gamma'$ is the contraction of $\Gamma$ lying over $p$.  
Thus, the property of supporting a quasi-stable model with strict piecewise linear function with slopes $I$ is cut out by inequalities on the lengths of the edges of $\Gamma$. 
\end{remark}

We will show below in Lemma \ref{lem:isom_of_stacks} that 
the log modification
$$\tilde{\rho}:\widetilde S^\theta_\ca L\to S$$
corresponding to the subdivision $\widetilde{\Sigma}^\theta$ of $\Sigma_S$
is isomorphic over $S$ to $S^\theta_\ca L$.

Let $f \colon T \to S$ be any logarithmic scheme over $S$, and let $t \in T$ be
a geometric point mapping to $s \in S$. The log curve $C \to S$ pulls back to a log curve $C_T \to T$. The induced tropical curve around $t$ has underlying graph $\Gamma= \Gamma_s$ metrized by the lengths $f^\sharp(\ell(e)) \in \ghost_{T,t}$, where $$f^\sharp:\ghost_{S,s} \to \ghost_{T,t}$$ is the induced homomorphism. Equivalently, the tropical curve is determined by the composition 
\[
\begin{tikzcd}
\tau \ar[r] & \sigma_\Gamma \ar[r,"\ell"] & \mathbb{R}_{\ge 0}^{E(\Gamma)}\, ,
\end{tikzcd}
\]
where $\tau$ is the cone dual to $\ghost_{T,t}$. As in Section \ref{subsec:subsoflogcurves}, to give a quasi-stable model $\widehat{C} \to C_T$, 
we must specify, for every geometric point $t \in T$, the structure of a quasi-stable model of $\Gamma$ metrized by $\ghost_{T,t}$, compatibly with \'etale specializations. In other words, to give a quasi-stable model with specified dual graph $\widehat{\Gamma}$, we must  lift the homomorphism $\tau \to \mathbb{R}_{\ge 0}^{E(\Gamma)}$ to a homomorphism
$$
\tau \to \mathbb{R}_{\ge 0}^{E(\widehat{\Gamma})}
$$
for the various cones $\tau$ dual to the $\ghost_{T,t}$, compatibly with \'etale specializations. By construction, the cone $\sigma_{\widehat{\Gamma}}$ is the universal cone over which the morphism $\ell$ lifts. Thus, $\widehat{\Gamma}$ inherits the structure of a tropical curve over $\sigma_{\widehat{\Gamma}}$, and also over its subcones $\tau_{\widehat{\Gamma},I}$. Since the $\sigma_{\widehat{\Gamma},I}$ are isomorphic to $\tau_{\widehat{\Gamma},I}$, we see that $\widehat{\Gamma}$ admits the structure of a tropical curve over $\sigma_{\widehat \Gamma,I}$. For any face $\sigma_{\Gamma'} < \sigma_{\Gamma}$, the tropical curve $\Gamma'$ over $\sigma_{\Gamma'}$ induced by $C \to S$ is obtained from $\Gamma$ by contracting some edges. If over $\sigma_{\Gamma'}$ an edge $e$ of $\Gamma$ is contracted, both corresponding edges of $\widehat{\Gamma}$ are contracted as well, since their lengths must add up to that of $e$. Thus $\widehat{\Gamma}$ naturally contracts to a quasi-stable model of $\Gamma'$ over $\sigma_{\Gamma'}$. Therefore, the various $\widehat{\Gamma}$ glue and, as described in Section \ref{subsec:subsoflogcurves}, determine a quasi-stable subdivision 
$$\widehat{C} \to C \times_S \widetilde{S}^\theta_{\ca L}\, .$$
Just as the cone $\sigma_{\widehat{\Gamma}}$ tautologically carries a quasi-stable subdivision of $\Gamma \times_{\sigma_\Gamma} \sigma_{\widehat{\Gamma}}$, the subcone $\tau_{\widehat{\Gamma},I}$ is the universal cone on which the flow $I$ can be realized as the slopes of a strict piecewise linear function unique up to pullback from the base, \edits{because it is defined as the set of points in $\sigma_{\widehat \Gamma}$ where the pairing \eqref{gg79} vanishes, and Lemma \ref{lem:flow_as_slopes} tells us that this is exactly the condition for the sPL function to exist}.  Hence  the curve $\widehat C$ carries a unique sPL function
\begin{equation}\label{eq:universal_alpha}
    \alpha \colon V(\widehat\Gamma) \to \ghost_{\sigma_{\widehat{\Gamma},I}}^\gp
\end{equation}
\edits{which vanishes along $x_1$ (see Definition \ref{def:S_theta}). }

 For each vertex $v \in \widehat{\Gamma}$, the sPL function $\alpha$ assigns an element of $\ghost_{\sigma_{\widehat{\Gamma},I}}^\gp$, or equivalently, a homomorphism 
 $$\alpha(v)\colon \sigma_{\widehat{\Gamma},I} \to \mathbb{R}\, . $$
 By construction, the function $\alpha$ has the following explicit description. Denote by $v_0$ the vertex in $\widehat{\Gamma}$ corresponding to the component that contains the marking $x_1$. Then, for any other $v \in V(\widehat{\Gamma})$, we have 
$$
\alpha(v) = \sum_{\vec{e} \in \gamma_{v_0 \to v}} I(\vec{e})\cdot \widehat{\ell}_e\, ,
$$
where $\gamma_{v_0 \to v}$ is \emph{any} oriented path from $v_0 \to v$, $I(\vec{e})$ is the integer slope of the flow $I$ along $\vec{e}$, and $$\widehat{\ell}_e: \sigma_{\widehat{\Gamma},I} \to \mathbb{R}$$ is the homomorphism determined by the length of the edge $e$. We thus obtain equation \eqref{xsse3} of Section \ref{Sect:logDRformulaintro}.

The definition of $\alpha(v)$ does not depend on the choice of a path between $v_0$ and $v$,
If $\gamma$ and $\gamma'$ are \edits{paths}
from $v_0 \to v$,  $\gamma-\gamma'$ is an oriented cycle, and
$$
\langle \gamma-\gamma', I \rangle_{\widehat \ell} = 0
$$
on $\sigma_{\widehat{\Gamma},I}$ by definition.\footnote{A similar argument is used in the proof of Lemma \ref{lem:flow_as_slopes}.} 
On the other hand, $\alpha$ does depend on the choice of $v_0$. 


We have defined a combinatorial sPL function on the tropical curve $\widehat{\Gamma}$. It is easy to see that these sPL functions over the various cones $\sigma_{\widehat{\Gamma},I}$ are compatible and so, as explained in Section \ref{sec:algebraizing_tropical_AJ}, they give rise to a global sPL function $\alpha$ on $\widehat{C}$.

The role of $\alpha$ is to twist the line bundle $\ca L$ into a $\theta$-stable bundle. By applying \eqref{eq: multideg} and part (iii) of our defining conditions for the tuples $(\widehat{\Gamma}, D, I)$,
we obtain
$$
\mdeg{\ca L(\alpha)} = \mdeg \ca L - \mathsf{div} (\alpha) = \mdeg \ca L - \mathsf{div} I = D\, .
$$
Therefore, the curve $\widehat C$ and the sPL function $\alpha$ together induce a map of stacks $$\phi\colon \widetilde S^\theta_\ca L \to S^\theta_\ca L\, $$
over $S$.


\begin{lemma}\label{lem:isom_of_stacks}
The map of stacks $\phi\colon \widetilde S^\theta_\ca L \to S^\theta_\ca L$ is an isomorphism. 
\end{lemma}
\begin{proof}
We may check the isomorphism strict-\'etale locally. Let $C/T$ be a nuclear log curve (in the sense of \cite[Definition 3.39]{HMOP}; this just means that $T$ is `small enough', and is automatic if for example $T$ is strictly henselian local) with a line bundle $\ca L$ of degree $d$. We must show the following two conditions are equivalent: 
\begin{enumerate}
    \item[(i)] There exists a quasi-stable subdivision $\widehat C \to C$ and a sPL function $\alpha$ on $\widehat C/T$ such that $\ca L(\alpha)$ is $\theta$-stable. 
    \item[(ii)] There exists a quasi-stable subdivision $\widehat \Gamma \to \Gamma$ and a flow $I$ on $\widehat{\Gamma}$ such that the condition $\Span{\gamma, I}_\ell = 0\in \ghost_T^\gp$ holds for every $\gamma \in H_1(\widehat\Gamma)$. 
\end{enumerate}
The condition (i) characterizes the functor of points of $S^\theta_\ca L$, and
condition (ii) characterizes the functor of points of $\widetilde{S}^\theta_\ca L$
by Lemma \ref{lem:subfunctorofpoints}.
The correspondence between subdivisions of the graph and of the curve is clear. If we start with a sPL function $\alpha$, we define the flow $I$ to be the associated slopes. Conversely given a flow $I$, we know $I$ can be realised as the slopes of a sPL function by Lemma \ref{lem:flow_as_slopes}. 
\end{proof}

\begin{remark}
During the course of the proof of Lemma \ref{Dcomment-I-should-try-to-prove-these-claims}, we have noted that while the subdivision of \cite{AbreuPacini} is not constructed in the level of generality that we
require 
(since the results of \cite{AbreuPacini} are for $S = \oM_{g,n}$ and a stability condition pulled back from $\oM_{g,1}$), a slight modification of their argument suffices to construct $\widetilde{S}_{\ca{L}}^\theta$ in general. In fact, more can be said.
From \cite{AbreuPaciniUniversal},  for their specific choice of stability condition, there is a cone stack $\Sigma_{\mathcal{P}_{g,n}^\theta}$ subdividing the universal tropical Jacobian $\mathrm{Tro}\ca{P}\mathrm{ic}_{g,n}$ (see for example \cite{Molcho2018The-logarithmic}), and a \emph{tropical Abel-Jacobi section} $$\Sigma_{\oM_{g,n}} \to \mathrm{Tro}\ca{P}\mathrm{ic}_{g,n}\, .$$
The fan $\widetilde{\Sigma}_{\Gamma}^\theta$ of \cite{AbreuPacini} (and of Lemma \ref{Dcomment-I-should-try-to-prove-these-claims} for 
their particular $\theta$) is nothing but the pullback of the subdivision 
\begin{equation}\label{lss4f}
\Sigma_{\mathcal{P}_{g,n}}^\theta \to \mathrm{Tro}\ca{P}\mathrm{ic}_{g,n}
\end{equation}
to the various cones $\sigma$ of $\Sigma_{\oM_{g,n}}$. 
In the proof of \cite{AbreuPaciniUniversal} that
\eqref{lss4f} 
is a subdivision, nothing about the specific $\theta$ is used, except nondegeneracy. Their proof essentially produces a subdivision \eqref{lss4f} for any stability condition, as discussed in \cite[Theorem 5.5]{MMUV}. By pulling-back the Abel-Jacobi section, we obtain the same subdivision $\widetilde{\Sigma}_{\oM_{g,n}}^\theta$
that we construct in Lemma \ref{Dcomment-I-should-try-to-prove-these-claims}
for the case $S = \oM_{g,n}$.

The choice of \cite{AbreuPacini} to work with stability conditions coming from 1-pointed curves is not about the subdivision $\oM_{g,A}^\theta$, but rather about the 
proof that the Abel-Jacobi section extends to a map $$\mathsf{aj}:\oM_{g,A}^\theta \to \ca{P}_{g,n}^\theta\, $$
The claim is proven 
in \cite{AbreuPacini} by writing a map in formal local coordinates and checking 
explicitly the required extension to
a well-defined global map. The argument uses an explicit description of formal local charts of $\ca{P}_{g,1}^\theta$ for their particular choice of stability condition. Our main observation here
is that in the category of log schemes, regardless of stability condition, the functor of points of $\widetilde{S}_{\mathcal{L}}^\theta$ is easy to describe and is represented by $S_{\ca L}^\theta$. The latter has an evident Abel-Jacobi section. So the formal local analysis is {\em not} needed to extend the Abel-Jacobi section. 

A more geometrical argument can also be given. The line bundle $\ca{L}$ always determines an Abel-Jacobi section $S$ to the logarithmic Jacobian $\mathrm{Log}\ca{P}\mathrm{ic}$ and a natural (logarithmic) fiber product diagram   
\begin{equation*}
\begin{tikzcd}
S^\theta_\ca L \arrow[r] \arrow[d] & S \arrow[d, "\ca L"]\\
\ca P^\theta \arrow[r] & \mathrm{Log}\ca{P}\mathrm{ic}, 
\end{tikzcd}
\end{equation*}
where the bottom map is a subdivision. Granting the theory of the logarithmic Jacobian, the above
diagram shows that $S^\theta_{\ca L}$ is a subdivision and that the Abel-Jacobi section extends automatically to $S^\theta_{\ca{L}} \to \ca{P}^\theta$. The fiber product description readily leads to the definition of the functor of points of $S_{\ca L}^\theta$. This perspective will be explained carefully in the sequel to \cite{MMUV}. 
\end{remark}

\subsection{A graph of genus 2}
\label{eg:sam_big_example}

A simple example in genus 2 is rich enough to contain most of the phenomena discussed. Let $C_0$ denote the {\em dollar sign} curve, 
the union of two rational curves $C_v$ and $C_w$
joined at $3$ nodes.
The dual graph $\Gamma$ of $C_0$ has two vertices $v$ and $w$ and three edges $e_1,e_2,e_3$:

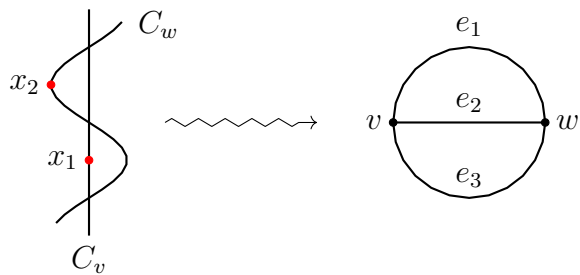
\[
\begin{tikzpicture}
\centering

\draw[thick, domain=-1.5:1.5] plot({-3},{\x});
\draw[thick, domain=-240:240] plot ({-3-0.5*sin \x}, {\x/180});
\node[below] at (-3,-1.5){$C_v$};
\node[right] at (-2.5,1.3){$C_w$};
\node[draw, red, circle, inner sep=1pt, fill] at (-3,-0.5){};
\node[left] at (-3,-0.5){$x_1$};
\node[draw, red, circle, inner sep=1pt,fill] at (-3.5,0.5){};
\node[left] at (-3.5,0.5){$x_2$};
\draw[->,line join=round,decorate, decoration={zigzag,amplitude=1.5,post length=2pt}] (-2,0)--(0,0);

\coordinate(w1) at (3,0){};
\node[right] at (w1){$w$};
\node[above] at (2,1){$e_1$};
\node[above] at (2,0){$e_2$}; 
\node[above] at (2,-1){$e_3$};
\node[draw,circle,inner sep=1pt,fill] at (w1){};
\coordinate(w2) at (1,0){};
\node[left] at (w2){$v$};
\node[draw,circle,inner sep=1pt,fill] at (w2){};
 \draw [thick,domain=0:360] plot ({cos(\x)+2}, {sin(\x)});
\draw[-,thick](w1)--(w2);

\end{tikzpicture}
\]
We place two markings on $C_0$. The first marking is on the component corresponding to $v$, and the second is on the component corresponding to $w$. 
 
Let $(S,s)$ be a nonsingular 3-dimensional base of a versal deformation of $C_0$,
 $$\pi: C \to S\,, \ \ \ \pi^{-1}(s) = C_0\, ,$$
together with sections $x_1, x_2 : S \to C$ meeting the fiber $C_0$ as above.
Both $C$ and $S$ carry canonical log structures. 
 The tropicalization of $S$  is the cone 
 $$\Sigma_S = \mathbb{R}_{\ge 0}^3\, ,$$
 and the tropicalization of $C/S$ is the fibration $\Sigma_C \to \Sigma_S$, which, over a point $(\delta_1,\delta_2,\delta_3)\in \Sigma_S$, assigns the graph $\Gamma$ with edge $e_i$ having length $\delta_i$ (with the understanding that when $\delta_i = 0$, the edge $e_i$ is  contracted in the fiber).

 For every nondegenerate degree 0 stability condition $\theta$, we find that $\theta$-stable divisors $D$ must satisfy the inequalities 
$$ \theta(v)
-\frac{3}{2}  < D(v) < \theta(v) + \frac{3}{2}\, , \ \ \ \ \
\theta(w)
-\frac{3}{2}  < D(v) < \theta(w) + \frac{3}{2}\, .
$$
Therefore, for a small stability condition (close enough to $\theta = 0$), $\theta$-stable divisors must satisfy $-1\leq D(v),D(w)\le 1$. Picking such a $\theta$, we find that the list of all admissible $\theta$-stable divisors 
(up to isomorphism of the graph) is:

$$
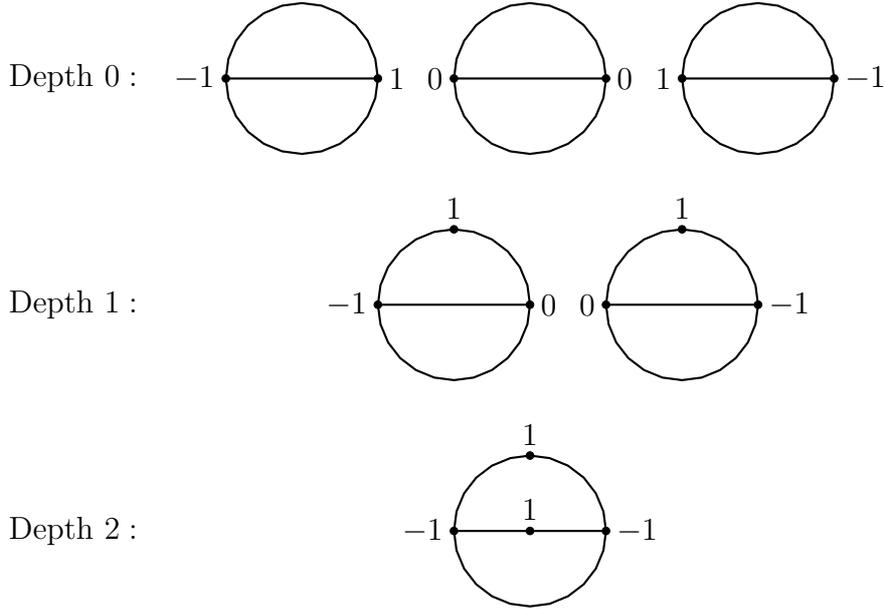
\begin{tikzpicture}
\coordinate(u0) at (-7,0){};
\node[right] at (u0){$\textup{Depth } 0:$};
\coordinate(u1) at (-1,0){};
\node[left] at (u1){$0$};
\node[draw,circle,inner sep=1pt,fill] at (u1){};
\coordinate(u2) at (1,0){};
\node[right] at (u2){$0$};
\node[draw,circle,inner sep=1pt,fill] at (u2){};
 \draw [thick,domain=0:360] plot ({cos(\x)}, {sin(\x)});
\draw[-,thick](u1)--(u2);

\coordinate(v1) at (-4,0){};
\node[left] at (v1){$-1$};
\node[draw,circle,inner sep=1pt,fill] at (v1){};
\coordinate(v2) at (-2,0){};
\node[right] at (v2){$1$};
\node[draw,circle,inner sep=1pt,fill] at (v2){};
 \draw [thick,domain=0:360] plot ({cos(\x)-3}, {sin(\x)});
\draw[-,thick](v1)--(v2);

\coordinate(w1) at (4,0){};
\node[right] at (w1){$-1$};
\node[draw,circle,inner sep=1pt,fill] at (w1){};
\coordinate(w2) at (2,0){};
\node[left] at (w2){$1$};
\node[draw,circle,inner sep=1pt,fill] at (w2){};
 \draw [thick,domain=0:360] plot ({cos(\x)+3}, {sin(\x)});
\draw[-,thick](w1)--(w2);

\coordinate(x0) at (-7,-3){};
\node[right] at (x0){$\textup{Depth } 1:$};
\coordinate(x1) at (-2,-3){};
\node[left] at (x1){$-1$};
\node[draw,circle,inner sep=1pt,fill] at (x1){};
\coordinate(x2) at (0,-3){};
\node[right] at (x2){$0$};
\node[draw,circle,inner sep=1pt,fill] at (x2){};
 \draw [thick,domain=0:360] plot ({cos(\x)-1}, {sin(\x)-3});
\draw[-,thick](x1)--(x2);
\coordinate(x3) at (-1,-2){};
\node[above] at (x3){$1$};
\node[draw,circle,inner sep=1pt,fill] at (x3){};

\coordinate(y1) at (1,-3){};
\node[left] at (y1){$0$};
\node[draw,circle,inner sep=1pt,fill] at (y1){};
\coordinate(y2) at (3,-3){};
\node[right] at (y2){$-1$};
\node[draw,circle,inner sep=1pt,fill] at (y2){};
 \draw [thick,domain=0:360] plot ({cos(\x)+2}, {sin(\x)-3});
\draw[-,thick](y1)--(y2);
\coordinate(y3) at (2,-2){};
\node[above] at (y3){$1$};
\node[draw,circle,inner sep=1pt,fill] at (y3){};

\coordinate(w0) at (-7,-6){};
\node[right] at (w0){$\textup{Depth } 2:$};
\coordinate(w1) at (-1,-6){};
\node[left] at (w1){$-1$};
\node[draw,circle,inner sep=1pt,fill] at (w1){};
\coordinate(w2) at (1,-6){};
\node[right] at (w2){$-1$};
\node[draw,circle,inner sep=1pt,fill] at (w2){};
 \draw [thick,domain=0:360] plot ({cos(\x)}, {sin(\x)-6});
\draw[-,thick](w1)--(w2);
\coordinate(w3) at (0,-6){};
\node[above] at (w3){$1$};
\node[draw,circle,inner sep=1pt,fill] at (w3){};
\coordinate(w4) at (0,-5){};
\node[above] at (w4){$1$};
\node[draw,circle,inner sep=1pt,fill] at (w4){};

\end{tikzpicture}
$$
Each graph of depth 1 and 2 stands for three graphs (after including
those related by symmetry).
There are $12$ divisors in total on the list\edits{.}

Consider the vector $A=(-3,3)$ of Abel-Jacobi data for the
double ramification cycle $\mathsf{DR}_{2,A}$. 
The associated tropical divisor $\mdeg_{k=0,A}$ is given by $\mdeg_{0,A}=3v - 3w$, where in the following we denote divisors on graphs as $\mathbb{Z}$-linear combinations of their vertices.
To determine the subdivision $$\widetilde{\Sigma}^{\theta}_{S} \to \Sigma_S\, ,$$
we must solve the tropical problem 
$$
 \mathsf{div} (f)=
\mdeg_{0,A} - D\, , 
$$  
where $D$ is an admissible $\theta$-semistable divisor on the above list, 
and $f$ is a sPL function on $\Gamma$ (or on $\widehat{\Gamma}$ in the depth 1 and 2 cases).

\vspace{8pt}
\noindent $\bullet$ $\mathsf{Depth}$ 0 cases:
\vspace{8pt}

For the divisor $D=0v + 0w $, there is a unique ray of solutions in $\Sigma_S$. To build a strict piecewise linear function on $\Gamma$, we must choose $3$ slopes $s_1,s_2,s_3 \in \mathbb{Z}$ on the edges of $\Gamma$. If we orient all three edges from $w$ to $v$
(with the convention that $s_i$ is positive if the function increases from $w$ to $v$), then, in order to have a sPL function on $\Gamma$, the condition 
$$
s_1\delta_1 = s_2\delta_2 = s_3\delta_3
$$
on the lengths of the edges of $\Gamma$ must be satisfied. In particular, the $s_i$ must either all be $0$, all positive, or all negative. If, in addition, we demand $$\mathsf{div}(f) = (s_1+s_2+s_3)v +(-s_1-s_2-s_3)w = \mdeg_{0,A} - D = 3v- 3w\, ,$$ we find a unique possible solution 
$$
s_1=s_2=s_3=1
$$
along the ray
$
\delta_1 = \delta_2 = \delta_3\,,
$
and no other solutions.

The divisor $D=-v+w$ can be analyzed similarly. 
The equation 
$$\mathsf{div}(f)  =\mdeg_{0,A} - D = 4v-4w$$ yields three possible solutions $$(s_1,s_2,s_3) = (2,1,1)\,, \ (1,2,1)\, ,\  (1,1,2)$$ 
along the the three rays
\begin{eqnarray*}
\begin{alignedat}{2}
2\delta_1 &= \delta_2 &&= \delta_3\,, \\
\delta_1  &=2\delta_2 &&= \delta_3\,, \\
\delta_1  &= \delta_2 &&= 2\delta_3\, .
\end{alignedat}
\end{eqnarray*}

For the divisor $D=v-w$, there are
no solutions at all to 
$$\mathsf{div}(f)  =\mdeg_{0,A} - D = 2v-2w\, .$$


\vspace{8pt}
\noindent $\bullet$ $\mathsf{Depth}$ 1 cases:
\vspace{8pt}

The graph $\Gamma$ is replaced here with a corresponding
quasi-stable model $\widehat{\Gamma}$.
Instead of studying the sPL function (as we did in the depth 0 cases),
we simply specify the possible underlying slopes of such functions corresponding to acyclic flows $I$ with associated divisor $\mdeg_{0,A}-D$. 

The first depth 1 divisor $D$ on the list contributes three possible acyclic flows (with all edges oriented from right to left and slopes of $I$ in red): 
$$
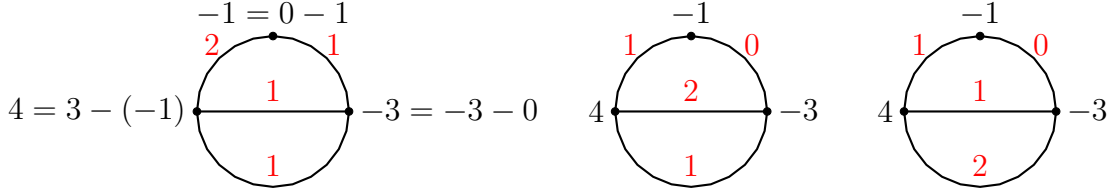
\begin{tikzpicture}
\coordinate(x1) at (-6,0){};
\node[left] at (x1){$4 = 3-(-1)$};
\node[draw,circle,inner sep=1pt,fill] at (x1){};
\coordinate(x2) at (-4,0){};
\node[right] at (x2){$-3 = -3-0$};
\node[draw,circle,inner sep=1pt,fill] at (x2){};
 \draw [thick,domain=0:360] plot ({cos(\x)-5}, {sin(\x)});
\draw[-,thick](x1)--(x2);
\coordinate(x3) at (-5,1){};
\node[above] at (x3){$-1=0-1$};
\node[draw,circle,inner sep=1pt,fill] at (x3){};
\coordinate(x4) at (-5.8,0.6){};
\node[above,red] at (x4){$2$};
\coordinate(x5) at (-4.2,0.6){};
\node[above,red] at (x5){$1$};
\coordinate(x6) at (-5,0);
\node[above,red] at (x6){$1$};
\coordinate(x7) at (-5,-1);
\node[above,red] at (x7){$1$};
\coordinate(y1) at (-0.5,0){};
\node[left] at (y1){$4$};
\node[draw,circle,inner sep=1pt,fill] at (y1){};
\coordinate(y2) at (1.5,0){};
\node[right] at (y2){$-3$};
\node[draw,circle,inner sep=1pt,fill] at (y2){};
 \draw [thick,domain=0:360] plot ({cos(\x)+0.5}, {sin(\x)});
\draw[-,thick](y1)--(y2);
\coordinate(y3) at (0.5,1){};
\node[above] at (y3){$-1$};
\node[draw,circle,inner sep=1pt,fill] at (y3){};
\coordinate(y4) at (-0.3,0.6){};
\node[above,red] at (y4){$1$};
\coordinate(y5) at (1.3,0.6){};
\node[above,red] at (y5){$0$};
\coordinate(y6) at (0.5,0);
\node[above,red] at (y6){$2$};
\coordinate(y7) at (0.5,-1);
\node[above,red] at (y7){$1$};

\coordinate(z1) at (3.3,0){};
\node[left] at (z1){$4$};
\node[draw,circle,inner sep=1pt,fill] at (z1){};
\coordinate(z2) at (5.3,0){};
\node[right] at (z2){$-3$};
\node[draw,circle,inner sep=1pt,fill] at (z2){};
 \draw [thick,domain=0:360] plot ({cos(\x)+4.3}, {sin(\x)});
\draw[-,thick](z1)--(z2);
\coordinate(z3) at (4.3,1){};
\node[above] at (z3){$-1$};
\node[draw,circle,inner sep=1pt,fill] at (z3){};
\coordinate(z4) at (3.5,0.6){};
\node[above,red] at (z4){$1$};
\coordinate(z5) at (5.1,0.6){};
\node[above,red] at (z5){$0$};
\coordinate(z6) at (4.3,0);
\node[above,red] at (z6){$1$};
\coordinate(z7) at (4.3,-1);
\node[above,red] at (z7){$2$};

\end{tikzpicture}
$$

The first of these flows, called  $(\{2,1\},1,1)$ for brevity,
has two specializations obtained by letting the exceptional vertex specialize to either of the original vertices $v$ or $w$: 
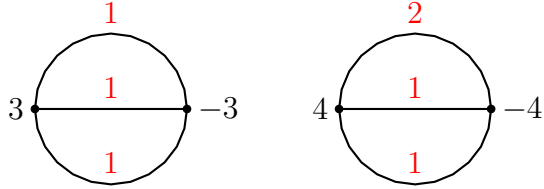
$$
\begin{tikzpicture}
\coordinate(u1) at (-3,0){};
\node[left] at (u1){$3$};
\node[draw,circle,inner sep=1pt,fill] at (u1){};
\coordinate(u2) at (-1,0){};
\node[right] at (u2){$-3$};
\node[draw,circle,inner sep=1pt,fill] at (u2){};
 \draw [thick,domain=0:360] plot ({cos(\x)-2}, {sin(\x)});
\draw[-,thick](u1)--(u2);
\coordinate(u3) at (-2,1){};
\node[above,red] at (u3){$1$};
\coordinate(u4) at (-2,0){};
\node[above,red] at (u4){$1$};
\coordinate(u5) at (-2,-1){};
\node[above,red] at (u5){$1$};

\coordinate(v1) at (1,0){};
\node[left] at (v1){$4$};
\node[draw,circle,inner sep=1pt,fill] at (v1){};
\coordinate(v2) at (3,0){};
\node[right] at (v2){$-4$};
\node[draw,circle,inner sep=1pt,fill] at (v2){};
 \draw [thick,domain=0:360] plot ({cos(\x)+2}, {sin(\x)});
\draw[-,thick](v1)--(v2);
\coordinate(v3) at (2,1){};
\node[above,red] at (v3){$2$};
\coordinate(v4) at (2,0){};
\node[above,red] at (v4){$1$};
\coordinate(v5) at (2,-1){};
\node[above,red] at (v5){$1$};
\end{tikzpicture}
$$
The flow $(\{2,1\},1,1)$ contributes a 2-dimensional cone of solutions to 
$$\mathsf{div}(f)= \mdeg_{0,A} - D  \, $$ 
defined by the convex hull of the solutions of the two specializations (given by  the rays $\delta_1 = \delta_2 = \delta_3$ and $2\delta_1 = \delta_2 = \delta_3$). Taken together, the depth $1$ $\theta$-stable divisors contribute  $12$ two dimensional cones: $9$ from the three possible flows on the first divisor (and
the those related by symmetry) and $3$ from the second.

\vspace{8pt}
\noindent $\bullet$ $\mathsf{Depth}$ 2 cases:
\vspace{8pt}

A similar analysis can be applied to the three depth $2$ divisors. (which contain two exceptional vertices). Each divisor $D$ contributes three $3$-dimensional cones, for a total of $9$. 

\vspace{12pt}
\noindent $\bullet$ $\mathsf{Full}$  $\mathsf{subdivision}$: 
\vspace{8pt}

Altogether, the subdivision $\widetilde{\Sigma}^{\theta}_{S}\to \Sigma_S = \mathbb{R}_{\ge 0}^3$  has $4$ new rays
(corresponding to the depth $0$ divisors), $12$ new 2-dimensional cones (corresponding to depth $1$ divisors), and $9$ maximal cones  (corresponding to the depth $2$ divisors). 
The 2-dimensional intersection of 
the subdivision  $\widetilde{\Sigma}^{\theta}_{S}$
with the hyperplane $\delta_1 + \delta_2 + \delta_3 = 1$
can be drawn as:
$$
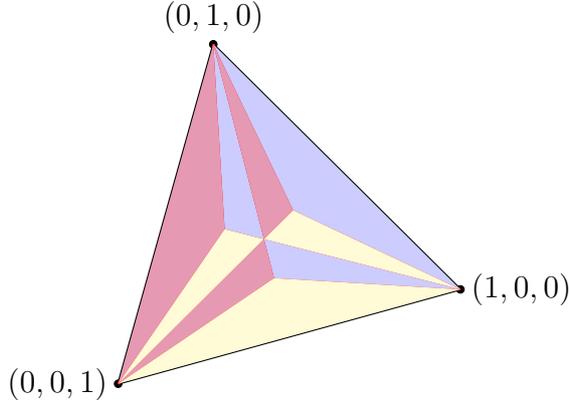
\begin{tikzpicture}[scale=0.65]
\coordinate (v1) at (5,0,0){};
    \node[right] at (v1){$(1,0,0)$};
    \node[draw,circle,inner sep=1pt,fill] at (v1){};
    \coordinate (v2) at (0,5,0){};
    \node[above] at (v2){$(0,1,0)$};
    \node[draw,circle,inner sep = 1pt,fill] at (v2){};
    \coordinate (v3) at (0,0,5){};
    \node[left] at (v3){$(0,0,1)$};  
    \node[draw,circle,inner sep=1pt,fill] at (v3){};
\draw[-,thick](v1)--(v2)--(v3)--(v1);

\coordinate(w1) at (5/3,5/3,5/3){};
\node[right] at (w1){};
    \node[draw,circle,inner sep=1.5pt,fill=green] at (w1){};
\coordinate(w2) at (1,2,2){};
\node[above] at (w2){};
\node[draw,circle,inner sep=1.5pt,fill=green] at (w2){};
\coordinate(w3) at (2,1,2){};
\node[right] at (w3){};
\node[draw,circle,inner sep=1.5pt,fill=green] at (w3){};
\coordinate(w4) at (2,2,1){};
\node[right] at (w4){};
\node[draw,circle,inner sep=1.5pt,fill=green] at (w4){};

\draw[-,thick,red](v1)--(w3)--(w1)--(w4)--(v1)--(w1);
\draw[-,thick,red](v2)--(w2)--(w1)--(w4)--(v2)--(w1);
\draw[-,thick,red](v3)--(w3)--(w1)--(w2)--(v3)--(w1);
\fill[blue!20, nearly transparent] (w1) -- (w3) -- (v1) -- (w1) -- cycle;
\fill[blue!20, nearly transparent] (w1) -- (w2) -- (v2) -- (w1) -- cycle;
\fill[blue!20, nearly transparent] (v1) -- (v2) -- (w4) -- cycle;
\fill[purple!40,nearly transparent](w1) -- (w3) -- (v3)--cycle;
\fill[purple!40,nearly transparent](v2)--(w4)--(w1)--cycle;
\fill[purple!40,nearly transparent](v2)--(v3)--(w2)--cycle;
\fill[yellow!20,nearly transparent](v1)--(v3)--(w3)--cycle;
\fill[yellow!20,nearly transparent](w1)--(w2)--(v3)--cycle;
\fill[yellow!20,nearly transparent](w1)--(w4)--(v1)--cycle;

\end{tikzpicture}
$$
The vertices in green correspond to the depth $0$ divisors, the red lines to the depth $1$ divisors. The maximal cells which correspond to different flows realizing the same underlying divisor 
are shaded in the same color.

\begin{remark}
The geometrical meaning of the subdivision is as follows.
There is the point $s \in S$ (of codimension $3$)  over which the curve $C \to S$ is a dollar sign curve. We blow-up $S$ at $s$, which replaces $s$ with  $\mathbb{P}^2$. We then blow-up the $3$ fixed points of $\mathbb{P}^2$ to obtain $S^\theta_\ca L$.  
\edits{The strata of $S^\theta_\ca L$ then correspond bijectively to the cones in the picture above. The Abel-Jacobi section at a point on the stratum corresponding to a cone lands in the part of the Jacobian parametrising line bundles whose multidegree is given exactly by the corresponding $\theta$-stable multidegree $D$. The pullback of the Abel-Jacobi image to the normalization of the curve $C$ is obtained by twisting the pullback of $\ca L$ to the normalization by multiples of the points lying over the nodes, where these multiples are determined by the flow $f$. The line bundle on $C$ itself (equivalently, the glueing data for reassembling the line bundle on $C$ from the line bundle on the normalization) varies as we move about in the stratum.  }
\end{remark}

\section{Theorem \ref{111ooo}} \label{sec:conditions_for_almost_twistability}


\subsection{Universal constructions}

Let $\mathfrak M_g$ be the stack of log curves, and let $\Picabs$ be the universal Picard stack of degree 0 line bundles on the universal curve over $\mathfrak M_g$, which we equip with the strict log structure. By the main result of \cite{BHPSS},
there is an operational class{\footnote{In Section \ref{intro111} and the reference
\cite{BHPSS}, the more precise notation
$\mathsf{CH}^*_{\mathsf{op}}$ is used for
operational Chow. We will drop the subscript ${\mathsf{op}}$
to simplify the notation (unless needed for emphasis).}}
\begin{equation}\label{ss33ff}
\mathsf{DR} \in \sf{CH}^g(\Picabs)\, .
\end{equation}
By \cite{HolmesSchwarz},
there is a natural lift of \eqref{ss33ff} to a logarithmic class $$\LogDR \in \LogChow^g(\Picabs)\, .$$ For a stack $S$, prestable curve $C/S$ of genus $g$, and line bundle $\ca L$ on $C$ of degree 0, we have a classifying map 
$$\phi_\ca L\colon S \to \Picabs\, .$$ The double ramification cycle is defined in \cite{BHPSS}
as an operational class on $S$ by
\begin{equation*}
    \sf{DR}(\ca L) = \phi_\ca L^*\sf{DR} \in \sf{CH}^g(S)\, . 
\end{equation*}
If $S$ is a log smooth log stack, $C/S$ a log curve of genus $g$, and $\ca L$ a line bundle on $C$ of degree 0,  we have a classifying map (of log stacks)
$$\phi_\ca L \colon S \to \Picabs\, ,$$ 
and we can pull back $\LogDR$. Following 
\cite{HolmesSchwarz}, we define
\begin{equation*}
    \LogDR(\ca L) = \phi_\ca L^*\LogDR \in \sf{logCH}^g(S).  
\end{equation*}

Recall that $A=(a_1,\ldots, a_n)$ is a vector
of integers summing to $k(2g - 2 + n)$. 
Let $C/S$ be the universal curve $\mathcal{C} \to \Mbar_{g,n}$
over the moduli space of stable curves
with markings $x_1,\ldots,x_n$, 
and let 
\begin{equation*}
    \ca L = (\omega_C^\log)^{\otimes k}\left(- \sum_{i=1}^n a_i x_i\right)\, .
\end{equation*}
By \cite{BHPSS} and \cite{HolmesSchwarz}, the above universal constructions are
{\em both} compatible with the previous definitions
of the
double ramification cycle:
\begin{eqnarray*}
    \sf{DR}(\ca L) &=& \sf{DR}_{g, A} \in \sf{CH}^g(\Mbar_{g,n})\, , \\
    \LogDR(\ca L) &=& \LogDR_{g, A} \in \sf{logCH}^g(\Mbar_{g,n})\, .  
\end{eqnarray*}

\subsection{Almost twistability} 

Let $\SSS$ be a log smooth log algebraic stack, let $\CCC/\SSS$ be a log curve of genus $g$, and let $\LLL$ be a line bundle on $\CCC$ of degree 0. Let 
$$\ca J_{\CCC/\SSS} \to \SSS$$ be the multidegree $\ul 0$ part of the relative Picard stack
of $\CCC/\SSS$.

\begin{definition}\label{def:almost_twistable}
The pair $(\CCC/\SSS, \LLL)$ is \emph{almost twistable} if there exists a dense open $i \colon U \hra \SSS$ which
satisfies the following two conditions:
\begin{enumerate}
\item[(i)]
the line bundle $\LLL$ has multidegree $\ul 0$ over $U$, 
\item[(ii)] the map $U \xrightarrow{\varphi_{\mathcal{L}}\circ i } {\mathcal{J}}_{\CCC/\SSS}$ is a closed immersion 
(or, equivalently, the image of
$\varphi_{\mathcal{L}}\circ i $
is closed in ${\mathcal{J}}_{\CCC/\SSS}$). 
\end{enumerate}
\end{definition}

Definition \ref{def:almost_twistable}
is the specialization of
\cite[Definition 4.10]{HolmesSchwarz} obtained
by setting the strict piecewise linear function $\alpha$ there
to be zero. \edits{Indeed, under this specialization, Condition (i) of \cite[Definition 4.11]{HolmesSchwarz} that $\alpha=0$ is a twisting function on $U$ just says that $\LLL$ already has multidegree $0$ there. On the other hand, the map $U \to {\mathcal{J}}_{\CCC/\SSS}$ is always a locally-closed immersion, and Condition (ii) of \cite[Definition 4.11]{HolmesSchwarz} is the valuative criterion for closedness}. The motivation for Definition \ref{def:almost_twistable} lies in the following result.

\begin{proposition}\label{lem:invariance_7_b}
Let $(\CCC/\SSS, \LLL)$ be almost twistable.  Let $\phi_\LLL\colon \SSS \to \Picabs$ be the map induced by $\ca L$. Then,
\begin{equation*}
\phi_{\LLL}^*\LogDR = \phi_{\LLL} ^* \sf{DR} 
\end{equation*} 
in $\LogChow^g(\SSS)$. 
\end{proposition}
\begin{proof}
When $\SSS$ is smooth, the claim is a special case of \cite[Lemma 4.13]{HolmesSchwarz}. However, the proof given there does not use smoothness as we work throughout with operational classes (the smoothness assumption is present in \cite{HolmesSchwarz} only as a running assumption which simplifies some other parts of the paper). 
\end{proof}



\subsection{Category of twists}

Let $S$ be a log smooth log algebraic stack, let $C/S$ \edits{be} a log curve of genus $g$, and let $\ca L$ be a line bundle on $C$ of degree 0.
We consider the category $\cat{Twist}(S)$ of tuples
\begin{equation*}
(T/S, \widehat{C} \to C_T, \alpha)
\end{equation*}
where $T/S$ is a log scheme, $\widehat{C} \to C_T$ is a quasi-stable model of $C_T/T$, and $\alpha$ is a sPL function on $\widehat{C}$. Let 
$$\cat{Twist}^\ul 0(S,\ca L) \hra \cat{Twist}(S)$$
be the open substack consisting of those objects where 
$$\widehat{C} \isom  C_T$$
and $\ca L_T(\alpha)$ has multidegree $\ul 0$ on $C_T/T$. 


If $T$ is a trait (the spectrum of a discrete valuation ring) with generic point $j\colon \eta \to T$, the log structure on $T$ is \emph{maximal} if the natural map 
$$\M_T \to \ca O_T \times_{j_*\ca O_\eta} j_*\M_\eta$$
is an isomorphism. 
For example, if $\M_{\eta} = \mathcal{O}_{\eta}^*$, then $\ghost_{T}(T) = \bb N$, and $\mathsf{sPL}(C)$ is canonically isomorphic to the group of Cartier divisors on $C$ supported over the closed point of $T$. \edits{
As explained in \cite[\S 3.5]{Marcus2017Logarithmic-com}, when checking the valuative criterion for properness in the logarithmic setting, test valuation rings should be restricted to those with maximal log structure. 
}
\begin{lemma}\label{lem:extend_iso}
Let $T$ be a trait with generic point $\eta$ and maximal log structure. Let $C/T$ be a log curve with line bundles $\ca F$, $\ca F'$ on $C$, and $\alpha \in \ghost_{C}^\gp(C_\eta)$ an sPL function, together with an isomorphism 
$$\psi\colon \ca F_\eta(\alpha) \to \ca F'_\eta\, .$$ 
Then there exists an sPL function $\bar\alpha \in \ghost_{C}^\gp(C)$ on $C$ restricting to $\alpha$ over $\eta$ and such that $\psi$ extends to an isomorphism $\ca F(\bar\alpha) \to \ca F'$. 
\end{lemma}
\begin{remark}
    \edits{The proof is an application of the separatedness of the logarithmic Picard stack of \cite{Molcho2018The-logarithmic}. To a first approximation, log line bundles can be seen as equivalence classes of line bundles under the action of twisting by piecewise linear functions. As such, separatedness of the moduli of log line bundles means that, if two line bundles are equivalent modulo piecewise-linear functions over $\eta$, then they are so over the whole of $T$. }
\end{remark}
\begin{proof}[\edits{Proof of Lemma \ref{lem:extend_iso}}] 
The line bundles $\ca F, \ca F'$ naturally correspond to $\mathcal{O}_C^*$-torsors on $C$. By extending the structure group, we obtain two $\M_C^\gp$-torsors $\ca F^{\textup{log}}, \ca F'^{\textup{log}}$. 
By the exact sequence
\begin{equation*}
\pi_*\ghost_C^\gp \to R^1\pi_*\ca O_C \to R^1\pi_*M_C^\gp
\end{equation*}
these agree on the generic point of $T$, and hence they agree on $T$ by \cite[Theorem 4.10.1]{Molcho2018The-logarithmic}. By another application of the same exact sequence
we see that the group of line bundles on $C$ whose associated $\M_C^\gp$ torsor is trivial is the image of $\ghost_{C}^\gp(C)$, hence there exists $\beta \in \ghost_{C}^\gp(C)$ and an isomorphism 
\begin{equation*}
    \ca F(\beta) \isom \ca F'. 
\end{equation*}
Then, $\ca O_{C_\eta}(\alpha - \beta_\eta)$ is trivial, hence $\alpha - \beta_\eta$ is constant, so there exists $\gamma \in \ghost_T^\gp(\eta)$ with $$\alpha - \beta_\eta = \pi^*\gamma\, . $$
By maximality of the log structure, there exists $\bar\gamma \in \ghost_T^\gp(T)$ restricting to $\gamma$. Setting $\bar\alpha = \beta + \pi^*\bar\gamma$ gives the result. 
\end{proof}


\begin{proposition}\label{prop:almost_twistable_condition}Let $X \hra \cat{Twist}(S)$ be an open substack containing $\cat{Twist}^\ul 0(S,\ca L)$ as a dense open. \edits{Write $C_X$ for the pullback of $C/S$ under $X \to S$, and write $\widehat{C}$ and $\alpha$ for the universal objects over $X$ determined by the map $X \to \cat{Twist}(S)$. } If $X$ is separated, 
then $(\widehat{C}/X, \ca L(\alpha))$ is an almost twistable family. 
\end{proposition}

\begin{proof}
Note that 
$${\mathcal{J}}_{\widehat{C}/X} = {\mathcal{J}}_{C_X/X}\, ,$$
\edits{since the curves differ by the insertion of rational bubbles. }
Let $U = \cat{Twist}^\ul 0(S,\ca L) \hra X$, so $\ca L(\alpha)$ defines a monomorphism 
$$\edits{U \to {\mathcal{J}}_{C_X/X}\, ,}$$
which we will show to be proper by the valuative criterion.

Let $T$ be a trait with generic point $\eta$, and fix a map 
$$t\colon T \to {\mathcal{J}}_{C_X/X}$$
such that $\eta$ lands in the image of $U$. The map $t$ corresponds to a line bundle $\ca F'$ on 
$C_T/T$ of multidegree $\ul 0$. Since $\ca F'_\eta$ lies in the image of $\cat{Twist}^\ul 0(S,\ca L)$, there exist an sPL function $\alpha$ on $C_\eta$ and an isomorphism $$\psi\colon \ca L_\eta(\alpha) \to \ca F'.$$
By Lemma \ref{lem:extend_iso}, we can extend the sPL function $\alpha$ and isomorphism $\psi$ over the whole of $T$, defining an object 
$$(T/S, C_T \isom C_T, \alpha)$$ of $\cat{Twist}^\ul 0(S,\ca L)$. The line bundle $\ca L(\alpha)$ induces a map $t'\colon T \to \ca J_{C_T/T}$, which agrees with $t$ at $\eta$, and hence agrees with $t$ by the separatedness of $X$. 
\end{proof}

\subsection{Proof of Theorem \ref{111ooo}}

To prove Theorem \ref{111ooo}, 
we will apply Propositions \ref{lem:invariance_7_b} and \ref{prop:almost_twistable_condition}  under the
following specialization of the geometry:
\begin{enumerate}
\item[$\bullet$] $C/S$ is the universal curve $\mathcal{C} \to \Mbar_{g,n}$
over the moduli space of stable curves,

\item[$\bullet$] $\mathcal{L}$ is the line bundle of total degree 0
on $\mathcal{C}/ \Mbar_{g,n}$
defined by
$$\ca L = (\omega_C^\log)^{\otimes k}\left(- \sum_{i=1}^n a_i x_i\right)\, ,$$
\item[$\bullet$]
$\SSS = \Mbar_{g,A}^\theta\stackrel{\rho}{\longrightarrow} \Mbar_{g,n}$ 
for a small nondegenerate stability condition $\theta$,  
\item[$\bullet$]
$\widehat{\mathcal{L}}= \mathcal{L}^{\theta}$ is the universal line bundle on the universal quasi-stable curve
$$\CCC= \mathcal{C}^\theta \to \Mbar_{g,A}^\theta\, ,$$
\item[$\bullet$] $\alpha=\alpha^\theta$ is the universal sPL function 
vanishing on $x_1$ on
$\mathcal{C}^\theta$ and
satisfying $$\ca L^\theta = \ca L( \alpha^\theta)\, .$$
\end{enumerate}



\begin{proposition}
\label{small_nond_implies_a_twistable}
For  a  small nondegenerate  stability condition $\theta$, 
the line bundle 
$\mathcal{L}^\theta$ on $\mathcal{C}^\theta \to \oM_{g,A}^\theta$ is almost twistable.  
\end{proposition}

\begin{proof}

There is a natural open immersion 
$$\oM_{g,A}^\theta \to \cat{Twist}(\oM_{g,n})$$
induced by the subdivision $\ca C^\theta$ and the sPL function $\alpha^\theta$, identifying $\oM_{g,A}^\theta$ with the open substack $X$ of $(\widehat{C}, \alpha) \in \cat{Twist}(\oM_{g,n})$ where $\ca L|_{\widehat{C}}(\alpha)$ is $\theta$-stable.
We must verify the conditions required by  Proposition \ref{prop:almost_twistable_condition} for the open set $X$:

\begin{enumerate}
\item[$\bullet$] Since multidegree $\ul 0$ line bundles on stable models are $\theta$-stable, 
$X$ contains $\cat{Twist}^{\ul 0}(  \Mbar_{g,n}    ,\ca L)$ as a dense open. 
\item[$\bullet$] Separatedness of $X$ follows from Theorem \ref{thm:is_subdivis}: since \edits{$X = \oM_{g,A}^\theta \xrightarrow{\rho} \oM_{g,n}$} is proper,  $X$ is also proper. The key input is the nondegeneracy of $\theta$. 
\end{enumerate}
Since the conditions of Proposition \ref{prop:almost_twistable_condition} are satisfied,
we conclude that $(\mathcal{C}^\theta/ \oM_{g,A}^\theta, \ca L^\theta)$
is an almost twistable family.
\end{proof}

To complete the proof of 
Theorem \ref{111ooo}, the class 
$$
\phi_{\mathcal{L}^\theta} ^* \sf{DR}  =\mathsf{DR}^{\mathsf{op}}_{g,\emptyset, \mathcal{L}^\theta} \in \mathsf{CH}^g_{\mathsf{op}}( \oM_{g,A}^\theta)\, 
$$
must be shown to represent $\mathsf{logDR}_{g,A} \in \mathsf{logCH}^g(\Mbar_{g,n})$.
We apply Proposition \ref{lem:invariance_7_b}
to the family $(\mathcal{C}^\theta/ \oM_{g,A}^\theta, \ca L^\theta)$
to conclude
\begin{equation*}
\phi_{\mathcal{L}^\theta}^*\LogDR = \phi_{\mathcal{L}^\theta} ^* \sf{DR} 
\end{equation*} 
The last step is to prove that 
$$
\phi_{\mathcal{L}^\theta} ^* \sf{logDR} \in \mathsf{CH}^g_{\mathsf{op}}( \oM_{g,A}^\theta)\, 
$$
represents $\mathsf{logDR}_{g,A} \in \mathsf{logCH}^g(\Mbar_{g,n})$.

The claim of the last step is true because of invariance
properties of the logarithmic double ramification
cycle \cite{HolmesSchwarz}:
\begin{enumerate}
\item[(i)]  If $f:C' \to C$ over $S$ is a subdivision of log curves, then
$$\mathsf{logDR}(f^*\ca L) = \mathsf{logDR}(\ca L)\, ,$$
\item[(ii)] If $\beta$ is a sPL function on $C$ then
$$\mathsf{logDR}(\ca L(\beta)) = \mathsf{logDR}(\ca L)\, ,$$
\item[(iii)]
If $f: S' \to S$ is a log modification, then $\mathsf{logCH}^*(S') = \mathsf{logCH}^*(S)$ and
$$\mathsf{logDR}(f^*\ca L) = \mathsf{logDR}(\ca L)\, .$$
\end{enumerate}
Invariances (i) and (ii) are proven
in \cite[Lemma 4.13]{HolmesSchwarz} and
\cite[Theorem 4.18]{HolmesSchwarz}.
Invariance (iii) is by the construction 
of $\mathsf{logDR}$.

\edits{We compute
\begin{equation*}
\begin{split}
    \mathsf{logDR}(\ca L^\theta) & = \mathsf{logDR}(\ca L|_{C^\theta}) \\
    & = \mathsf{logDR}(\ca L|_{C \times_{\Mbar_{g, n}}{\Mbar^\theta_{g, A}}}) \\
    & = \mathsf{logDR}(\ca L) 
\end{split}
\end{equation*}
by Invariances (ii), (i), and (iii) respectively.  }


\section{Piecewise polynomials and strata classes}\label{PPsclass}
\subsection{Overview}
The $\mathbb{Q}$-algebras
of piecewise polynomials and strict piecewise polynomials associated to $\Mbar_{g,n}$ were defined in 
Section \ref{Sect:logtautclasses} 
and used there to define the logarithmic tautological ring
$$\mathsf{logR}^*(\oM_{g,n}) \subset \mathsf{logCH}^*(\oM_{g,n})\, .$$
In order to derive the formulas of Theorem
\ref{222ttt} from Theorem \ref{111ooo},
 an explicit correspondence between the normally decorated strata classes of \cite{Molcho2021-Hodgebundle} and strict piecewise polynomials\footnote{Warning to the reader: the piecewise polynomials of \cite{Molcho2021-Hodgebundle} are the strict piecewise polynomials here.}
 is required. 
The explicit correspondence is
established in Proposition \ref{lem:comparing_graph_sum_and_PP} of Section \ref{sec: pp}.
In Section \ref{sec:classical_DR_in_PP_language}, we rewrite
the formula from \cite{BHPSS}  for the universal double
ramification cycle 
in the language of piecewise polynomials in preparation for the proof of Theorem \ref{222ttt} in 
Section \ref{pprr22}.





\subsection{Relating piecewise polynomials and graph sums}
\label{sec: pp}
\subsubsection{Normally decorated strata classes}
Let $(X,D)$ be a nonsingular algebraic stack equipped with a normal crossings divisor, viewed as a log stack. 
We begin by briefly recalling the definition of \emph{normally decorated strata classes} 
for $(X,D)$
from \cite[\S 5.1]{Molcho2021-Hodgebundle}.

Let $S$ be a codimension $k$ stratum in $X$, and let $B_S$ be the set of branches of $D$ through $S$ (so $|B_S| = k$). Let $$\epsilon: \widetilde{S} \to X$$ be the normalization of the closure of $S$,  let $$p:P \to \widetilde{S}$$ be the $G$-torsor over $\widetilde{S}$ determined by the monodromy group $G$ of Definition \ref{monodef} acting on the set $B_S$, and  let
$$j = \epsilon \circ p: P \to X$$ denote the composition.
Since the map $\epsilon$ is unramified,  we can define the normal bundle $N_\epsilon$ which splits when pulled back to $P$:
$$
p^*N_\epsilon = \oplus_{b \in B_S} N_b\, . 
$$ 
The assignment $b \mapsto c_1(N_b)$ extends to 
a $\mathbb{Q}$-algebra homomorphism
\begin{equation*}
c_1(N)\colon \bb Q[B_S] \to \sf{CH}^*(P)
\end{equation*}
\edits{where $\bb Q[B_S]$ is a polynomial ring in variables indexed by $B_S$}. Given $F \in \bb Q[B_S]$, let $$F(c_1(N))\in \sf{CH}^*(P)$$ be the image under $c_1(N)$. 

\begin{definition}
A \emph{normally decorated strata} class of $(X,D)$ 
is a class of the form 
$$j_*F(c_1(N)) \in \mathsf{CH}^*(X)\, $$
where  $F \in \bb Q[B_S]$ is a polynomial in the branches of $D$ through a stratum $S$.  
\end{definition}

\edits{
\begin{remark}
    Normally decorated strata are defined using the monodromy torsors 
    of \cite{Molcho2021-Hodgebundle}, which are $G$-torsors over the normalizations of the strata of $X$ for the monodromy group $G$. In other words, the monodromy torsors are connected to the automorphism group of the corresponding stacky cone in the \emph{canonical} tropicalization of $(X,D)$. It is possible to construct monodromy torsors that resolve the automorphism group of the stacky cone in an arbitrary tropicalization and develop a more flexible theory of normally decorated strata classes, but
    we will not pursue
    these directions here.
\end{remark}
}

\subsubsection{Piecewise polynomials}
Let $(X,D)$ be a nonsingular algebraic stack equipped with a normal crossings divisor, viewed as a log stack. \edits{For the rest of this section, we will work with the \emph{canonical} tropicalization $X \to \mathcal{A}_X$.} The combinatorial definition of (strict) piecewise polynomials given in Section \ref{Sect:logtautclasses}  for $\Mbar_{g,n}$ carries over unchanged to $(X,D)$. See \cite{HolmesSchwarz, Molcho2021-Hodgebundle, Molcho2021-Case-Study} for 
detailed foundations.
We will describe the piecewise polynomial of $(X,D)$ corresponding to a
normally decorated strata class.


Let $\Sigma_X$ denote the \edits{canonical} fan of $(X,D)$, and let $\sigma$ be the stacky cone corresponding to a stratum $S\subset X$. The stacky cone $\sigma$ has a strict cover by the cone $\tau = \mathbb{R}_{\ge 0}^{B_S}$, in the sense of \cite[Definition 2.5]{Cavalieri2020-Conestack}. In particular, the interior of $\sigma$ is a quotient of the interior of $\tau$ by the monodromy group $G$. Let  $$\mathfrak p\colon \tau \to \sigma\ \ \ \ \text{and} \ \ \ \ \mathfrak j: \tau \to \Sigma_X$$
denote the  quotient and the composition of the quotient with the inclusion. 
The set $\tau(1)$ of rays in $\tau$ is naturally identified with the set $B_S$ of branches, so there
is an identification of  $\bb Q$-algebras
\begin{equation*}
 \sf{sPP}(\tau) = \bb Q[B_S]\, . 
\end{equation*}
The monodromy group $G$ acts naturally and compatibly
on $B_S$, hence on $\tau$, on $\bb Q[B_S]$, and on $\sf{sPP}(\tau)$. 
%
%


%

Pullback of strict piecewise polynomials yields an injection
\begin{equation}\label{eq:sPP_embedding}
\mathfrak{i}: \sf{sPP}(\sigma) \to \sf{sPP}(\tau)\, .
\end{equation}
The image of $\mathfrak{i}$ is contained in the $G$-invariant polynomials and contains those $G$-invariant polynomials 
which vanish on the boundary of $\tau$; equivalently, the 
$G$-invariant polynomials which  are divisible by 
$$\delta_\tau = \prod_{b \in B_S} b\, .$$
We also define a map \edits{in the other direction} 
\begin{equation}\label{pFpF}
\edits{\mathfrak{j}_0}\colon \sf{sPP}(\tau) \to \sf{sPP}(\sigma)\, , \ \ \  F \mapsto \delta_\tau \sum_{g \in G} g^*F.
\end{equation}

\edits{As an illustration, take $\tau = \mathbb{R}_{\geq 0}^{2}$ with an action of $G=\mathbb{Z}/2\mathbb{Z}$ generated by $(x,y) \mapsto (y,x)$. For the projection $\mathfrak{p} : \tau \to \sigma = [\tau / G]$ we have $\sPPoly{\tau}=\mathbb{Q}[x,y]$ and $\sPPoly{\sigma}=\mathbb{Q}[x+y, x\cdot y]$. Then $\mathfrak{i} : \mathbb{Q}[x+y, x\cdot y] \to \mathbb{Q}[x,y]$ is the inclusion and 
$$\mathfrak{j}_0 : \mathbb{Q}[x,y] \to \mathbb{Q}[x+y, x \cdot y], f(x,y) \mapsto xy \cdot (f(x,y)+f(y,x)),$$
so $\mathfrak{j}_0(1) = 2xy$.
Geometrically, this example is related to the punctured Whitney umbrella. }

\subsubsection{Extension}
 Let $\sigma$ be the stacky cone, with strict cover $\tau$, corresponding to a
 codimension $k$ stratum $S\subset X$.
 We describe how to canonically extend the polynomial function $\mathfrak{j}_0 F$ defined by
 \eqref{pFpF}
 from $\sigma$ to the whole fan $\Sigma_X$. 
 We will construct a nontrivial extension of $\mathfrak{j}_0 F$ over cones which contain $\sigma$.
For cones not containing $\sigma$, the extension will be 0.

 Let $\sigma'\supset
 \sigma$ be any stacky cone containing $\sigma$.
 As before, $\sigma'$ is a quotient of a cone 
 $\tau' = \mathbb{R}_{\ge 0}^{k'}$ for $k' \ge k$ with perhaps additional faces identified,
 $$\mathfrak{p}:\tau' \to \sigma'\, .$$
  The map $\mathfrak{p}$ sends some of the $k$-dimensional faces of $\tau'$
 onto $\sigma$ and the rest  onto faces of $\sigma'$ different from $\sigma$.
 
 Let $\mathcal{F}_k$ be the 
 set of the $k$-dimensional faces of the fiber product $\tau \times_{\sigma'} \tau'$. Explicitly, $\mathcal{F}_k$ is the set of pairs $(\widetilde \tau,\gamma)$ where $\widetilde \tau$ is a face of $\tau'$ and $\gamma$ is an isomorphism  fitting into the diagram
\[
\begin{tikzcd}
\widetilde\tau \ar[r,"\subset"] \ar[d,"\gamma"] &\tau' \ar[d,"\mathfrak{p}"] \\
\tau \ar[r] & \, \sigma'.  
\end{tikzcd}
\] 
Because the image of $\tau$ in $\sigma'$ is generically the quotient $\sigma \cong \tau/G$, the group $G$ acts freely on the isomorphisms $$\gamma\colon \widetilde{\tau} \to \tau$$ commuting with the map $\tau \to \sigma'$,  so  $\mathcal{F}_k$ is a disjoint union of $G$-torsors. 


The standard variables of the
polynomials on $\tau'$ 
are in bijective correspondence
to
the set of rays $\tau'(1)$.
Let
$$D^{n}(\tau') 
\subset \sf{sPP}(\tau')$$
be the 
 $\mathbb{Q}$-linear subspace
 spanned by all monomials 
 in at most $n$ different variables.
For example, if $|\tau'(1)|=3$
with variables $x_1,x_2,x_3$, then
$$2x_1^{9}x_2^3+7x_2x_3\in 
D^2(\tau') \ \ \ \text{and} \ \ \   x_1x_2x_3\notin D^2(\tau')\, .$$
Elements in $D^{n}(\tau')$ are determined by their values on all of the $n$-dimensional faces of $\tau'$.

The subspaces $D^n(\tau')$
define an increasing filtration of $\sf{sPP}(\tau')$ as $n$ increases. \edits{Pullback of piecewise polynomials by $\mathfrak{p}$ identifies $\sf{sPP}(\sigma')$ with a subring of $\sf{sPP}(\tau')$. Thus, we can restrict the filtration from $\sf{PP(\tau')}$ to obtain a filtration,} denoted $D^n(\sigma')$, on $\sf{sPP}(\sigma')$.


\begin{proposition}
Let $F \in \sf{sPP}(\tau)$. There is a unique polynomial $$\mathfrak j^{\sigma'}_*F \in D^k(\sigma')$$ which agrees with $\mathfrak{j}_0 F$ on $\sigma$ and is $0$ on all other $k$-dimensional faces of $\sigma'$. 
\end{proposition}

\begin{proof}
The uniqueness claim
is immediate since the function is in $D^k$ and its values on all of the $k$-dimensional faces of $\sigma'$ are
specified. 

We will construct the desired function on $\tau'$  and
then argue that it descends to $\sigma'$. For  $(\widetilde \tau,\gamma) \in \mathcal{F}_k$, the polynomial $\delta_{\widetilde \tau}\gamma^*F \in \sf{sPP}(\widetilde \tau)$ vanishes on the boundary of $\widetilde \tau$, and we write $\overline{\delta_{\widetilde \tau}\gamma^*F }\in \sf{sPP}(\tau')$ for the basic{\footnote{The {\em basic extension} of
a polynomial on a face of $\mathbb{R}_{\geq 0}^n$ is by pullback
via the canonical projection $\mathbb{R}^n_{\geq 0} \to \mathbb{R}^m_{\geq 0}$ to the face.}} extension.
Define
$$
\mathfrak{j}^{\sigma'}_* F = \sum_{(\widetilde \tau,\gamma) \in \mathcal{F}_k} \overline{\delta_{\widetilde \tau} \gamma^*F} \in \sf{sPP}(\tau'). 
$$
Certainly, $\mathfrak j^{\sigma'}_*F \in D^k(\tau')$.

We claim that $\mathfrak j^{\sigma'}_*F$ lies in the image of $\sf{sPP}(\sigma') \hookrightarrow \sf{sPP}(\tau')$. 
Let $G'$ be the monodromy group of $\sigma'$ which
acts on $\tau'$. The group $G'$ also 
acts on the set of $k$-dimensional faces of $\tau'$. \edits{Since $\tau' \to \sigma'$ is $G'$-invariant}, if a $k$-dimensional face $\kappa \subset \tau'$ is taken to another face $\kappa'$ by  $g' \in G'$, and the image of $\kappa$ in $\sigma'$ is $\sigma$, then  the image of $\kappa'$ is also $\sigma$. \edits{Furthermore,} $G'$ 
permutes all $k$-dimensional faces that do not surject to $\sigma$. 
Therefore, $\mathfrak{j}^{\sigma'}_*F$ is invariant with respect to the $G'$-action. Furthermore, since 
$\mathfrak{j}^{\sigma'}_*F$
vanishes on the boundary of every $k$-dimensional cone and
is symmetric with respect to any additional identification of faces, 
$\mathfrak{j}^{\sigma'}_*F$ lies
in the image of $\sf{sPP}(\sigma') \hookrightarrow \sf{sPP}(\tau')$. 

Finally, we must check that $\mathfrak j^{\sigma'}_*F$ agrees with $\mathfrak{j}_0 F$ on $\sigma$ and is $0$ on all other $k$-dimensional faces of $\sigma'$. The restriction of the
function $\mathfrak j^{\sigma'}_*F$ to $\sigma$ is determined by the $k$-dimensional faces that surject onto $\sigma$. Choose such a face $\tau_0$ of $\tau'$. For any $(\widetilde{\tau},\gamma) \in \mathcal{F}_k$ with $\widetilde{\tau} \neq \tau_0$, the function  $\overline{\delta_{\widetilde \tau}\gamma^*F}$ vanishes on $\tau_0$ by construction. Thus, since the isomorphisms $\gamma: \tau_0 \to \tau$ in $\mathcal{F}_k$ can be identified with elements of $g \in G$, we have
$$
\mathfrak{j}^{\sigma'}_* F|_{\tau_0} = \sum_{g \in G} \overline{\delta_{\tau}g^*F}  = \mathfrak p_*F\, , 
$$
which is the required agreement.
\end{proof}

%

\begin{definition}
Given $F \in \sf{sPP}(\tau)$, we define $\mathfrak j_*F \in \sf{sPP}(\Sigma_X)$ to be $\mathfrak j_*^{\sigma'}F$ on those stacky cones $\sigma'$ of $\Sigma_X$ 
which contain $\sigma$ as a face  and $0$ on the other stacky cones of $\Sigma_X$. 
\end{definition}



\begin{example} \label{Exa:jpushforward}
For $(X,D) = (\oM_{g,n}, \partial \oM_{g,n})$, consider the cone $\sigma_{\Gamma_0} \subseteq \Sigma_X$ associated to the stable graph $\Gamma_0$ with one vertex of genus $0$ and precisely $g$ loops. Then $\tau=\mathbb{R}_{\geq 0}^g$ carries the strict piecewise polynomial $F=1$.  Then, the 
function{\footnote{The function was explained to us by D. Ranganathan.}} 
$$\mathfrak j_*F \in \sf{sPP}(\Sigma_X)$$
is given as follows:
\begin{itemize}
    \item on cones $\sigma_\Gamma$ where $\Gamma$ has a vertex of positive genus, the function $\mathfrak j_*F$ vanishes (since these do not contain $\sigma_{\Gamma_0}$),
    \item on the other cones $\sigma_\Gamma$, where $\Gamma$ has first Betti number equal to $g$, we have
    $$
    \mathfrak j_*F = g! \cdot \sum_{\substack{T \subseteq \Gamma\\\text{spanning tree}}} \prod_{e \in E(\Gamma) \setminus E(T)} \ell_e\,.
    $$
    Indeed, in this case the set $\mathcal{F}_g$ corresponds to the set of graph morphisms $\Gamma \to \Gamma_0$\edits{, modulo the equivalence relation identifying morphisms which differ by flipping some of the loops (see Example \ref{Exa:Mbar_monodromy_first})}. Every such morphism precisely contracts a spanning tree $T \subseteq \Gamma$, and for a given tree $T$ there are $g!$ such \emph{equivalence classes of} morphisms. 
    \edits{Indeed, these just correspond to the $g!$ bijections from the $g$ edges $f$ of $\Gamma_0$ to the noncontracted edges $e \in E(\Gamma) \setminus E(T)$ of $\Gamma$. }
    \qedhere
\end{itemize}
\end{example}

\subsubsection{Commutation}
The
explicit correspondence  of Section \ref{excor}  
requires the following basic commutation result.

\begin{lemma}\label{lem:compare_decorated_strata_and_sPP}
Let $S$ be a stratum of $X$ of codimension $k$ 
with branch set $B_S$ and 
 monodromy torsor  $$p:P \to \widetilde{S}\, .$$
 Then the following diagram commutes:
\begin{equation}\label{dddiii}
 \begin{tikzcd}
  \bb Q[B_S] \arrow[r, "\mathfrak j_*"]\arrow[d, "c_1(N)"]  & \sf{sPP}(\Sigma_X) \arrow[d, "\Phi"] \\
  \sf{CH}^*(P) \arrow[r, "j_*"] & \sf{CH}^*(X). 
\end{tikzcd}
\end{equation}
%
%
%
%
%
\end{lemma}

\begin{proof}\leavevmode
The argument is carried out in three steps.

\vspace{10pt}
\noindent \textbf{Step I: The simple normal crossings case. }
\vspace{10pt}

We first treat the case where $(X,D)$ is a scheme with simple normal crossings and $\bar{S}$ is the intersection of the irreducible components $D_b$ corresponding to the branches $b \in B_S$. 
The closure $\bar{S}$ is then normal, the monodromy group $G$ is trivial, and $N_b = \mathcal{O}(D_b)|_{\bar{S}}$. 
The commutation of \eqref{dddiii} is then easily
verified:
\begin{align*}
j_*F(c_1(N))&= j_*j^*F(D_b : b \in B_S) = [\bar S]\cdot  F(D_b : b \in B_S) \\&= \left(\prod_{b \in B_S} D_b \right) \cdot  F(D_b : b \in B_S)  = \Phi(\mathfrak{j}_* F)\,. 
\end{align*}

\vspace{10pt}
\noindent \textbf{Step II: Reduction of
the general case to $X = \ca A_{\sigma'}$ a stacky Artin cone. }
\vspace{10pt}

We can move the issue to the Artin fan $\mathcal{A}_X$ of $X$. There is a Cartesian diagram
\[
\begin{tikzcd}
P \ar[r,"j"] \ar[d,"t_P"] & X \ar[d,"t_X"] \\ \mathcal{P} \ar[r, "\mathcal{A}_j"] & \mathcal{A}_X 
\end{tikzcd}
\]
where $t_X$ is smooth, and $\mathcal{P}$ is the monodromy torsor over the normalization of the closure of the stratum $t_X(S) \subset \mathcal{A}_X$. Furthermore,
$$
N_{\mathcal{P}/\mathcal{A}_X} = \oplus_{b \in B_S} \mathcal{N}_b
$$
and $N_{P/X} = t_P^*N_{\mathcal{P}/\mathcal{A}_X}$ compatibly with the splittings of $N_{P/X}$ and $N_{\mathcal{P}/\mathcal{A}_X}$. Then,
$$
j_*(F(c_1(N)) = j_*(t_P^*F(c_1(\mathcal{N})) = t_X^* \mathcal{A}_{j*}F(c_1(\mathcal{N})). 
$$
So, instead of proving the commutation of \eqref{dddiii},
we may instead prove the commutation of:
\begin{equation*}
 \begin{tikzcd}
  \bb Q[B_S] \arrow[r, "\mathfrak j_*"]\arrow[d, "c_1(N)"]  & \sf{sPP}(\Sigma_X) \arrow[d, "\Phi"] \\
  \sf{CH}^*(\mathcal{P}) \arrow[r, "\mathcal{A}_{j_*}"] & \sf{CH}^*(\mathcal{A}_X). 
\end{tikzcd}
\end{equation*}
Furthermore, for the Artin fan, 
$$
\Phi\colon \mathsf{sPP}(\Sigma_X) \to \mathsf{CH}^*(\mathcal{A}_X) 
$$   
is an isomorphism \cite{Molcho2021-Hodgebundle}, so we
drop $\Phi$ from the notation. 

Because the Chow ring of an Artin fan satisfies \'etale descent (see Lemma \ref{lem:chow_is_sheaf_on_AF} below), we may prove the desired equalities \'etale locally on $X = \mathcal{A}_X$. Since passing to a Zariski open does not affect the monodromy or the normal bundle, we may further assume that $$X = \mathcal{A}_{\sigma'}\, ,$$ 
where $\Sigma_X = \sigma'$  with $\sigma \subset \sigma'$. 

\vspace{10pt}
\noindent \textbf{Step III: Reduction to simple normal crossings. }
\vspace{10pt}

\edits{A worked example for this step of the proof can be found in Remark \ref{rk:proof_example}. }

Choose an \'etale cover $\mathcal{A}_{\tau'}$ of $\mathcal{A}_{\sigma'}$ with $\tau' = \mathbb{R}_{\ge 0}^{k'}$. The diagram commutes for $X = \mathcal{A}_{\tau'}$ by Step I.
To prove commutativity for $X = \ca A_{\sigma'}$, we must
keep track of how passing from $\ca A_{\sigma'}$ to $\ca A_{\tau'}$ affects the monodromy. 

Consider the fiber diagram of stacks
\[
\begin{tikzcd}
Q \ar[r,"h"] \ar[d,"p"] & \mathcal{A}_{\tau'} \ar[d,"q"] \\ P \ar[r,"j"] & \mathcal{A}_{\sigma'}
\end{tikzcd}
\]
where $Q\stackrel{\sim}{=} P \times_{\mathcal{A}_{\sigma'}} \mathcal{A}_{\tau'}$. Let $\widetilde \tau$ be a cone in $\tau'$ that maps isomorphically to $\sigma \subset \sigma'$. Geometrically, such a cone corresponds to a stratum $T_{\widetilde \tau} \subset \mathcal{A}_{\tau'}$ that maps to $S$, or, equivalently, to a connected component of the preimage of $S$ in $\mathcal{A}_{\tau'}$. By  definition, $P$ is a $G$-torsor over the normalization $\widetilde{S}$ of the closure $\overline{S}$ of $S$. Since $q$ is \'etale, the pullback of $\widetilde{S}$ to $\mathcal{A}_{\tau'}$ is the normalization of the union of the closures $\overline T_{\widetilde \tau}$ of the $T_{\widetilde \tau}$. Since $\mathcal{A}_{\tau'}$ is simple normal crossings, the latter is simply the disjoint union of the closures $\overline T_{\widetilde \tau}$. Furthermore, as $\mathcal{A}_{\tau'}$ has no monodromy, the monodromy torsor of each $\overline T_{\widetilde \tau}$ is $\overline T_{\widetilde \tau}$ itself, and the pullback of $P$ to $\mathcal{A}_{\tau'}$ is a trivial $G$-torsor. We conclude that $Q$ is a disjoint union
$$
Q \cong \coprod_{(\widetilde \tau,\gamma) \in \mathcal{F}_k} Q_{(\widetilde \tau,\gamma)} 
$$
with isomorphic connected components $Q_{(\widetilde \tau,\gamma)} \cong \overline T_{\widetilde \tau}$. 

Let $B_{\widetilde \tau}$ be the set of branches of the divisor in $\mathcal{A}_{\tau'}$ cutting out $T_{\widetilde \tau}$. By construction, the group $G$ acts on $B_{\widetilde \tau}$, and the set of $G$-equivariant bijections between $B_{S}$ and $B_{\widetilde \tau}$ is a $G$-torsor, naturally identified with the torsor of pairs $({\widetilde \tau},\gamma) \in \mathcal{F}_k$.  Since $q$ is \'etale, we find that 
$$
p^*N_{P/X} = N_{Q/\mathcal{A}_{\tau'}}\, .
$$
Writing $p_{({\widetilde \tau},\gamma)}$ for the restriction of $p$ to $Q_{({\widetilde \tau},\gamma)}$, we have that for each $b \in B_S$, 
$$
p_{({\widetilde \tau},\gamma)}^*N_{b} = N_{({\widetilde \tau},\gamma)(b)} 
$$
where we have written $({\widetilde \tau},\gamma)(b)$ for the image of $b$ under the bijection $$B_S \to B_{\widetilde \tau}$$
corresponding to $({\widetilde \tau},\gamma)$. Therefore, for any polynomial in $F \in \mathbb{Q}[B_S]$, an element $({\widetilde \tau},\gamma) \in \ca{F}_k$ determines a polynomial $\gamma^*F \in \mathbb{Q}[B_{\widetilde \tau}]$, and the pullback of $j_*F(c_1(N_{P/X}))$ to $\mathcal{A}_{\tau'}$ is the sum
$$
\sum_{({\widetilde \tau},\gamma)\in \mathcal{F}_k} \gamma^*F(c_1(N_{\overline{T}_{\widetilde \tau}/\mathcal{A}_{\tau'}}))\, .
$$
By Step I, the corresponding polynomial on $\mathcal{A}_{\tau'}$ is
$$
\sum_{({\widetilde \tau},\gamma) \in \mathcal{F}_k} \overline{\delta_{{\widetilde \tau}} \gamma^*F} \, ,
$$
which is precisely the polynomial on $\mathcal{A}_{\tau'}$ that descends to $\mathfrak{j}_*F$ on $X$. 

If we denote the projection $Q \to \coprod_{\widetilde \tau} \overline{T}_{\widetilde \tau}$ by $\pi$, the inclusion $\overline{T}_{\widetilde \tau} \to \mathcal{A}_{\tau'}$ by $j_{\widetilde \tau}$, and the homomorphism $\mathbb{Q}[B_{\widetilde \tau}] \to \mathsf{sPP}(\tau')$ by $(\mathfrak{j}_{\widetilde \tau})_*$, we have proven  that the following diagram commutes: 
\[
\begin{tikzcd}[row sep=2.5em]
\mathbb{Q}[B_S] \arrow[rr,"\mathfrak{j}_*"] \arrow[dr,"\sum_{({\widetilde \tau},\gamma)} \gamma^*F"] \arrow[dd,swap,"c_1(N_{P/X})"] &&
\mathsf{sPP}(\sigma') \arrow[dd,swap,"\Phi" near start] \arrow[dr,"q^*"] \\
& \oplus_{\widetilde \tau} \mathbb{Q}[B_{\widetilde \tau}] \arrow[rr,crossing over,"\sum_{\widetilde \tau} (\mathfrak{j}_{\widetilde \tau})_*" near start] &&
  \mathsf{sPP}(\tau') \arrow[dd,"\Phi"] \\
\mathsf{CH}(P) \arrow[rr,"j_*" near end] \arrow[dr,swap,"\pi_*p^*"] && \mathsf{CH}(X) \arrow[dr,swap,"q^*"] \\
& \oplus_{\widetilde \tau} \mathsf{CH}(\overline{T}_{\widetilde \tau}) \arrow[rr,"\sum_{\widetilde \tau} (j_{\widetilde \tau})_*"] \arrow[uu,<-,crossing over,"\oplus_{\widetilde \tau} c_1(N_{\overline{T}_{\widetilde \tau}/\mathcal{A}_{\tau'}})" near end]&& \mathsf{CH}(\mathcal{A}_{\tau'})
\end{tikzcd}
\]
Here, the (direct) sums are taken over the set of cones $\widetilde \tau$ in $\tau'$ that map isomorphically to $\sigma \subset \sigma'$. 
Since we know the commutation for the front face of the cube, and all arrows from the back face to the front face are injective, the commutation
for the back face follows. 
\end{proof}

\begin{lemma}\label{lem:chow_is_sheaf_on_AF}
\edits{Let $\ca A_X$ be an Artin fan, and write $(\ca A_X)_{et}$ for
the small strict \'etale site of $\ca A_X$.}  Then the functor 
\begin{equation*}
    \Chow^*\colon (\ca A_X)_{et}^\op \to \bb Q\cat{Alg}
\end{equation*}
is a sheaf. 
\end{lemma}

\begin{proof}The result is a rephrasing of \cite[Theorem 14]{Molcho2021-Hodgebundle} which shows that the Chow ring of the Artin fan coincides with the algebra of \edits{strict} piecewise polynomials. 
\edits{The point is that this identification also holds for every strict \'etale cover of $\ca A_X$. More precisely, since for every strict \'etale map $U \to \ca A_X$, $U$ is itself an Artin fan, \cite[Theorem 14]{Molcho2021-Hodgebundle} shows that  }
\begin{equation*}
\Chow^*(U)=
\sf{sPP}(U). 
\end{equation*}
\edits{The result is proven in \cite{Molcho2021-Hodgebundle} for the case where $\ca A_X$ is the canonical Artin fan of a variety $X$, but all that is actually used is that the cone stack $\ca A_X$ can be stratified by a finite number of strata, each isomorphic to $B(\mathbb{G}_m^k \rtimes G)$ for nonnegative integers $k$ and finite group $G$. Since $\sf{sPP}$ is by construction a strict \'etale sheaf, the result follows.}
\end{proof}

\begin{remark}\label{rk:proof_example}
\edits{To aid the reader in following the proof above, we work through the later steps in the special case where $X = \Mbar_{1,2}$ and $S = B \bb Z/2$ is the closed stratum corresponding to a cycle of two projective lines. Step II reduces us to the case where $X = \ca A_{\sigma'}$ is the Artin fan 
$$\left[\frac{\Spec \bb C[\bb N^2]}{\Spec \bb C [\bb Z^2] \rtimes \bb Z/2}\right]$$
and $S = \widetilde S$ is the (stacky) closed point of $X$.  We choose 
$$\ca A_{\tau'} = \left[\frac{\Spec \bb C[\bb N^2]}{\Spec \bb C [\bb Z^2] }\right], $$
 so that $P$ is the closed point of $\ca A_{\tau'}$, and $Q$ can be identified with the disjoint union of two copies of $P$:}
 \[
\begin{tikzcd}
Q = B \mathbb{G}_m^2 \sqcup B \mathbb{G}_m^2 \arrow[d] \arrow[rr] &  & {[\mathbb{A}^2 / \mathbb{G}_m^2]} = \mathcal{A}_{\tau'} \arrow[d] \\
P = B \mathbb{G}_m^2 \arrow[r] & B \mathbb{G}_m^2 \rtimes \mathbb{Z}/2 \arrow[r]  & {[\mathbb{A}^2 / \mathbb{G}_m^2\rtimes \mathbb{Z}/2]}= \mathcal{A}_{\sigma'}. 
\end{tikzcd}
\]\edits{We find that $\sigma = \sigma'$ and $\tau = \tau'$ because we study a closed stratum. 
The diagram}
\[
 \begin{tikzcd}
  \bb Q{[B_S]} \arrow[rr, "\mathfrak j_*"]\arrow[d, "c_1(N)"]  & & \sf{sPP}(\Sigma_X) \arrow[d, "\Phi"] \\
  \sf{CH}^*(P) \arrow[r, "p_*"] & \sf{CH}^*(\widetilde S) \arrow[r] & \sf{CH}^*(X). 
\end{tikzcd}
\]
\edits{specializes to}
\[
 \begin{tikzcd}[column sep=4cm]
  \bb Q{[x,y]} \arrow[rr, "{f(x,y)\mapsto xy(f(x,y)+f(y,x))}"]\arrow[d, "\mathrm{id}"]  & & \bb Q{[x,y]}^{\mathbb{Z}/2} \arrow[d, "\mathrm{id}"] \\
  \bb Q{[x,y]} \arrow[r, "{f(x,y) \mapsto f(x,y)+f(y,x)}"] &  \bb Q{[x,y]}^{\mathbb{Z}/2} \arrow[r, "{g(x,y) \mapsto xyg(x,y)}"] & \bb Q{[x,y]}^{\mathbb{Z}/2}
\end{tikzcd}
\]
\end{remark}

\subsubsection{Explicit correspondence}\label{excor}

We return to our main case of interest: strict piecewise polynomials on the stack of log curves $\mathfrak{M}_g$.

\vspace{8pt}
\noindent $\bullet$ Let $\Gamma$ be a graph corresponding to a cone $\sigma_\Gamma$. The rays of $\sigma_\Gamma$ are identified with the edges $E(\Gamma)$ of $\Gamma$. The corresponding stratum \edits{$\mathfrak{M}^\Gamma \subseteq \mathfrak{M}_g$} is given by the immersion
\begin{equation}\label{eq:stratum}
    \edits{\mathfrak{M}^\Gamma = }\prod_{v \in V(\Gamma)}\mathfrak M^{sm}_{g(v), H(v)}/\on{Aut}(\Gamma) \to \mathfrak{M}_g\, ,
\end{equation}
where $H(v)$ is the set of half-edges attached to the vertex $v$, see \cite[Proposition 2.5]{BaeSchmitt}. Define
\begin{equation*}
\mathfrak M_\Gamma = \prod_{v \in \V(\Gamma)} \mathfrak M_{g(v), H(v)}\, . 
\end{equation*}
The normalization of the closure of the stratum is given by 
\begin{equation*}
\mathfrak M_\Gamma/\on{Aut}_\Gamma \to \mathfrak{M}_g\, , 
\end{equation*}
and the universal monodromy torsor is given by
\begin{equation} \label{eqn:monodromy_torsor_Mgamma}
\mathfrak M_\Gamma\edits{/\on{Aut}_\Gamma^\textup{loop}}\to \mathfrak M_\Gamma/\on{Aut}_\Gamma \, , 
\end{equation}
\edits{where as in Example \ref{Exa:Mbar_monodromy_first} the group $\Aut_\Gamma^\textup{loop} \subseteq \Aut_\Gamma$ is the subgroup of automorphisms of $\Gamma$ acting as the identity on $E(\Gamma)$. By the same argument as in Example \ref{Exa:Mbar_monodromy_first}, we see that the monodromy group $G_\Gamma \subseteq \mathrm{Sym}(E(\Gamma))$ of the stratum $\mathfrak{M}^\Gamma$ is the quotient
\[
G_\Gamma = \Aut_\Gamma / \on{Aut}_\Gamma^\textup{loop}\,.
\]
Then \eqref{eqn:monodromy_torsor_Mgamma} is precisely the $G_\Gamma$-monodromy torsor over the normalization of the strata closure for $\mathfrak{M}^\Gamma$.

For the calculations below, it will be more convenient to compose the monodromy torsor \eqref{eqn:monodromy_torsor_Mgamma} with the additional \'etale cover $\mathfrak M_\Gamma \to \mathfrak M_\Gamma/\on{Aut}_\Gamma^\textup{loop}$ since that composition induces} the familiar map
\begin{equation*}\label{eq:j_Gamma}
j_\Gamma\colon \mathfrak M_\Gamma \to \mathfrak{M}_g\, . 
\end{equation*}

\vspace{8pt}
\noindent $\bullet$ 
Let $d\colon E(\Gamma) \to \bb Z_{\ge 0}$. Define a polynomial on the set of edges of $\Gamma$: 
\begin{equation*}
F = \prod_{e \in E(\Gamma)} e^{d_e} \in \bb Q[E(\Gamma)], 
\end{equation*}

\noindent Then, $\mathfrak j_*F$  is the sPP function on $\mathfrak M$ defined as follows: for a stable graph $\Gamma'$, the value on the cone $\sigma_{\Gamma'}$ is given by the formula
$$
\edits{\frac{1}{|\Aut_\Gamma^\textup{loop}|}}\sum_{f: \Gamma' \to \Gamma} \prod_{e \in E(\Gamma)} \ell_{f(e)}^{d_e + 1}\,.
$$
As in Example \ref{Exa:jpushforward}, the graph morphisms $f : \Gamma' \to \Gamma$ precisely correspond to the elements $(\widetilde \tau, \gamma)$ that we sum over in the definition of $\mathfrak j_*^{\sigma_{\Gamma'}}$.
\edits{The factor $\frac{1}{|\Aut_\Gamma^\textup{loop}|}$ arises since again we identify morphisms whose actions on vertices and edges agree. The equivalence classes of this relation are precisely the orbits of the free action of $\Aut_\Gamma^\textup{loop}$ on the set of morphisms $\Gamma' \to \Gamma$.}

The above formula automatically vanishes outside the star of the cone of $\Gamma$, since \edits{for any cone $\sigma_{\Gamma'}$ outside this star, there exist no} graph morphisms $\Gamma' \to \Gamma$ \edits{and thus the above sum is empty}. 
Define the class
\begin{equation*}
\mathsf{m} = \prod_{e = (h, h') \in E(\Gamma)} (-\psi_h - \psi_{h'})^{d_e} \in \sf{CH}^*(\Mbar_\Gamma)\, . 
\end{equation*}
\begin{proposition}
\label{lem:comparing_graph_sum_and_PP}
The following explicit correspondence holds:
$$\edits{\frac{1}{|\Aut_\Gamma^\textup{loop}|} \cdot}j_{\Gamma,*}(\mathsf{m}) = \Phi(\mathfrak j_*F)\,.$$
\end{proposition}
\begin{proof}
We apply Lemma \ref{lem:compare_decorated_strata_and_sPP}
where 
$S\edits{=\mathfrak{M}^\Gamma}$ is the stratum of 
$X = \mathfrak M_g$
defined by \eqref{eq:stratum}.
Then, $B_S$ is naturally in bijection with $E(\Gamma)$. We view $F$ as an element
of the top left corner of \eqref{dddiii}. 
We must show that
$$j_{\Gamma,*}(\mathsf{m}) 
= \edits{|\Aut_\Gamma^\textup{loop}|\cdot} j_*(F(c_1(N)))\, . $$
Since $j\colon P \edits{= \mathfrak M_\Gamma/\on{Aut}_\Gamma^\textup{loop}} \to X = \mathfrak{M}_g$ \edits{fits in a diagram}
\[
\begin{tikzcd}
\mathfrak{M}_\Gamma \arrow[r,"j'"] \arrow[rr, bend right, "j_\Gamma"] & \mathfrak{M}_\Gamma/\on{Aut}_\Gamma^\textup{loop} \arrow[r, "j"] & \mathfrak{M}_g
\end{tikzcd}\,
\]
we need only show
$$\mathsf{m} = \edits{(j')^*} F(c_1(N))\in \mathsf{CH}(\mathfrak M_\Gamma)\, ,$$ 
which is clear. 
\end{proof}



\edits{
\begin{example}
We write the correspondence of Proposition \ref{lem:comparing_graph_sum_and_PP} explicitly in the case of the moduli space $\Mbar_{1,2}$. For more details of the computation (with illustrations), we refer the reader to \cite[Chapter 2]{rosa_phd_thesis}. The strict piecewise polynomials on $\Mbar_{1,2}$ can be identified with the subring of the polynomial ring 
\begin{equation*}
    \bb Q[x, y, z]
\end{equation*}
consisting of those polynomials which are symmetric under swapping $x$ and $y$. Here, $z$ is the strict piecewise linear function associated to the boundary divisor with a separating node, and $x$ and $y$ are the two branches of the divisor of curves with a non-separating node at its self-intersection point. Figure \ref{fig:Phi_strata_Mbar12} gives the correspondence between strict piecewise polynomials and fundamental classes of strata closures.
\begin{figure}
    \centering
\begin{tikzpicture}
    \node (T) at (0, 0) {$\Gamma$};
    
    \node[circle, draw] (A1) at (2, 0) {1};
    \draw (A1) -- ++(135:0.5) node[above left] {1};
    \draw (A1) -- ++(45:0.5) node[above right] {2};

    \begin{scope}[xshift=-0.5cm]
    \node[circle, draw] (A2) at (4, 0) {1};
    \node[circle, draw] (A3) at (5, 0) {0};
    \draw (A2) -- (A3);
    \draw (A3) -- ++(45:0.5) node[above right] {1};
    \draw (A3) -- ++(-45:0.5) node[below right] {2};
    \end{scope}
    
    \node[circle, draw] (A4) at (7, 0) {1};
    \draw (A4) to[out=135,in=-135,looseness=8] (A4);
    \draw (A4) -- ++(45:0.5) node[above right] {1};
    \draw (A4) -- ++(-45:0.5) node[below right] {2};

    \begin{scope}[xshift=0.25cm]
    \node[circle, draw] (A5) at (9, 0) {1};
    \node[circle, draw] (A6) at (10, 0) {0};
    \draw (A5) -- (A6);
    \draw (A5) to[out=135,in=-135,looseness=8] (A5);
    \draw (A6) -- ++(45:0.5) node[above right] {1};
    \draw (A6) -- ++(-45:0.5) node[below right] {2};
    \end{scope}
    
    \node[circle, draw] (A7) at (12, 0) {0};
    \node[circle, draw] (A8) at (13, 0) {0};
    \draw (A7) to[out=45,in=135] (A8);
    \draw (A7) to[out=-45,in=-135] (A8);
    \draw (A7) -- ++(135:0.5) node[above left] {1};
    \draw (A8) -- ++(45:0.5) node[above right] {2};
    
    \node (f) at (0, -2) {$f_\Gamma$};
    \node (one) at (2, -2) {$1$};
    \node (z) at (4, -2) {$z$};
    \node (x_plus_y) at (7, -2) {$x + y$};
    \node (xy_z) at (9.6, -2) {$(x + y)z$};
    \node (xy) at (12.5, -2) {$xy$};
\end{tikzpicture}
    \caption{A list of stable graphs $\Gamma$ for $\Mbar_{1,2}$ and the corresponding strict piecewise polynomials $f_\Gamma$ such that $\Phi(f_\Gamma)$ is the fundamental class of the associated stratum closure in $\Mbar_{1,2}$}
    \label{fig:Phi_strata_Mbar12}
\end{figure}
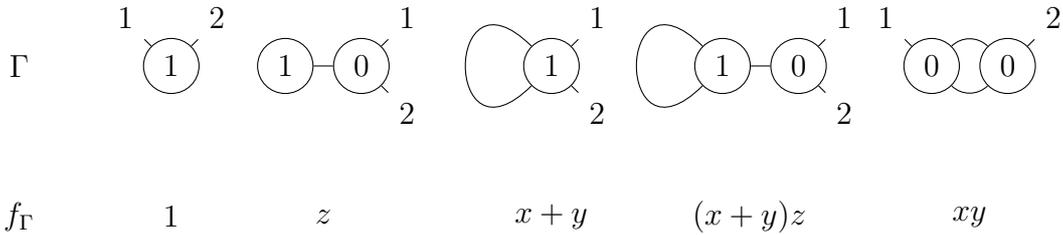
\end{example}}
\subsection{Pixton's formula in terms of piecewise polynomials}

\label{sec:classical_DR_in_PP_language}
\subsubsection{Contributions}
We translate here 
the formula of \cite{BHPSS} for the universal double ramification cycle into the language of piecewise polynomials.
The formula is written in the
Chow cohomology of the universal
Picard stack 
$\Picabs$ of
degree 0 line bundles over the stack of genus $g$ 
curves $\mathfrak{M}_g$.
The result will be used in Section \ref{pprr22} to prove our formula for the logarithmic double ramification cycle. 
 

Let $\edits{\pi\colon}C \to S$ be a log curve of genus $g$, and 
let $\mathcal L$ be a line bundle on $C$
of degree 0. The double ramification cycle $$\DR \in \Chow^g(\Picabs )$$ can be 
naturally expressed
in terms of strict piecewise polynomials on $S$. We require
the following two constructions:
\begin{enumerate}
\item[$\bullet$] Let $\eta^\Picabs = \pi_*(c_1(\ca L)^2) \in \Chow^1(\Picabs )$. 
\item[$\bullet$]
For every positive integer $r$, let $\mathrm{Cont}^r$ be
 the strict piecewise power series
on $S$ defined on the  cone associated to a graph $\Gamma$ of genus $g$ by:
\begin{equation}\label{eq:basic_dr_formula}
\mathrm{Cont}^r_\Gamma = \sum_{w} r^{-h^1(\Gamma)}\prod_{e \in E(\Gamma)} \exp\left(\frac{\overline{w}(\vec e)\overline{w}(\cev e)}{2} \ell_e\right) \ \in\  \bb Q[[\ell_e: e \in E(\Gamma)]]\, .
\end{equation}
The sum runs over {\em admissible weightings} mod $r$: 
flows $w$ with values in $\mathbb{Z}/r\mathbb{Z}$ 
such that
\[
\mathrm{div}(w) = {\mdeg}(\mathcal{L}) \in (\mathbb{Z}/r\mathbb{Z})^{V({\Gamma})}\, .
\]
Inside the exponential, 
$\overline{w}(\vec e)$ and $\overline{w}(\cev e)$
denote the unique representative of $w(\vec e)  \in \mathbb{Z}/r\mathbb{Z}$
and $w(\cev e)  \in \mathbb{Z}/r\mathbb{Z}$
in $\{0, \ldots, r-1\}$.

\end{enumerate}

\begin{remark}\label{rk:poly_as_sPP}
To make sense of the contribution
formula 
\eqref{eq:basic_dr_formula}, we must show that the collection of power series $\mathrm{Cont}^r_\Gamma$ yields a strict piecewise power series on 
the cone stack of $S$. 
The issue is about the compatibility of the formula on
different cones $\Gamma$.

An \'etale specialization of geometric points of $\Picabs $,
$$\phi\colon t \mapsto s\, ,$$
yields a natural contraction map $\Gamma_s \to \Gamma_t$,
an inclusion $E(\Gamma_t) \hookrightarrow E(\Gamma_s)$ of
edges, and an injective map of rings 
\begin{equation*}
    P_\phi\colon \bb Q[[\ell_e \colon e \in E(\Gamma_t)]] \to \bb Q[[\ell_e \colon e \in E(\Gamma_s)]].
\end{equation*}
Suppose for every geometric point $s$ in
$\Picabs$, 
we are given a power series 
$$f^s \in \bb Q[\edits{[}\ell_e \colon e \in E(\Gamma_s)\edits{]}]\, .$$
If, for every \'etale specialization $\phi\colon t \mapsto s$, we have 
\begin{equation}\label{eq:specialization_compatibility}
    P_\phi(f^t) = f^s\big|_{\{\ell_e=0 \, |\,  e\notin E(\Gamma_t)\}}\, ,
\end{equation}
then the collection of power series $\{ f_s \}$ glue to a global strict piecewise power series on the cone stack of $S$. 
Every truncation of the collection of
power series $\{f^s\}$ can be viewed as a 
strictly piecewise polynomial on $\Picabs$.

If the $f^s$ are compatible under graph contractions and graph automorphisms, then the compatibility
\eqref{eq:specialization_compatibility} is satisfied over $\Picabs$.
For the  power series $\mathrm{Cont}^r$ of  \eqref{eq:basic_dr_formula},
these two compatibilities hold \edits{(see \cref{lem:contraction_invariance} below)}.
\end{remark}

\begin{lemma}\label{lem:contraction_invariance}
Formula \eqref{eq:basic_dr_formula} for $\mathrm{Cont}^r$
is compatible with respect to  graph automorphisms and contractions. 
\end{lemma}
\begin{proof}
The claim for automorphisms is clear by definition (since only the structure of the graph $\Gamma$ is used).

Since a general contraction is a composition of  contractions of 
single edges, we need only study the contraction of a single edge $e$,
$$ \Gamma \to \Gamma'\,.$$
There are two cases:

\begin{enumerate}
\item[(i)] The edge $e$ is not a loop. Weightings on the contracted graph $\Gamma'$ are then naturally in bijection with weightings on $\Gamma$, and 
$h^1(\Gamma') = h^1(\Gamma)$.

\item[(ii)] The edge $e$ is a loop. Then, $h^1(\Gamma') = h^1(\Gamma) - 1$. For a given weighting on $\Gamma'$ there are exactly $r$ lifts to weightings on $\Gamma$ (obtained by assigning weights $i$ and $r-i$ to half edges of $e$ for $0 \le i < r$).  \edits{To specialise $\mathrm{Cont}^r_\Gamma$ to $\Gamma'$, we set $\ell_e = 0$. The result differs from the expression for $\mathrm{Cont}^r_{\Gamma'}$ in two ways: we sum over $r$ additional weightings, and we divide by $r$; these cancel exactly. }
\end{enumerate}
Therefore, in both cases,  compatibility \eqref{eq:specialization_compatibility} holds.
\end{proof}

\begin{lemma}
For truncations of a fixed degree
on quasi-compact opens of $S$, 
$\mathrm{Cont}^r$ is eventually polynomial in $r$. 
\end{lemma}

\begin{proof}
For every graph $\Gamma$, the eventual polynomiality in $r$
of $\mathrm{Cont}^r_\Gamma$ is obtained from the theory of Ehrhart polynomials,
see \cite[Appendix A]{JPPZ}.
Quasi-compactness is require\edits{d} to ensure that only finitely many graphs
are considered.
\end{proof}

We define $\mathrm{Cont}$ to be the $r=0$ value of 
the polynomial determined by $\mathrm{Cont}^r$ for large $r$.
Every fixed degree truncation of
$\mathrm{Cont}$ is 
a piecewise polynomial on the cone stack of $\Picabs$ . 

\subsubsection{The universal \texorpdfstring{$\mathsf{DR}$}{DR} cycle}
\label{Pix-uni}
With these preparations in place, we define
\begin{equation*}
    \mathbf{P}_g = \exp\Big(-\frac{1}{2} \eta^\Picabs\Big) \cdot \Phi(\mathrm{Cont}) \in \sf{CH}^*(\Picabs )\, . 
\end{equation*}
On the other hand, 
a mixed degree operational Chow class $$\mathsf{P}_{g,\emptyset,0} \in \sf{CH}^*(\Picabs)$$ is defined in \cite{BHPSS} whose codimension $g$ part $\mathsf{P}_{g,\emptyset,0}^g$ computes $\DR$.
\begin{theorem}\label{thm:regularDR}
We have the equality
\begin{equation*}
  \mathbf{P}_g=  \mathsf{P}_{g,\emptyset,0} 
    \in \Chow^*(\Picabs)
    \,.
\end{equation*}
Hence, we have
$    \DR= \mathbf{P}^g_g$,
where $\mathbf{P}^g_g$ denotes the codimension $g$ part of 
$\mathbf{P}_g$. 
\end{theorem}
\begin{proof}
By the definition of \cite{BHPSS},
$\mathsf{P}_{g,\emptyset,0}$ is given by the coefficient of $r^0$
of\footnote{The half-edges $h,h'$ appearing in the formula from \cite{BHPSS} naturally correspond to the directed edges $\vec e, \cev e$, and under this correspondence, the notions of admissible weightings mod $r$ likewise agree.}
\begin{align*} 
\exp\left(-\frac12 \eta^\Picabs \right) \sum_{
\substack{\Gamma\in \G_{g,0,0} \\
w\in \mathsf{W}_{\Gamma,r}}
}
\frac{r^{-h^1(\Gamma)}}{|\Aut(\Gamma)| }
\;
j_{\Gamma*}\Bigg[&
\prod_{e\in \E(\Gamma)}
\frac{1-\exp\left(-\frac{w(\vec e)w(\cev e)}2(\psi_{\vec e}+\psi_{\cev e})\right)}{\psi_{\vec e} + \psi_{\cev e}} \Bigg]\, . 
\end{align*}
Following the notation of \cite{BHPSS},
the above sum is over the set of all possible graphs
$\G_{g,0,0}$ for the universal Picard stack and over the
set $\mathsf{W}_{\Gamma,r}$
of all possible admissible weightings mod $r$ on $\Gamma$.

It suffices to show, for fixed $r$, that the class
$$\Phi(\mathrm{Cont}^r) =\Phi\left(\, 
\left\{ \sum_{w} r^{-h^1(\Gamma)}\prod_{e} \exp\left(\frac{w(\vec e)w(\cev e)}{2} \ell_e\right)\right\}_\Gamma \, \right)$$
equals the class
\begin{align*}
\sum_{
\substack{\Gamma \in \G_{g,0,0} \\
w\in \mathsf{W}_{\Gamma,r}}
}
\frac{r^{-h^1(\Gamma)}}{|\Aut(\Gamma)| }
\;
j_{\Gamma*}\Bigg[&
\prod_{e\in \E(\Gamma)}
\frac{1-\exp\left(-\frac{w(\vec e)w(\cev e)}2(\psi_{\vec e}+\psi_{\cev e})\right)}{\psi_{\vec e} + \psi_{\cev e}} \Bigg]\, .
\end{align*}

Consider a graph $\Gamma\in \G_{g,0,0}$. 
Let $\mathrm{Cont}^r(\Gamma)$ be the part of $\mathrm{Cont}^r$ containing those monomials in which {\em all} the $\ell_e$ for $e \in E(\Gamma)$ appear with {\em positive} exponent.\footnote{More precisely, for a graph $\Gamma\in \G_{g,0,0}$, we can describe the piecewise polynomial $\mathrm{Cont}^r(\Gamma)$ as follows. On the cone associated to a stable graph $\Gamma'$,
 $\mathrm{Cont}^r(\Gamma)$  is given by those monomial terms of $\mathrm{Cont}^r_{\Gamma'}$
 of \eqref{eq:basic_dr_formula} for which the contraction of all edges of $\Gamma'$ associated to variables $\ell_e$ which \emph{do not} appear in the monomial is a graph isomorphic to $\Gamma$.}
We easily see that
\begin{equation*}
\mathrm{Cont}^r = \sum_{\Gamma\in  \G_{g,0,0}     } \mathrm{Cont}^r(\Gamma)\, . 
\end{equation*}
\edits{Indeed, on the cone associated to the stable graph $\Gamma'$, for each monomial term $c_\alpha \ell^\alpha$ appearing in $\mathrm{Cont}^r$ let $\Gamma$ be the graph obtained from $\Gamma'$ by contracting all edges $e \in E(\Gamma')$ whose variable $\ell_e$ does not appear in the monomial $\ell^\alpha$. Then $c_\alpha \ell^\alpha$ appears in $\mathrm{Cont}^r(\Gamma)$ (and not in any other $\mathrm{Cont}^r(\Gamma'')$), and the above sum just expresses how the monomial terms of $\mathrm{Cont}^r$ are classified according to the possible contracted graphs $\Gamma$.}
We claim that the class $\Phi(\mathrm{Cont}^r(\Gamma))$ is exactly 
equal to 
\begin{align*} 
\sum_{
\substack{w\in \mathsf{W}_{\Gamma,r}}
}
\frac{r^{-h^1(\Gamma)}}{|\Aut(\Gamma)| }
\;
j_{\Gamma*}\Bigg[&
\prod_{e\in \E(\Gamma)}
\frac{1-\exp\left(-\frac{w(\vec e)w(\cev e)}2(\psi_{\vec e}+\psi_{\cev e})\right)}{\psi_{\vec e} + \psi_{\cev e}} \Bigg]\, . 
\end{align*}
The equality
follows from \edits{Proposition} \ref{lem:comparing_graph_sum_and_PP}.
\edits{Indeed, for the power series
\[
F_\Gamma = |\Aut_\Gamma^\textup{loop}| \cdot \sum_{
\substack{w\in \mathsf{W}_{\Gamma,r}}
}
\frac{r^{-h^1(\Gamma)}}{|\Aut(\Gamma)| }
\;
\Bigg[
\prod_{e\in \E(\Gamma)}
\frac{1-\exp\left(-\frac{w(\vec e)w(\cev e)}2e\right)}{e} \Bigg]\in \mathbb{Q}[[E(\Gamma)]] 
\] 
in formal variables $e$ (for $e \in E(\Gamma)$), we can check monomial by monomial that $\mathfrak{j}_* F_\Gamma = \mathrm{Cont}^r(\Gamma)$, in the notation from \cref{lem:comparing_graph_sum_and_PP}, whereas the $\psi$-monomial $\mathsf{m}_\Gamma$ on $\Mbar_\Gamma$ associated to $F_\Gamma$ is precisely the term appearing in the pushforward $j_{\Gamma*}$ above. The intuition here is} 
that replacing $\ell_e^d$ by $(-\psi_{\vec e} - \psi_{\cev e})^{d-1}$ everywhere in the expression $\exp(\frac{w(\vec e)w(\cev e)}{2} \ell_e)$ yields exactly
$$\frac{1-\exp\left(-\frac{w(\vec e)w(\cev e)}2(\psi_{\vec e}+\psi_{\cev e})\right)}{\psi_{\vec e} + \psi_{\cev e}}\, .$$
\edits{The factor $|\Aut_\Gamma^\textup{loop}|$ in the formula of $F_\Gamma$ precisely cancels the opposite factor appearing in Proposition \ref{lem:comparing_graph_sum_and_PP}.}
We see that the language of piecewise polynomials expresses Pixton's
formula more efficiently.
\end{proof}

\section{Theorems \ref{222ttt} and \ref{333ttt}} 
\label{pprr22}


\subsection{Families of curves}
\label{famcur}
Let $S$ be a log smooth log algebraic stack, let $C/S$ be a log curve of genus $g$, and let $\ca L$ be a line bundle on $C$ of degree 0.
We assume, in addition, that we have:
\begin{enumerate}
\item[(i)]
a log modification $\widehat S \to S$, 
\item [(ii)] a quasi-stable model
$\widehat C \to C \times_S \widehat S$,
\item[(iii)] $\alpha \in \ghost_{\widehat C}^\gp(\widehat C)$ a strict piecewise linear function (in the sense of Section \ref{sec:algebraizing_tropical_AJ}). 
\end{enumerate}

For the proofs of Theorems \ref{222ttt} and
\ref{333ttt}, the fundamental geometry is
the following. Recall that $A=(a_1,\ldots, a_n)$ is a vector
of integers summing to $k(2g - 2 + n)$. 
Let $C/S$ be the universal curve $\mathcal{C} \to \Mbar_{g,n}$
over the moduli space of stable curves
with markings $x_1,\ldots,x_n$, 
and let 
\begin{equation*}
    \ca L = (\omega_C^\log)^{\otimes k}\left(- \sum_{i=1}^n a_i x_i\right)\, .
\end{equation*}
The additional data is given by:
\begin{enumerate}
\item[(i)]
$\SSS = \Mbar_{g,A}^\theta\stackrel{\rho}{\longrightarrow} \Mbar_{g,n}$ 
for a small nondegenerate stability condition $\theta$, 
\item[(ii)]
$\CCC= \mathcal{C}^\theta \to \mathcal{C} \times_{\Mbar_{g,n}} \Mbar_{g,A}^\theta$
is the universal quasi-stable curve,
\item[(iii)] $\alpha=\alpha^\theta$ is the universal sPL function on
$\mathcal{C}^\theta$ and
satisfies $\ca L^\theta = \ca L( \alpha^\theta)$.
\end{enumerate}


\subsection{Geometric construction of a logarithmic class on \texorpdfstring{$S$}{S} }
\newcommand{\DRap}{\mathsf{DR}^\mathsf{approx}}
\newcommand{\DRapp}{\mathsf{DR}^\approx}
The line bundle $\LLL=\ca L(\alpha)$ on $\widehat C$ defines a map 
$$\phi_\LLL\colon \widehat S \to \Picabs\, ,$$ and, in turn, a class 
$$\DR(\LLL)= \phi_\LLL^*(\DR) \in \Chow^g(\widehat S) \subset \LogChow^g(S)\, .$$
To emphasise that we are considering 
$\DR(\LLL)$ 
as a logarithmic class on $S$, we use
the notation
\begin{equation*}
\DRapp_S(\widehat S, \LLL) =
\DR(\LLL)
\in \LogChow^g(S)\, , 
\end{equation*}
an approximation to the true logarithmic double ramification cycle $$\LogDR(\ca L) \in \LogChow^g(S)$$
of \cite{HolmesSchwarz}.

By Proposition \ref{lem:invariance_7_b}, a sufficient condition for the equality 
\begin{equation*}
\DRapp_S(\widehat S,\LLL) = \LogDR(\ca L) 
\end{equation*}
is that $(\widehat C/\widehat S,\LLL)$ is almost twistable. 

In Section \ref{formmm}, we give a formula for $\DRapp_S(\widehat S, \LLL)$ which applies regardless of whether the family is almost twistable. 
Applied with Theorem A, the formula yields Theorem \ref{222ttt}. Theorem \ref{333ttt} is obtained
from Pixton's relations in the universal setting
\cite{BHPSS}.

\subsection{The functions \texorpdfstring{$\DRpp$}{Pfrak} and \texorpdfstring{$\DRpl$}{Lfrak}}



We take a (strict) geometric point $\widehat s$ of $\widehat S$ mapping to $s$, a strict geometric point of $S$. Over these points,
we have a graph $\widehat \Gamma$ of $\widehat C$ contracting to a graph $\Gamma$ of $C$, a map of cones $$\on{Hom}(\ghost_{\widehat S, \widehat s}, \bb R_{\ge 0}) = \widehat \sigma \to \sigma_\Gamma = \on{Hom}(\ghost_{S, s}, \bb R_{\ge 0}) \, ,$$ and a sPL function $\alpha$ on $\widehat \Gamma$. 

\begin{remark}\label{rem:univ_version_of_setup}
In the geometric setting of Theorems \ref{222ttt} and \ref{333ttt} (as specified
at the end of Section \ref{famcur}),
$$\sigma_\Gamma = \bb R_{\ge 0}^{E(\Gamma)}\, ,$$
and the choice of the point $\widehat s$ determines a flow $I$. 
The cone $\widehat \sigma$ is then given by 
$$\widehat \sigma = \sigma_{\widehat \Gamma, I}$$
defined in \eqref{eq:tau_cone_def}. The sPL function $\alpha$ is uniquely determined by the conditions:
\begin{enumerate}
\item[$\bullet$] $\alpha(x_1) = 0$,
\item[$\bullet$] the slopes of $\alpha$ are equal to $I$, 
\end{enumerate}
see \eqref{eq:universal_alpha}. In particular,  $\alpha$ here exactly matches \eqref{xsse3}. 
\end{remark}


We resume the general case, continuing to view the sPL function $\alpha \in \ghost_{\widehat C}(\widehat C)$ as a function $$V(\widehat \Gamma) \to \ghost_{\widehat S, \widehat s}^\gp\, ,$$ following Section \ref{sec:algebraizing_tropical_AJ}. Let 
$${\mdeg}(\LLL): V(\widehat{\Gamma}) \to \mathbb{Z}$$
be the multidegree.
We have a length function 
$$\ell\colon E(\widehat \Gamma) \to \ghost_{\widehat S, \widehat s}\, , $$
and seeing elements of $\ghost_{\widehat S, \widehat s}$ as strict piecewise linear functions on $\widehat{\sigma}$, we can define a power series on $\widehat{\sigma}$ by
\begin{equation}\label{eq:f1r}
\mathrm{Cont}^r_{\widehat{\sigma}} = \sum_{w} r^{-h^1(\widehat{\Gamma})}\prod_{e \in E(\widehat{\Gamma})} \exp\left(\frac{\overline{w}(\vec e)\overline{w}(\cev e)}{2} \ell_e\right) \ \in\  \bb Q[[\ell_e: e \in E(\widehat{\Gamma})]]\, .
\end{equation}
The sum runs over {\em admissible weightings} mod $r$: 
flows $w$ with values in $\mathbb{Z}/r\mathbb{Z}$ 
such that
\[
\mathrm{div}(w) =  {\mdeg}(\LLL)\in (\mathbb{Z}/r\mathbb{Z})^{V({\widehat{\Gamma}})}\, .
\]
Inside the exponential, 
$\overline{w}(\vec e), \overline{w}(\cev e) \in \{0, \ldots, r-1\}$
denote the unique representative of 
$w(\vec e), w(\cev e)  \in \mathbb{Z}/r\mathbb{Z}$
respectively.




\begin{definition} [The function $\DRpp$]
As explained in Section \ref{sec:classical_DR_in_PP_language}, the formula \eqref{eq:f1r} is
polynomial in $r$ for sufficiently large $r$. 
The function $\DRpp$, a strict piecewise power series on the cone
complex of $\widehat{S}$, is defined by the
the $r=0$
specialization of \edits{the polynomial in $r$ defined by} formula \eqref{eq:f1r}.
\end{definition}


\begin{definition}[The function $\DRpl$]
Following \eqref{eq:L_formula}, we define a sPL function
$\DRpl$ on the cone stack of $\widehat S$. 
On $\widehat{\sigma}$, the definition is
\begin{equation} \label{exxxx}
\DRpl = \sum_{v \in V(\widehat \Gamma)} 
\mdeg(\ca L \otimes \widehat{\ca L})(v) \cdot 
\alpha(v)\, , 
\end{equation}
where{\footnote{Over $\widehat{S}$, we follow the convention
that $\ca L$ denotes the pullback of $\ca L$ from $C$
to $\widehat{C}$.}} 
$\on{deg}(\ca L \otimes \widehat{\ca L})(v)$ 
is an integer and $\alpha(v) \in \ghost_{\widehat S, \widehat s}^\gp$ is viewed as a linear function on $\widehat\sigma$. We extend the expression
\eqref{exxxx} to a strict piecewise linear function
on the cone stack of $\widehat{S}$ using the method of Remark \ref{rk:poly_as_sPP}. 
\end{definition}

\subsection{The formula}
\label{formmm}
Application of the map $\Phi$ of \eqref{eqn:Phimapintro} to $\DRpp$ and $\DRpl$
yields classes
$$\Phi(\DRpp), \Phi(\DRpl)  \in \LogChow^*(S).$$
Let the logarithmic class
$
\eta  \in \LogChow^1(S)
$
be the image of
$$\eta=\pi_*(c_1(\ca L)^2) \in \Chow^1(\widehat S)$$
under the inclusion
$\Chow^*(\widehat S) \subset \LogChow^*(S)$.

\begin{definition}[The class $\mathbf{P}^{\widehat{S}}_{g,\widehat{\ca L}}\, $] The logarithmic class associated to $(\widehat{C}/\widehat{S}, \widehat{L})$
by Pixton's formula is defined by:
\begin{equation*}
\mathbf{P}^{\widehat{S}}_{g,\widehat{\ca L}}
= \exp\Big(-\frac{1}{2} (\eta + \Phi(\DRpl))\Big)\cdot \Phi(\DRpp) \in \LogChow^*(S)\, . 
\end{equation*}
Let 
$\mathbf{P}^{g,\widehat{S}}_{g,\widehat{\ca L}}$
be the codimension $g$ part of $\mathbf{P}^{\widehat{S}}_{g,\widehat{\ca L}}$ ,
$$
\mathbf{P}^{g,\widehat{S}}_{g,\widehat{\ca L}} \in \LogChow^\edits{g}(S)\, .
$$
\end{definition}

Following Section \ref{Pix-uni},
 the mixed degree operational Chow class $$\mathsf{P}_{g,\emptyset,0} \in \sf{CH}^*(\Picabs)$$ from \cite{BHPSS} has codimension $g$ part $\mathsf{P}_{g,\emptyset,0}^g$ which computes the universal $\DR$
 class.

\begin{theorem}\label{thm:DRapprox_formula}
We have
\begin{equation} \label{eqn:Pgapprox}
    \phi_{\widehat{\ca L}}^*(\mathsf{P}_{g,\emptyset,0}) =
    \mathbf{P}^{\widehat{S}}_{g,\widehat{\ca L}}
    \in \mathsf{logCH}^*(S)\,.
\end{equation}
In particular, we have
\begin{equation} \label{eqn:DRapprox}
\DRapp_S(\widehat S, \LLL) =
\mathbf{P}^{g,\widehat{S}}_{g,\widehat{\ca L}} \in \LogChow^\edits{g}(S)\, .
\end{equation}
\end{theorem}

\begin{proof}
On $\widehat{S}$, we have the classes 
\begin{equation*}
    \eta = \pi_*(c_1(\ca L)^2) \;\;\; \text{and}\;\;\; 
    \widehat{\eta} = \pi_*(c_1(\LLL)^2)\, , 
\end{equation*}
which are related by
\begin{equation*}
\begin{split}
        \widehat{\eta} & = \pi_*(c_1(\LLL)^2) \\
        & = \pi_*(c_1(f^*\ca L)^2) + 2\pi_*(c_1(f^*\ca L) \cdot c_1(
    \ca O(\alpha))) + \pi_*(c_1(\ca O(\alpha))^2)\\
     & = \eta + \pi_*(c_1(\ca L \otimes \LLL) \cdot c_1(\ca O(\alpha)))\, . 
    \end{split}
\end{equation*}
Next, for every line bundle  $\ca F$ on $\widehat C$ we claim that
\begin{equation}\label{eq:annoying}
    \pi_*(c_1(\ca F) \cdot c_1(\ca O(\alpha))) = \Phi\Big( \sum_{v \in \widehat V} (\mdeg\ca F)(v)\cdot  \alpha(v)\Big)\, ,  
\end{equation}
To  prove formula \eqref{eq:annoying}, we first reduce to the case where $c_1(\ca O_C(\alpha))$ is a vertical prime divisor $D$ on $\widehat C$ lying over a prime divisor $\pi(D)$ on $\widehat S$.  The left side yields $\pi(D)$ times the degree of $F$ on the generic fiber of $D \to \pi(D)$, which is exactly what is computed by the right side. 

Applying \eqref{eq:annoying} yields
$\widehat{\eta} = \eta + \Phi(\DRpl)$. By Theorem \ref{thm:regularDR}, we have{\footnote{Here, $\eta^\Picabs$ denotes the universal class on $\Picabs$ from Section \ref{sec:classical_DR_in_PP_language}.}}
\begin{align*}
    \phi_{\widehat{\ca L}}^*(\mathsf{P}_{g,\emptyset,0}) &=
\phi_{\LLL}^*\left(\exp\Big(-\frac{1}{2} \eta^\Picabs\Big) \cdot \Phi(\mathrm{Cont})\right)
\\
& = \exp\Big(-\frac{1}{2} \widehat{\eta}\, \Big) \cdot \phi_\LLL^* \Phi(\mathrm{Cont}) \\
& = \exp\Big(   -\frac{1}{2} (\eta + \Phi(\DRpl))  \Big) \cdot \phi_\LLL^* \Phi(\mathrm{Cont}) 
\end{align*}
By definition,  $\DRpp$ 
exactly matches the piecewise power series $\phi_\LLL^* \Phi(\mathrm{Cont})$.
Hence,
 $$\phi_{\widehat{\ca L}}^*(\mathsf{P}_{g,\emptyset,0}) =
\exp\Big(-\frac{1}{2} (\eta + \Phi(\DRpl))\Big)\cdot \Phi(\DRpp) \in \LogChow^*(S)\, ,$$
which proves \eqref{eqn:Pgapprox}. Equation \eqref{eqn:DRapprox}
follows from the codimension $g$ part of \eqref{eqn:Pgapprox}
together with \cite{BHPSS}.
\end{proof}

\subsection{Proofs of Theorems \ref{222ttt} and \ref{333ttt}}
\begin{proof}[Proof of Theorem \ref{222ttt}]
The result follows immediately by specializing Theorem \ref{thm:DRapprox_formula} to the fundamental geometry
specified in Section \ref{famcur}.
By Theorem \ref{111ooo}, the left  side of equation \eqref{eqn:DRapprox} specializes to $\mathsf{logDR}_{g,A}$, and the right side of equation \eqref{eqn:DRapprox} specializes
specializes to 
${\mathbf{P}}_{g,A}^{g,\theta}$. We exactly obtain the formula of
Theorem \ref{222ttt}.
\end{proof}

\begin{proof}[Proof of Theorem \ref{333ttt}]
Part (i) is a consequence of Theorem \ref{222ttt}, since both sides of the equality compute the cycle $\mathsf{logDR}_{g,A}$. Part (ii) follows from the
vanishing
\begin{equation}\label{vannnn}
\mathsf{P}^h_{g,\emptyset,0} = 0 \in
    \mathsf{CH}^h(\Picabs)\ \ \  \text{for} \ h>g
    \end{equation}
    combined with equality \eqref{eqn:Pgapprox} of Theorem \ref{thm:DRapprox_formula}. \edits{The vanishing \eqref{vannnn} follows from the universal DR relations of 
 \cite[Theorem 8]{BHPSS}.}
\end{proof}
  \edits{The proof of \cite[Theorem 8]{BHPSS} uses  a vanishing result for Pixton's double ramification formula with target varieties on moduli spaces of stable maps (which was shown in \cite{Bae} using a strategy that generalizes the earlier proof on the moduli space of stable curves given in \cite{cj}).}

\section{Genus 1 calculations}\label{sec:genus_1}
\subsection{Form of the answer}
\label{ffans}
We apply Theorem \ref{222ttt} here to explicitly compute $\LogDR_{g,A}$ in the case $g=1$. The result, 
presented in Theorem~\ref{thm:g1}, is of the form
\[
\LogDR_{1,A} = \DR_{1,A} + \text{(piecewise linear correction term)} \ \in \mathsf{logCH}^1(\oM_{1,n})\, ,
\]
where the term $\DR_{1,A}$ is the (pullback of the) standard double ramification cycle
in $\mathsf{CH}^1(\oM_{1,n})$.
The formula 
for $\DR_{1,A}$
is written in sPP language  using Theorem~\ref{thm:regularDR}. We have 
\[
\DR_{1,A} = -\frac{1}{2}\eta + \Phi(\mathrm{Cont})\, ,
\]
where $\eta$ is the class from \eqref{eqn:etadefinition} and $\mathrm{Cont}$ is the strict piecewise linear function given by
\begin{equation}\label{eq:DRg1}
\mathrm{Cont}_\Gamma = -\frac{1}{12}\left(\sum_{e\text{ a nonseparating edge}}\ell_e\right)-\frac{1}{2}\left(\sum_{e\text{ a separating edge}}w(e)^2\ell_e\right),
\end{equation}
where $w$ is any flow on $\Gamma$ satisfying $\mathrm{div}(w) = \mdeg_{k,A}$. \edits{Here $\mdeg_{k,A}$ is the multidegree of the line bundle $(\omega_C^{\mathrm{log}})^{\otimes k}(-\sum_i a_i x_i)$ as defined in Section \ref{sasc}. In the second sum, we have one term for each \emph{undirected} separating edge; although evaluating $w(e)$ requires choosing an orientation for $e$, the value of $w(e)^2$ does not depend on this choice because $w(\cev e) = - w(\vec e)$. While there can be different choices for the flow $w$ in general, its value on a separating edge is uniquely determined by $\mathrm{div}(w)$ (as the sum of the degrees on one side of the edge), so the term $w(e)^2$ in \eqref{eq:DRg1} is well-defined.}

\subsection{\texorpdfstring{$\theta$}{theta}-stability}
Let $\theta$ be a small nondegenerate stability condition of type $(1,n)$ and degree 0.
We can completely describe the $\theta$-stable multidegrees $D$. 
A graph $\Gamma$ of genus 1 has cycle number either $0$ or $1$. 
Nothing interesting happens when the cycle number is $0$ (in which case the graph is a tree).
\begin{lemma}\label{lem:0cycle}
Let $\Gamma$ be a tree. Let $\theta$ be a small nondegenerate stability condition of degree $0$. Then, the only $\theta$-stable multidegree $D$ on $\Gamma$ is $D=0$.
\end{lemma}
\begin{proof}
Let $e$ be any edge of $\Gamma$ which splits $\Gamma$ into subgraphs $\Gamma_1$ and $\Gamma_2$. Because there is only one edge connecting $\Gamma_1$ with $\Gamma_2$, there is only one integer satisfying the $\theta$-stability inequality for $D(\Gamma_1)$. Because $\theta$ is small, this integer must be $0$. Therefore, $D$ has total degree $0$ on each of the subgraphs $\Gamma_1$ and $\Gamma_2$. We easily conclude that $D$ has degree $0$ on every vertex.
\end{proof}

The situation is slightly more complicated for a graph $\Gamma$ with cycle number $1$. Such a graph has a unique cycle $C_\Gamma$ along with trees glued to the vertices of $C_\Gamma$.
\begin{lemma}\label{lem:1cycle}
Let $\Gamma$ be a connected graph with cycle number $1$ and cycle $C_\Gamma$. Let $\theta$ be a small nondegenerate stability condition of degree $0$. Let $D$ be a $\theta$-stable multidegree $D$ on $\Gamma$. Then, $D(v) = 0$ for all $v$ not on the cycle $C_\Gamma$. Moreover, as one goes around the cycle, the nonzero values of $D(v)$ alternate between $1$ and $-1$.
\end{lemma}
\begin{proof}
The first claim follows by the same argument as for Lemma \ref{lem:0cycle}. For the 
second claim, it suffices to show that, for any subset of consecutive vertices around the cycle $C_\Gamma$, the sum of their degrees is in the set $\{-1,0,1\}$. 

Consider the subgraph $\Gamma_1$ consisting of consecutive vertices around $C_\Gamma$ together with the trees glued at those vertices. Because there are only two edges connecting $\Gamma_1$ with its complement, there are only two integers satisfying the $\theta$-stability inequality for $D(\Gamma_1)$. Because $\theta$ is small, one of these integers must be $0$, and then the other is $\pm 1$. Therefore, $D(\Gamma_1)\in \{-1,0,1\}$, and since $D$ has degree $0$ outside the cycle, we obtain the desired property.
\end{proof}

The Lemmas \ref{lem:0cycle} and \ref{lem:1cycle}  apply not only to a genus $1$ stable graph $\Gamma$, but to any quasi-stable modification $\widehat{\Gamma}$ \edits{(we extend the stability condition to the quasi-stable model as explained after Definition \ref{def:stabilitycondition}). }

\subsection{The formula}
The computation of $\mathsf{logDR}$ is particularly simple in genus $1$  because the function $\DRpp$ 
of \eqref{eqn:logDRmixed} in degree $1$ is just the first term in \eqref{eq:DRg1}:  the strict piecewise linear function
\[
-\frac{1}{12}\left(\sum_{e\text{ a nonseparating edge}}\ell_e\right)\, .
\]
The interesting part
of formula \eqref{eqn:logDRmixed} in genus 1 is $\DRpl$, which
we will compute explicitly.

We begin with a general result (which holds in any genus) rewriting 
$\DRpl$
 in a more convenient form. 

\begin{lemma}\label{lem:drpl2}
Let $(\widehat{\Gamma}, D, I)$ be the data giving a cone $\sigma_{\widehat{\Gamma},I}\in\widetilde{\Sigma}^{\theta}_{\Gamma}$.
Let $J$ be any flow on $\widehat{\Gamma}$ satisfying $\mathrm{div}(J) = D$, and let $I_0 = I+J$ (so $\mathrm{div}(I_0) = \mdeg_{k,A}$). Then on $\sigma_{\widehat{\Gamma},I}$, we have
\[
\DRpl = \sum_{e\in E(\widehat{\Gamma})}(I_0(\vec e)^2-J(\vec e)^2)\cdot\widehat{\ell}_e\big|_{\widehat{\ell} = \widehat{\ell}(\ell)}\, ,
\]
where \edits{$\widehat{\ell}(\ell)$ is the unique collection of edge lengths of the subdivision $\widehat \Gamma$ of $\Gamma$ (depending on the lengths $\ell$ of $\Gamma$) such that there exists a strict piecewise linear function on $\widehat \Gamma$ with slopes given by the flow $I$ (see \cref{Dcomment-I-should-try-to-prove-these-claims}), and where} the sum runs over the unoriented edges of $\widehat{\Gamma}$: 
any orientation of $e$ can be chosen (the summand
is invariant).
\end{lemma}
\begin{proof}
We use the same notation as in the definition of $\DRpl$ in \eqref{xsse3} and \eqref{eq:L_formula}, but we omit the final change of variables $|_{\widehat{\ell} = \widehat{\ell}(\ell)}$ for brevity of notation. We have
\begin{align*}
\DRpl &= \sum_{v \in V(\widehat{\Gamma})} (D + \mdeg_{k,A})(v)\alpha(v)\\
&= \sum_{v, \vec e \to v}(I_0(\vec e)+J(\vec e))\alpha(v)\\
&= \sum_{e = (v,v') \in E(\widehat{\Gamma})}\Big((I_0(v\to v')+J(v\to v'))\alpha(v')
 + (I_0(v'\to v)+J(v'\to v))\alpha(v)\Big)
\\ 
&= \sum_{e = (v,v') \in E(\widehat{\Gamma})}(I_0(v\to v')+J(v\to v'))(\alpha(v')-\alpha(v))\\
&= \sum_{e \in E(\widehat{\Gamma})}(I_0(\vec e)+J(\vec e))I(\vec e)\widehat{\ell}_e\\
&= \sum_{e \in E(\widehat{\Gamma})}(I_0(\vec e)+J(\vec e))(I_0(\vec e)-J(\vec e))\widehat{\ell}_e
\end{align*}
as desired. In the last two lines, any orientation of $e$ can be chosen (the summand
is invariant).
\end{proof}

Let $\Gamma$ be a stable graph of genus 1, and let $(\widehat{\Gamma}, D, I)$ be the data giving a cone $$\sigma_{\widehat{\Gamma},I}\in\widetilde{\Sigma}^{\theta}_{\Gamma}\, .$$
We will compute $\DRpl$ on $\sigma_{\widehat{\Gamma},I}$. If $\Gamma$ is a tree, then
$D=0$ and $\widehat{\Gamma}=\Gamma$
by
Lemma~\ref{lem:0cycle}. Nothing interesting is happening: there is a unique flow $I$ with $\mathrm{div}(I) = \mdeg_{k,A}$, and the resulting function
$\DRpl$ simply contributes part of the regular double ramification cycle formula. 

Assume now that $\Gamma$ has cycle number $1$ with cycle $C_\Gamma$. We fix an orientation on $C_\Gamma$, and let $C$ be the flow on $\Gamma$ (or on $\widehat{\Gamma}$) with value $1$ around that cycle and $0$ elsewhere.
In order to apply Lemma~\ref{lem:drpl2},  we must pick a flow $J$ with $\mathrm{div}(J) = D$. By Lemma~\ref{lem:1cycle}, $D$ must have a very specific form, and we can find such
{\footnote{If $D\ne 0$, it is easy to see that such a $J$ is actually unique.}}
a $J$ supported on $C_\Gamma$ that only takes on values $0$ and $1$ (when read in the chosen orientation around the cycle).
We then have
\begin{align*}
\DRpl &= \sum_{e\in E(\widehat{\Gamma})}(I_0(e)^2-J(e)^2)\cdot\widehat{\ell}_e\big|_{\widehat{\ell} = \widehat{\ell}(\ell)}\\
&= \sum_{e\in E(\widehat{\Gamma})}(I_0(e)^2-J(e)C(e))\cdot\widehat{\ell}_e\big|_{\widehat{\ell} = \widehat{\ell}(\ell)}\\
&= \sum_{e\in E(\widehat{\Gamma})}(I_0(e)^2-I_0(e)C(e))\cdot\widehat{\ell}_e\big|_{\widehat{\ell} = \widehat{\ell}(\ell)}\\
&= \sum_{e\in E(\Gamma)}(I_0(e)^2-I_0(e)C(e))\cdot\ell_e\, ,
\end{align*}
where the first equality is Lemma~\ref{lem:drpl2}, the second equality follows from $J$ only taking on values $0$ and $1$ (when oriented consistently with $C$ around the cycle), the third equality follows from the relation between $I = I_0-J$ and the variable substitution $\widehat{\ell} = \widehat{\ell}(\ell)$, and the fourth equality follows because $I_0$ and $C$ are constant along the edge subdivided by $\widehat{\Gamma}$.

This final expression only depends on $I_0$. Since $I_0 = I+J$ and $$C(e)^2 \ge C(e)J(e) \ge 0$$ for all $e$, we have the inequalities
\begin{equation*}
\sum_{e\in E(\widehat{\Gamma})}I_0(e)C(e)\widehat{\ell}_e \ge \sum_{e\in E(\widehat{\Gamma})}I(e)C(e)\widehat{\ell}_e \ge \sum_{e\in E(\widehat{\Gamma})}(I_0(e)-C(e))C(e)\widehat{\ell}_e\, .
\end{equation*}
The middle term here vanishes after the variable substitution $\widehat{\ell} = \widehat{\ell}(\ell)$, while the first and third terms can be simplified because $I_0$ and $C$ are constant along the edge subdivided by $\widehat{\Gamma}$. If we let $F = I_0$ denote the flow on $\Gamma$, we obtain the inequalities
\begin{equation}\label{eq:g1coarse}
\sum_{e\in E(\Gamma)}F(e)C(e)\ell_e \ge 0 \ge \sum_{e\in E(\Gamma)}(F(e)-C(e))C(e)\ell_e\, .
\end{equation}
But for general $\ell_e$, exactly one flow $F$ with $\mathrm{div}(F) = \mdeg_{k,A}$ satisfies these inequalities (since the set of such flows is an arithmetic progression with difference $C$). Therefore the inequalities \eqref{eq:g1coarse} describe a (coarser) subdivision where $\DRpl$ can be defined. We summarize what we have computed in the following result.

\begin{proposition}\label{prop:g1drpl}
Let $\Gamma$ be a graph with cycle number $1$. Let $C$ be a flow on $\Gamma$ with value $1$ around the unique cycle and $0$ elsewhere. Then
\[
\DRpl = \sum_{e\in E(\Gamma)}(F(e)^2-F(e)C(e))\cdot\ell_e,
\]
where $F$ is any flow on $\Gamma$ satisfying $\mathrm{div}(F) = \mdeg_{k,A}$ and the inequalities \eqref{eq:g1coarse}.
\end{proposition}

By putting the pieces of the formula
for $\mathsf{logDR}$ together, we obtain a simple description of the piecewise linear function giving the difference between $\mathsf{logDR}$ and standard $\mathsf{DR}$ in genus $1$.
\begin{theorem}\label{thm:g1}
In genus 1, we have $$\LogDR_{1,A} = \DR_{1,A} - \frac{1}{2} \Phi(\DRpl')\, ,$$ where $\DRpl'$ is the piecewise linear function defined as follows: $\DRpl'$ is zero on cones coming from a stable tree $\Gamma$, and otherwise $\DRpl'$ is given by
\[
\DRpl' = \sum_{e\in E(C_\Gamma)}(F(e)^2-F(e)C(e))\cdot\ell_e
\]
(for $C$ and $F$ as in Proposition~\ref{prop:g1drpl}).
\end{theorem}
\begin{proof}
In genus $1$ and degree 1, the $\mathsf{logDR}$ formula \eqref{eqn:logDRmixed} yields
\[
\LogDR_{1,A} = -\frac{1}{2}\eta + \Phi([\DRpp]_{\text{deg $1$}}) - \frac{1}{2}\Phi(\DRpl).
\]
We compare the above with the standard 
$\mathsf{DR}$ formula discussed in Section \ref{ffans}. The $\eta$ terms are identical, and we have already observed that the degree $1$ part of $\DRpp$ agrees with the nonseparating edge term in \eqref{eq:DRg1}. Meanwhile, the part of $-\frac{1}{2}\Phi(\DRpl)$ coming from separating edges in Proposition~\ref{prop:g1drpl} agrees with the separating edge term in \eqref{eq:DRg1}. An edge in $\Gamma$ is nonseparating if and only if it is in the cycle $C_\Gamma$, so removing those terms from $\DRpl$ leaves precisely $\DRpl'$.
\end{proof}

\begin{remark} \label{55kk5}
After a little combinatorial manipulation, we obtain  an alternative description of the piecewise linear function $\DRpl'$. Suppose $F$ is a flow on $\Gamma$ satisfying $\mathrm{div}(F) = \mdeg_{k,A}$. Define a piecewise linear function on $\sigma_\Gamma$ by
\[
m_F = \min\left(\sum_{\substack{e\in E(\Gamma)\\F(e)C(e) > 0}}F(e)C(e)\ell_e,\sum_{\substack{e\in E(\Gamma)\\F(e)C(e) < 0}}-F(e)C(e)\ell_e\right).
\]
The function $m_F$ is zero if $F$ is not acyclic \edits{(where not being acyclic means that we can find a cycle $\gamma$ of directed edges in $\Gamma$ such that $F(\vec e)\geq 0$ for all $\vec e$ in $\gamma$ and with at least one of the inequalities being strict)}.  It is then easily checked that on $\sigma_\Gamma$ we have
\[
\DRpl' = 2\sum_Fm_F\, ,
\]
where the sum runs over all acyclic $F$.
\end{remark}

\begin{remark}\label{ff5tq}
The logarithmic genus $1$ computation can
be used to compute an interesting class in $\mathsf{CH}^2(\Mbar_{1,n})$ as follows. Take two  genus $1$ cycles with the same $n$,
$$\LogDR_{1,A}\, ,\,  \LogDR_{1,B}\in \mathsf{logCH^1(\oM_{1,n}})\, .$$
 Individually, the cycles push forward to the 
 standard cycles $$\DR_{1,A}, \DR_{1,B}\in \mathsf{CH}^1(\oM_{1,n})\,$$ but the pushforward of their product, $$\pi_*(\LogDR_{1,A}\cdot\LogDR_{1,B}) \in \mathsf{CH}^2(\Mbar_{1,n})\,,$$
 is {\em not} equal to the product of the standard 
 $\mathsf{DR}$ cycles. These {\em double-double ramification cycles} are discussed further in 
 Section \ref{sagesage}.
 Here, we simply sketch how to compute them in genus $1$ using the above formulas for $\mathsf{logDR}$.

By Theorem~\ref{thm:g1} and a straightforward argument
for the vanishing of the cross terms, we have
\[
\pi_*(\LogDR_{1,A}\cdot\LogDR_{1,B}) = \DR_{1,A}\cdot\DR_{1,B} + \frac{1}{4}\pi_*(\Phi(\DRpl'_A\cdot \DRpl'_B))\, .
\]
By Remark \ref{55kk5}, this second term is defined on a cone $\sigma_\Gamma$ by the polynomial
\[
\pi_*\left(\sum_{F_A,F_B}m_{F_A}m_{F_B}\right),
\]
where $F_A,F_B$ run over acyclic flows balancing $A$ and $B$ respectively and $\pi_*$ is the pushforward from piecewise polynomials on a subdivision of a cone to polynomials on that cone. After working out the algebra of the pushforward and translating into standard tautological class notation, we obtain 
\begin{multline*}
\pi_*(\LogDR_{1,A}\cdot\LogDR_{1,B}) = \DR_{1,A}\cdot\DR_{1,B} \\
- \sum_{\Gamma\text{ double edge graph}}\frac{1}{2}j_{\Gamma*}\left(\frac{A_\Gamma^2B_\Gamma^2-A_\Gamma^2-B_\Gamma^2+\mathsf{gcd}(A_\Gamma,B_\Gamma)^2}{12}\right),
\end{multline*}
where the sum runs over the isomorphism classes of stable graphs $\Gamma$ with two vertices and a double edge connecting them and $A_\Gamma,B_\Gamma$ are the absolute values of the degrees of $\mdeg_{k,A}$, $\mdeg_{k,B}$ respectively on one of the vertices (it does not matter which one since the total degree is $0$).
\end{remark}


\section{Computational tools and higher double ramification cycles}

\label{sagesage1}

\subsection{A Sage implementation of $\mathsf{logDR}$} \label{sagesage}



A Sage package \href{https://gitlab.com/modulispaces/admcycles/-/tree/master/admcycles/logtaut}{\texttt{logtaut}} for computations with logarithmic double ramification cycles is being developed as part of the \href{https://gitlab.com/modulispaces/admcycles}{\texttt{admcycles}} package \cite{admcycles}. The package is able to compute/manipulate
the following mathematical structures discussed in our paper:
\begin{itemize}
    \item the vector space of stability conditions $\theta$ of type $(g,n)$ and degree $d$ (with a description of a natural basis),
    \item cone stacks and their (strict) piecewise polynomial functions (together with natural operations such as pushforwards under subdivision maps),
    \item the particular cone stack $\Sigma_{\oM_{g,n}}$, the subdivision induced by a vector $A \in \mathbb{Z}^n$ together with a nondegenerate stability condition $\theta$ of degree $|A|$, and the natural map \[\mathsf{sPP}(\Sigma_{\oM_{g,n}}) \to \mathsf{R}^*(\oM_{g,n})\, \]
    discussed in Section \ref{sec: pp},
    \item the piecewise polynomial functions $\DRpl$ and $\DRpp$ of equation \eqref{eqn:logDRmixed} determining the logarithmic double ramification cycle via Theorem \ref{222ttt}.
\end{itemize}

As an example of a calculation in \texttt{logtaut}, we can write an explicit relation in $\mathsf{logR}^2(\oM_{1,3})$ obtained by Theorem \ref{333ttt} (ii). 
Concretely, we can combine the vanishing from Theorem \ref{333ttt}
with the classical vanishing of the formula $\mathsf P_{g,A}^h$ for the double ramification cycle in degree $h$ higher than $g$.\footnote{The formula $\mathsf P_{g,A}^h$ was first presented in \cite{JPPZ} and the claimed vanishing was proven in \cite{cj}.} Taking the difference, we obtain
\[
\Delta = {\mathbf{P}}_{1,(2,-1,-1)}^{2,\theta} - \mathsf P_{1,(2,-1,-1)}^{2} = 0 \in \mathsf{logR}^2(\oM_{1,3})\,,
\]
for a small and nondegenerate stability condition $\theta.$\footnote{The choice of the stability condition $\theta$ does not affect the relation $\Delta$ in this case.} After expressing $\mathsf P_{1,(2,-1,-1)}^{2}$ in terms of piecewise polynomials via Theorem \ref{thm:regularDR}, the advantage of the difference $\Delta$ is that the piecewise polynomial parts of ${\mathbf{P}}_{1,(2,-1,-1)}^{2,\theta}$ and $\mathsf P_{1,(2,-1,-1)}^{2}$ cancel on all but two maximal cones in $\Sigma_{\oM_{1,3}}$. With \texttt{logtaut}, we then compute that $\Delta$ is the codimension $2$ part of the mixed-degree class
\begin{equation}
    \left(1 + 2\psi_1 + \frac{1}{2} (\psi_2 + \psi_3) \right) \cdot \Phi(\mathfrak{P})
\end{equation}
for the piecewise polynomial function $\mathfrak{P}$ on $\Sigma_{\oM_{1,3}}$ illustrated in Figure \ref{fig:logRrelation}. 

\begin{figure}[htb]
    \centering
\[
\begin{tikzpicture}[baseline=0pt, vertex/.style={circle,draw,font=\Large,scale=0.5, thick}]
\node[vertex] (A) at (-1.5,0) {0};
\node[vertex] (B) at (0,0) {0};
\node[vertex] (C) at (1.5,0) {0};
\draw[thick] (A) to[bend left] node[midway, above]{$\ell_1$} (B);
\draw[thick] (A) to[bend right] node[midway, below]{$\ell_2$} (B);
\draw[thick] (B) to node[midway, above]{$\ell_3$} (C);
\draw[thick] (A) to (-2, 0) node[left] {$1$};
\draw[thick] (C) to (1.7, 0.5) node[right] {$2$};
\draw[thick] (C) to (1.7, -0.5) node[right] {$3$};

\begin{scope}[shift={(0.75,0)}]
\node[vertex] (A) at (-1.5,-3) {0};
\node[vertex] (B) at (0,-2) {0};
\node[vertex] (C) at (0,-4) {0};
\draw[thick] (A) to node[midway, above]{$\ell_1$} (B);
\draw[thick] (A) to node[midway, below]{$\ell_2$} (C);
\draw[thick] (B) to node[midway, right]{$\ell_4$} (C);
\draw[thick] (A) to (-2, -3) node[left] {$1$};
\draw[thick] (B) to (0.5, -2) node[right] {$2$};
\draw[thick] (C) to (0.5, -4) node[right] {$3$};
\end{scope}

\begin{scope}[shift={(8,-1)}]
\draw[thick] (-2,0) node[left] {$\ell_1=1$} -- (2,0) node[right] {$\ell_2=1$} -- (0,3.464) node[above] {$\ell_3=1$} -- (-2,0) -- (0,-3.464) node[below] {$\ell_4=1$} -- (2,0) ;
\draw[thick] (0,3.464) -- (0,-3.464);
\draw (-2.3,2.7) node[text width=2cm] {$\frac{1}{2}\ell_1 \ell_2 + \frac{1}{2}\ell_2^2$\\$+ 2 \ell_2 \ell_3 - \ell_2$} ;
\draw (2.3,2.7) node[text width=2cm] {$\frac{1}{2}\ell_1 \ell_2 + \frac{1}{2}\ell_1^2$\\$+ 2 \ell_1 \ell_3 - \ell_1$} ;
\draw (-2.3,-2.7) node[text width=2cm] {$\frac{1}{2}\ell_1 \ell_2 + \frac{1}{2}\ell_2^2$\\$- \ell_2$} ;
\draw (2.3,-2.7) node[text width=2cm] {$\frac{1}{2}\ell_1 \ell_2 + \frac{1}{2}\ell_1^2$\\$- \ell_1$} ;
\end{scope}
\end{tikzpicture}
\]
    \caption{The piecewise polynomial $\mathfrak{P}$, which is nonvanishing on precisely two (neighboring) maximal cones of $\Sigma_{\oM_{1,3}}$. We illustrate the stable graphs associated to these cones and the slice through them corresponding to all edge lengths summing to $1$. The cones are subdivided at the ray $\ell_1=\ell_2=1, \ell_3=\ell_4=0$, and we give the piecewise polynomial $\mathfrak{P}$ on the resulting four chambers.}
    \label{fig:logRrelation}
\end{figure}
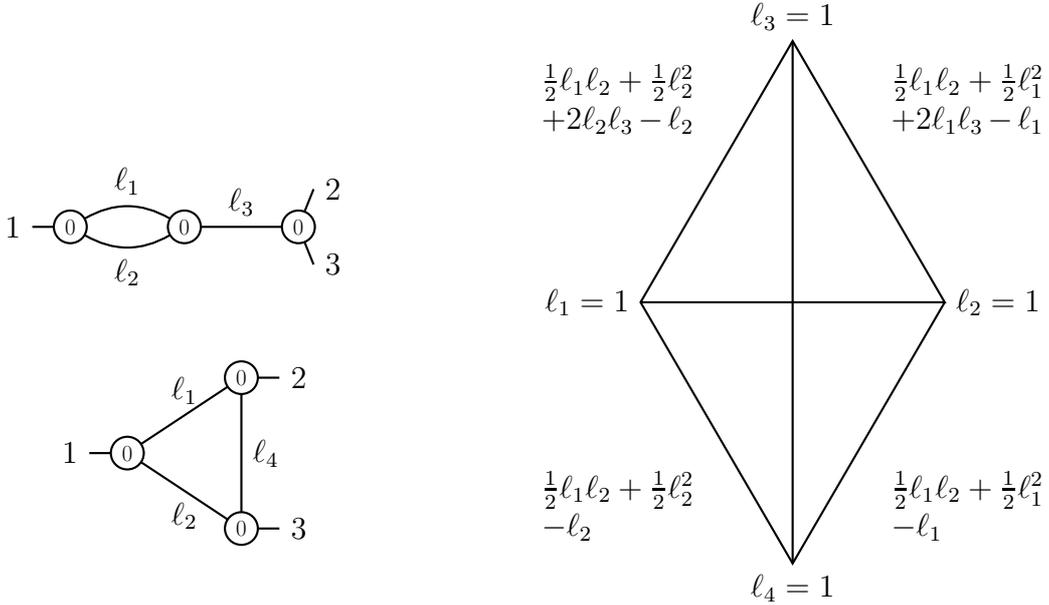
Using the methods of Section \ref{sec: pp}, we can translate the relation 
$\Delta=0$
into the language of normally decorated strata classes: consider the graphs
\[
\Gamma = 
\begin{tikzpicture}[baseline=-0.1 cm, vertex/.style={circle,draw,font=\Large,scale=0.5, thick}]
\node[vertex] (A) at (-1.5,0) {0};
\node[vertex] (B) at (0,0) {0};
\draw[thick] (A) to[bend left] (B);
\draw[thick] (A) to[bend right] (B);
\draw[thick] (A) to (-2, 0) node[left] {$1$};
\draw[thick] (B) to (0.4, 0.4) node[right] {$2$};
\draw[thick] (B) to (0.4, -0.4) node[right] {$3$};
\end{tikzpicture}
\quad,  \quad 
\delta_{2,3} = 
\begin{tikzpicture}[baseline=-0.1 cm, vertex/.style={circle,draw,font=\Large,scale=0.5, thick}]
\node[vertex] (A) at (-1.5,0) {1};
\node[vertex] (B) at (0,0) {0};
\draw[thick] (A) to (B);
\draw[thick] (A) to (-2, 0) node[left] {$1$};
\draw[thick] (B) to (0.4, 0.4) node[right] {$2$};
\draw[thick] (B) to (0.4, -0.4) node[right] {$3$};
\end{tikzpicture} \ .
\]
Let $\pi: \widehat{\mathcal{M}} \to \oM_{1,3}$ be the blow-up of the stratum associated to $\Gamma$, and denote by $E_\Gamma$ the exceptional divisor of $\pi$. After expressing $\Phi(\mathfrak P)$ in terms of normally decorated strata classes, we obtain
\begin{equation}
    \Delta \, = \, \frac{1}{2} \cdot \pi^* \left[\Gamma \right] - (2 \psi_1 + \frac{1}{2}(\psi_2 + \psi_3) - 2 \delta_{23}) \cdot E_\Gamma + \frac{1}{2} E_\Gamma^2\, \in\,  \mathsf{CH}^2(\widehat{\cM})\,. 
\end{equation}
It is interesting to note that the coefficient $(2 \psi_1 + \frac{1}{2}(\psi_2 + \psi_3) - 2 \delta_{23})$ appearing here is precisely the compact type part of the formula $\mathsf P_{1,(2,-1,-1)}^1$, giving the pullback of the theta divisor under the Abel-Jacobi map associated to $(2,-1,-1)$.

\edits{Due to the complexity of the objects involved (such as the subdivisions of fans and their rings of piecewise polynomials), the calculations in \texttt{logtaut} are mostly feasible for low genus and numbers of markings (around $2g+n \leq 7$). With future optimization, further cases could be covered.}

\subsection{Double-double ramification cycles} \label{sagesage3}
One of the initial motivations for the study of $\mathsf{logDR}$ was that of understanding the {\em double-double ramification cycle} which arises in the virtual localization formula
for the Gromov-Witten theory of log toric surfaces \cite{Graber-inpreparation} as a vertex term (replacing the
quadratic Hodge integrals in the localization formula for the Gromov-Witten theory
of toric surfaces \cite{GraberPandharipandeLoc}).
\edits{The geometry of these higher double ramification cycles plays a central role 
in the study of stable maps to log toric targets \cite{Dhruv,RanganathanKumaran}.}
Using the formula of Theorem \ref{222ttt} and the Sage package \texttt{logtaut}, we can now compute these cycles in moderate genus (limited by computing capacity).

To motivate the higher double ramification cycles,
we begin by recalling that the natural formula
\begin{equation*}
    \mathsf{DR}_{g, A} \cdot \mathsf{DR}_{g, B} \stackrel{?}{=} \mathsf{DR}_{g, A+B}\cdot \mathsf{DR}_{g, B}
\end{equation*}
holds on the locus of curves of compact type, but fails on $\Mbar_{g,n}$. However, 
by a result of \cite{Holmes2017Multiplicativit},  the analogue for $\mathsf{logDR}$ is always true:
\begin{equation}\label{eq:logDR_mult}
    \mathsf{logDR}_{g, A} \cdot \mathsf{logDR}_{g, B} = \mathsf{logDR}_{g, A+B}\cdot \mathsf{logDR}_{g, B} \in \mathsf{logCH}^{2g}(\Mbar_{g,n})\, ,
\end{equation}
which 
is a special case of a $\on{GL}_2(\bb Z)$-invariance property for
$\mathsf{logDR}$. 
The {\em double-double ramification cycle} is defined by 
\begin{equation*}
    \mathsf{DDR}_{g, A, B} = \pi_*\left(\mathsf{logDR}_{g, A} \cdot \mathsf{logDR}_{g, B}\right) \in \mathsf{CH}^{2g}(\Mbar_{g,n}), 
\end{equation*}
and may be viewed as a ``corrected" version of the product $\mathsf{DR}_{g, A} \cdot \mathsf{DR}_{g, B}$. Higher
double ramification cycles are defined by pushforwards of products of more
logarithmic double ramification cycles.

The double-double ramification cycle is trivial in genus $0$:
$$\mathsf{DDR}_{0,A,B} = 1 \in \mathsf{logCH}^{0}(\Mbar_{0,n})\, .$$
A full calculation in genus 1 of $\mathsf{DDR}_{1,A,B}$ was presented in
Remark \ref{ff5tq} using  the formula for $\mathsf{logDR}$ of Theorem \ref{222ttt}.
That $\mathsf{DDR}_{g,A,B}$ and the higher analogues are tautological classes on
the moduli spaces of curves for all $g$
follows also from results of \cite{HolmesSchwarz,Molcho2021-Case-Study}, but the formula of Theorem \ref{222ttt} 
provides the only known effective method of calculation.
The package \texttt{logtaut} produces
explicit formulas for $\mathsf{DDR}_{g, A, B}$ in the tautological ring.

We illustrate below the double-double cycle in
 case $g=2$, $n=3$, $A = [3,-3,0]$, and $B = [0,3,-3]$. Since the class $$\mathsf{DDR}_{2, A, B} \in \mathsf{CH}^{4}(\Mbar_{2,3})$$ is invariant under permuting the markings (as can be seen by applying \eqref{eq:logDR_mult}), we can express the answer via \texttt{logtaut}
as a sum over graphs\footnote{Following the conventions of \texttt{admcycles}, in the formula for $\mathsf{DDR}_{2, A, B}$,
each stable graph $\Gamma$ represents the sum of $3!=6$ pushforwards of the fundamental class under gluing maps (associated to the $6$ ways of adding labels to the $3$ legs of the graph).} without specifying the ordering of the legs: 

\vspace{-25pt}
\tikz{\coordinate (A) at (0,0); \coordinate (B) at (1,0); \coordinate (C) at (0.6,0.5);\coordinate (D) at (2,0);\coordinate (E) at (3,0);\coordinate (F) at (4,0)}
\tikzset{baseline=0, label distance=-3mm}
\def\NC{\draw (0,0.25) circle(0.25);}
\def\NL{\draw plot [smooth,tension=1.5] coordinates {(0,0) (-0.2,0.5) (-0.5,0.2) (0,0)};}
\def\NR{\draw plot [smooth,tension=1.5] coordinates {(0,0) (0.2,0.5) (0.5,0.2) (0,0)};}
\def\NN{\NL\NR}
\def\NNN{\NN \begin{scope}[rotate=180] \NR \end{scope}}
\def\NNNN{\NN \begin{scope}[rotate=180] \NN \end{scope}}
\def\NRS{\begin{scope}[shift={(B)}] \NR \end{scope}}
\def\NRD{\begin{scope}[rotate around={-90:(B)}] \NRS \end{scope}}
\def\DE{\draw plot [smooth,tension=1] coordinates {(0,0) (0.5,0.15) (1,0)}; \draw plot [smooth,tension=1.5] coordinates {(0,0) (0.5,-0.15) (1,0)};}
\def\DES{\begin{scope}[shift={(B)}] \DE \end{scope}}
\def\TE{\DE \draw (A) -- (B);}
\def\QE{\DE \draw plot [smooth,tension=1] coordinates {(A) (0.5,0.05) (B)}; \draw plot [smooth,tension=1.5] coordinates {(A) (0.5,-0.05) (B)};}
\def\T{\draw (0.2,0) -- (C) -- (B) -- (0.2,0);}
\def\TT{\draw (C) -- (B) -- (0.2,0); \draw plot [smooth,tension=1] coordinates {(0.2,0) (0.3,0.3) (C)}; \draw plot [smooth,tension=1] coordinates {(0.2,0) (0.5,0.2) (C)};}
\newcommand{\nn}[3]{\draw (#1)++(#2:3mm) node[fill=white,fill opacity=.85,inner sep=0mm,text=black,text opacity=1] {$\substack{\psi^#3}$};}
\renewcommand{\gg}[2]{\fill (#2) circle(1.3mm) node {\color{white}$\substack #1$};}
\newcommand{\oneleg}[1]{\draw (#1) -- + (0, -0.3);} 
\newcommand{\twolegs}[1]{\draw (#1) -- + (-0.2, -0.3);\draw (#1) -- + (0.2, -0.3);}
\newcommand{\threelegs}[1]{\draw (#1) -- + (0, -0.3);\draw (#1) -- + (-0.2, -0.3);\draw (#1) -- + (0.2, -0.3);}

\newcommand{\selfloop}[1]{\draw (#1) + (0,0.25) circle(0.25);}
\newcommand{\SE}[2]{\draw (#1) -- (#2);}
\newcommand{\BE}{\draw plot [smooth,tension=1] coordinates {(0,0) (1,0.45) (2,0)};}
\newcommand{\LBE}{\draw plot [smooth,tension=1] coordinates {(0,0) (1.5,0.6) (3,0)};}
\def\DEtwo{\draw plot [smooth,tension=1] coordinates {(1,0) (1.5,0.15) (2,0)}; \draw plot [smooth,tension=1.5] coordinates {(1,0) (1.5,-0.15) (2,0)};}

\begin{align*}
& \mathsf{DDR}_{2, A, B} = \frac{93}{640} \tikz{\threelegs{A} \BE \SE{A}{B}; \DEtwo;\gg{0}{A} \gg{0}{B} \gg{0}{D}}
-\frac{87}{64} \tikz{\twolegs{A} \oneleg{D} \BE \SE{A}{B}; \DEtwo;\gg{0}{A} \gg{0}{B} \gg{0}{D}}
+ \frac{183}{160} \tikz{\twolegs{A} \oneleg{B}\DE \SE{B}{D};\draw (2,0.25) circle(0.25);\gg{0}{A} \gg{0}{B} \gg{0}{D}}
\\
&-\frac{49}{160} \tikz{\oneleg{A} \oneleg{B} \oneleg{D}\BE \SE{A}{B}; \DEtwo;\gg{0}{A} \gg{0}{B} \gg{0}{D}}
+ \frac{27}{320} \tikz{\oneleg{A} \twolegs{D} \DE; \DEtwo;\gg{0}{A} \gg{0}{B} \gg{0}{D}}
+ \frac{213}{640} \tikz{\oneleg{A} \oneleg{B} \oneleg{D}\DE; \DEtwo;\gg{0}{A} \gg{0}{B} \gg{0}{D}}
\\
&+ \frac{711}{640} \tikz{\oneleg{A} \oneleg{B} \oneleg{D}\SE{A}{B}\SE{B}{D};\draw (2,0.25) circle(0.25);\draw (0,0.25) circle(0.25);\gg{0}{A} \gg{0}{B} \gg{0}{D}}
-\frac{93}{640} \tikz{\threelegs{B}\SE{A}{B}\SE{B}{D};\draw (2,0.25) circle(0.25);\draw (0,0.25) circle(0.25);\gg{0}{A} \gg{0}{B} \gg{0}{D}}
+ \frac{321}{1280} \tikz{ \twolegs{B}\oneleg{D}\SE{A}{B}\SE{B}{D};\draw (2,0.25) circle(0.25);\draw (0,0.25) circle(0.25);\gg{0}{A} \gg{0}{B} \gg{0}{D}}
\\
&+ \frac{9}{256} \tikz{\twolegs{D}\oneleg{B}\SE{A}{B}\SE{B}{D};\draw (2,0.25) circle(0.25);\draw (0,0.25) circle(0.25);\gg{0}{A} \gg{0}{B} \gg{0}{D}}
-\frac{549}{20} \tikz{\twolegs{A}\oneleg{D}\DE \SE{B}{D};\SE{D}{E};\gg{0}{A} \gg{0}{B} \gg{0}{D} \gg{1}{E}}
+ \frac{243}{20} \tikz{\oneleg{A} \oneleg{B} \oneleg{D}\SE{A}{B} \SE{B}{D};\SE{D}{E};\LBE \gg{0}{A} \gg{0}{B} \gg{0}{D} \gg{1}{E}}
\\
&+ \frac{7569}{160} \tikz{\oneleg{A} \oneleg{B} \oneleg{D}\selfloop{A}\SE{A}{B} \SE{B}{D};\SE{D}{E} \gg{0}{A} \gg{0}{B} \gg{0}{D} \gg{1}{E}}
+ \frac{639}{32} \tikz{\twolegs{B} \oneleg{D}\selfloop{A}\SE{A}{B} \SE{B}{D};\SE{D}{E} \gg{0}{A} \gg{0}{B} \gg{0}{D} \gg{1}{E}}
+ \frac{1251}{160} \tikz{ \oneleg{B} \twolegs{D}\selfloop{A}\SE{A}{B} \SE{B}{D};\SE{D}{E} \gg{0}{A} \gg{0}{B} \gg{0}{D} \gg{1}{E}}
\\
&-\frac{693}{160} \tikz{\oneleg{E} \oneleg{B} \oneleg{D}\selfloop{A}\SE{A}{B} \SE{B}{D};\SE{D}{E} \gg{0}{A} \gg{0}{B} \gg{0}{D} \gg{1}{E}}
+ \frac{6561}{20} \tikz{\oneleg{E} \oneleg{B} \oneleg{D}\SE{A}{B} \SE{B}{D};\SE{D}{E} \SE{E}{F} \gg{1}{A} \gg{0}{B} \gg{0}{D} \gg{0}{E} \gg{1}{F}}. 
\end{align*}

\vspace{5pt}
Two constructions of tautological relations in $\mathsf{logR}^*(\Mbar_{g,n})$
were given in 
Theorem \ref{333ttt}.
The 
$\on{GL}_2(\bb Z)$-invariance property for
the double-double ramification cycle 
yields a third construction:  apply Theorem \ref{222ttt} to
all four terms of \eqref{eq:logDR_mult}.

\bibliographystyle{abbrv} 
\bibliography{prebib.bib}

\vspace{8pt}

\noindent Mathematisch Instituut, Universiteit Leiden\\
\noindent holmesdst@math.leidenuniv.nl

\vspace{8pt}

\noindent Departement Mathematik, ETH Z\"urich\\
\noindent samouil.molcho@math.ethz.ch

\vspace{8pt}

\noindent Departement Mathematik, ETH Z\"urich\\
\noindent rahul@math.ethz.ch

\vspace{8pt}

\noindent Department of Mathematics, University of Michigan\\
\noindent pixton@umich.edu

\vspace{8pt}

\noindent \edits{Departement Mathematik, ETH Z\"urich}\\
\noindent johannes.schmitt@math.\edits{ethz}.ch

\end{document}